\documentclass{siamart250211}

\usepackage[a4paper, margin=2cm]{geometry}
\usepackage{amssymb}
\usepackage{mathtools}
\usepackage{tikz}
\usepackage{subcaption}
\usepackage{placeins}

\def\eps{\varepsilon}
\def\Id{{\rm \bf Id}}
\def\id{{\rm \bf id}}
\newcommand{\D}[2]{\mathbf{D}_{#1}^{\boldsymbol{\Delta},#2}}
\newcommand{\Dhat}[2]{\widehat{\mathbf{D}}_{#1}^{\boldsymbol{\Delta},#2}}
\newcommand{\sD}{\partial^C}
\definecolor{color1}{HTML}{0072BD}
\definecolor{color2}{HTML}{D95319}
\definecolor{color3}{HTML}{77AC30}
\definecolor{color4}{HTML}{7E2F8E}

\pgfmathsetmacro\xx{1/sqrt(2)}
\pgfmathsetmacro\xy{1/sqrt(6)}
\pgfmathsetmacro\zy{sqrt(2/3)}

\ifpdf
  \DeclareGraphicsExtensions{.eps,.pdf,.png,.jpg}
\else
  \DeclareGraphicsExtensions{.eps}
\fi

\newsiamremark{remark}{Remark}
\newsiamremark{hypothesis}{Hypothesis}
\crefname{hypothesis}{Hypothesis}{Hypotheses}
\newsiamthm{claim}{Claim}
\newsiamremark{fact}{Fact}
\crefname{fact}{Fact}{Facts}

\title{Well-Posedness of Discretizations for Fractional Elasto-Plasticity\thanks{Submitted to the editors DATE.
\funding{The authors MF and DN received funding from the Deutsche Forschungsgemeinschaft (DFG, German Research Foundation) -- Project-ID 258734477 -- SFB 1173, the Austrian Science Fund (FWF)
under the special research program Taming complexity in PDE systems (grant SFB F65) as well as project I6667-N and from the European Research Council (ERC) under the
European Union’s Horizon 2020
research and innovation programme (Grant agreement No. 101125225).}}}

\author{Michael Feischl\thanks{Institute of Analysis and Scientific Computing, TU Wien 
  (\email{michael.feischl@asc.tuwien.ac.at}).}
\and David Niederkofler\thanks{Institute of Analysis and Scientific Computing, TU Wien 
  (\email{david.niederkofler@asc.tuwien.ac.at}).}
\and Barbara Wohlmuth\thanks{School of Computation, Information and Technology, TU München
(\email{wohlmuth@cit.tum.de})}}

\usepackage{amsopn}

\ifpdf
\hypersetup{
  pdftitle={Well-Posedness of Discretizations for Fractional Elasto-Plasticity},
  pdfauthor={M. Feischl, D. Niederkofler, B. Wohlmuth}
}
\fi

\begin{document}

\maketitle

% REQUIRED
\begin{abstract}
  We consider a fractional plasticity model based on linear isotropic and kinematic hardening as well as a standard von-Mises yield function, where the flow rule is replaced by a Riesz--Caputo fractional derivative. The resulting mathematical model is typically non-local and non-smooth. Our numerical algorithm is based on the well-known radial return mapping and exploits that the kernel is finitely supported.
  We propose explicit and implicit discretizations of the model and show the well-posedness of the explicit in time discretization in combination with a standard finite element approach in space. Our numerical results in 2D and 3D illustrate the performance of the algorithm and the influence of the fractional parameter.
\end{abstract}

% REQUIRED

% REQUIRED

\section{Introduction}
Fractional time derivatives play a significant role in the mathematical modeling of physical phenomena and hence can be found in many application driven partial differential systems 
\cite{seki2003fractional,kosztolowicz2013application,Liu18}.  Such derivatives can account for anomalous diffusion related to a continuous time random walk, see, e.g. \cite{gorenflo2002time} and provide a flexible framework for taking into account non-local effects. Subdiffusion models have recently gained much attraction in a wide range of  anomalous transport processes, e.g., in heterogeneous porous media \cite{CB04,HH21}, in cell 
membranes \cite{RSKIFK05} or nanoscale biophysics \cite{Kou08}.
As fractional time derivatives can describe a smooth transition between purely elastic and purely viscous materials, they are attractive to use
in complex rheological models for viscoelasticity. We refer to the pioneering work \cite{bagley1983theoretical} and the more recent contributions \cite{arikoglu_2014,xiao_2016}. The application area for fractional
derivative based models is quite rich and ranges from
the long-term creep behavior of
concrete \cite{hinze_2022} to biological tissue  modeling \cite{MERAL2010939}. We refer also to the textbook
\cite{mainardi2022fractional} on fractional viscoelasticity 
and the references therein.
Being pioneered in the work \cite{diethelm2010analysis}, the analysis
of fractional systems is particularly challenging due to the fact that standard results such as chain or product rules do not hold and even for the handling of ordinary differential equations more advanced Gronwall inequalities are needed \cite{RA17,YE20071075}. As a consequence, energy estimates are more demanding and require specially tailored approaches. We refer to \cite{zacher2019time,zacher2009weak} for weak solution concepts for time fractional diffusion equations and evolutionary integro-differential equations. In \cite{diethelm2020good} specific challenges in case
of the numerical solution of fractional-order differential problems 
are adressed.

In the present work, we consider a simplified small strain plasticity setting.
General elasto-plasticity models are widely used in practical engineering, but experiments often show that the constitutive response of the material is non-associated~\cite{Qu.2022}, i.e., the direction of plastic strain rate is not necessarily orthogonal to the yield surface. To account for this, advanced models with plastic potential functions are used, but result in very complex approaches. Another direction  aims to use so-called fractional plasticity models, which avoid plastic potentials all along and are able to incorporate non-local and history dependent behavior, see, e.g., \cite{Sumelka.2014,Sumelka.2014b}.
The mathematical model we consider is based on linear isotropic  and kinematic hardening and a standard von-Mises yield function \cite{Han.1999} and replaces the classical flow rule  by a Riesz--Caputo fractional derivative \cite{Sumelka.2016,Sumelka.2018} with a fixed and finitely supported kernel.
We use this model as a prototype problem where two challenges need to be tackled simultaneously. The non-locality of the fractional operator meets the non-smoothness of the inequality constraint. It is well known that inequality  constraints can be often reformulated as variational inequalities or can be equivalently rewritten as a unconstrained but non-linear saddle point problem in terms of a so called NCP (nonlinear complementarity problem) function. There is a rich literature on semi-smooth Newton methods, primal-dual active set strategies and radial return mappings. We only refer to a few:
Radial return strategies have been originally introduced by Simo and Hughes  \cite{Simo.1998} for plasticity formulations but have been by now generalized to many settings and application areas, see the review \cite{Wohlmuth_2011}.
Semi-smooth Newton methods in the context of elasto-plasticity had been discussed in a series of papers by Christensen and coworkers \cite{chr98,CHR02,chr13}. We refer to \cite{HIK02,Che07,hinze_2022} for a mathematical analysis  and the concept of NCP functions as well as to the textbook \cite{deuflhard2011newton}.
For an overview of fractional calculus in plasticity modeling, we refer to~\cite{Qu.2022}, which also provides some numerical algorithms for computing return mappings. An application of Caputo fractional derivatives to rock-fill materials and soil is found in~\cite{Wu.2022, Lu.2019}, while other authors~\cite{Shen.2022, Lu.2022} use Riemann-Liouville fractional derivatives for modeling. Granular soils are modeled in~\cite{Sun.2018} using a combination of left and right Caputo fractional derivatives. Algorithms to evaluate the fractional derivatives can be found in~\cite{Sumelka.2016, Sumelka.2018}. Algorithms to evaluate the corresponding return mapping are introduced in \cite{Qu.2021, Zhou.2020}. Although general associated plasticity models are analyzed mathematically in a very thorough manner (see, e.g., \cite{Han.1999} for an overview), non-associated plasticity models are not treated as rigorously. This is partially due to the more recent developments in this topic, but also due to the much higher complexity. First results for existence and uniqueness analysis for plastic potential based non-associated models can be found in \cite{Lubliner.1984, CastrenzePolizzotto.1998}. However, for non-associated fractional plasticity models, such results are still missing.

Our main theoretical results show well-posedness of the space-time discretization and of a suitable return mapping in case of a non-local and non-smooth plasticity model. The results can be generalized to different PDE models having the same characteristic structure. Our numerical results illustrate the influence of the fractional exponent and the kernel support on the mechanical behavior and the convergence of the associated semi-smooth Newton iteration.

The remainder of the work is structured as follows: In the following subsections, we introduce the model problem and the notation.
Standard low order time integration and the applied finite element discretization of our model are specified in Section~\ref{sec:disc} in which we also introduce
the fractional component based on a finitely supported kernel. In Section~\ref{sec:return}, we show that our radial return mapping is well-defined and semi-smooth and that the elements of its subdifferential are positive definite. Based on these characteristic properties of our radial return mapping, we are then in a position  to prove the well-posedness of the explicit Euler in Section~\ref{sec:wellposed}. Section~\ref{sec:semi} is devoted to the radial return mapping of the implicit Euler discretization, and we show semi-smoothness. 
Finally, the numerical results, presented in Section~\ref{sec:numerics}, underline our theoretical findings and illustrate the influence of the model and discretization parameters.

\subsection{Notation}
We will denote vectors as bold lower-case letters, e.g. $\mathbf{v}=(v_1,\ldots,v_d) \in \mathbb{R}^d$, while matrices or second-order tensors are depicted by bold lower-case Greek letters, e.g. $\boldsymbol{\tau} \in \mathbb{R}^{d \times d}$.  Fourth-order tensors will be denoted by bold upper-case letters, e.g. $\mathbf{C} \in \mathbb{R}^{d \times d \times d \times d}$, or calligraphic upper-case letters, e.g. $\mathcal{S} \in \mathbb{R}^{d \times d \times d \times d}$. We denote the Euclidean inner product by $\mathbf{v}\cdot \mathbf{w}$ as well as the Frobenius inner product or double contraction by $\boldsymbol{\tau}: \boldsymbol{\sigma}=\sum_{i,j=1}^d \tau_{ij}\sigma_{ij}$. Similarly, the product of a fourth-order tensor with a matrix is defined as
$(\mathbf{C}\boldsymbol{\tau})_{ij}=\sum_{k,l=1}^d C_{ijkl}\tau_{kl}$.  The outer or tensor product is defined as
$(\mathbf{v} \otimes \mathbf{w})_{ij} = v_i w_j$ for vectors and $(\boldsymbol{\tau} \otimes \boldsymbol{\sigma})_{ijkl}=\tau_{ij}\sigma_{kl}$ for tensors. The identity matrix is denoted by $\id$ and the fourth order identity tensor is denoted by $\Id$. 

We use the standard definitions ${\rm dev}(\boldsymbol{\tau})=\boldsymbol{\tau}-\frac{1}{d}\text{tr}(\boldsymbol{\tau})\id$ and $\text{tr}(\boldsymbol{\tau})=\sum_{i=1}^d\tau_{ii} $ and denote by $|\cdot|$ the Euclidean norm for vectors and the Frobenius norm for matrices, respectively. For scalar-valued functions we denote the gradient by $\partial_{(\cdot)}f(\cdot)$.
For vector- or matrix-valued functions, $\partial_{(\cdot)}\mathbf{f}(\cdot)$ denotes the Jacobian, which can be a second- or fourth-order tensor. The time-derivative of a (tensor-valued) function $\mathbf{h}$ is denoted by $\dot{\mathbf{h}}$.
Finally, throughout the work, $\Omega\subseteq \mathbb{R}^d$, $d=2,3$ is a bounded domain with $C^1$-boundary $\partial \Omega= \overline \Gamma_D \cup \overline \Gamma_N$, $ \Gamma_D \cap \Gamma_N = \emptyset$ (Dirichlet and Neumann boundary).

\subsection{The model problem}
We consider the following elasto-plastic problem (see~\cite{Han.1999} for details). On the bounded piecewise $C^1$-domain $\Omega \subset \mathbb{R}^d$ with outward pointing unit normal vector $\mathbf{n}$, 
we introduce the stress tensor $\boldsymbol{\sigma}: \Omega \times \mathbb{R} \to M^d$ ($M^d \subset \mathbb{R}^{d \times d}$ is the space of symmetric matrices), the strain tensor $\boldsymbol{\eps}: \Omega\times \mathbb{R} \to M^d$, and the displacement vector $\mathbf{u}:\Omega\times \mathbb{R} \to \mathbb{R}^d$.   We assume that the total strain $\boldsymbol{\eps}$ can be decomposed additively in a purely elastic part and a purely plastic part, i.e., 
\begin{align*}
	\boldsymbol{\eps}=\boldsymbol{\eps}^e+\boldsymbol{\eps}^p
\end{align*}
and consider a linear-isotropic elasticity law for the stress-strain relation 
$
	\boldsymbol{\sigma}(\mathbf{x},t) = \mathbf{C}\boldsymbol{\eps}^e(\mathbf{x},t),
$ 
where $\boldsymbol{C}$ is a fourth order tensor uniquely defined by
\begin{align*}
    \mathbf{C}\boldsymbol{\eps}(\mathbf{x},t) = 2 \mu {\rm dev}(\boldsymbol{\eps}(\mathbf{x},t))+ \kappa \text{tr}(\boldsymbol{\eps}(\mathbf{x},t))\id
\end{align*}
for the material parameters $\mu, \kappa>0$.

In the following, we will often omit the arguments $\mathbf{x}\in\Omega$ and $t\in\mathbb{R}$ for convenience. Additionally to the linearized strain-displacement relation
\begin{align*}
	\boldsymbol{\eps}(\mathbf{u}) = \frac{1}{2}(\nabla\mathbf{u}+(\nabla\mathbf{u})^T),
\end{align*} 
we will introduce variables relating to the hardening behavior of the material. To that end, we consider linear isotropic and kinematic hardening with variables $\boldsymbol{\xi}_1:\Omega \times \mathbb{R} \to M^d$ and $\xi_2:\Omega \times \mathbb{R} \to \mathbb{R}$, respectively. The hardening relation to the corresponding internal forces $\boldsymbol{\chi}=(\boldsymbol{\chi}^1, \chi^2)$ is chosen as
\begin{align*}
    \boldsymbol{\chi}^1 &= - k_1 \boldsymbol{\xi}_1
    \quad\text{and}\quad 
    \chi^2= -k_2 \xi_2\quad\text{for some }k_1,k_2>0.
\end{align*}
We consider the von-Mises yield function (see, e.g.,~\cite{Sauter})
\begin{align} \label{eq:vMises}
    f(\boldsymbol{\sigma},\boldsymbol{\chi}) = |{\rm dev}(\boldsymbol{\sigma}+\boldsymbol{\chi}^1)|+\chi^2-Y_0,
\end{align}
where $Y_0>0$ is the maximal equivalent stress. In classical models (see, e.g.,~\cite{Sauter}), an associated flow rule is constructed from the (classical) derivatives of the yield function. 
In the present work, we aim to generalize this model by using fractional derivates.

\subsubsection{Fractional derivatives}
 Fractional derivatives generalize the notion of classical derivatives to non-integer orders and one of the most common definitions is that of Riesz-Caputo~\cite{Agrawal.2007} given by 
 \begin{align} \label{eq:fracdef}
		\prescript{RC}{a}{\mathbf{D}}^{\alpha}_b h(t) = \frac{1}{2}\Big(\prescript{C}{a}{\mathbf{D}}^{\alpha}_t h(t) +  (-1)^m\prescript{C}{t}{\mathbf{D}}^{\alpha}_b h(t)\Big)\quad\text{for }a<t<b,
	\end{align}
where $m\in \mathbb{N}$ and  $\alpha>0$ such that $m-1<\alpha<m$,  $a<b \in \mathbb{R}$, and $h \in C^{m}([a,b])$. Here, $\prescript{C}{a}{\mathbf{D}}^{\alpha}_t$ and $\prescript{C}{t}{\mathbf{D}}^{\alpha}_b$ define the left- and right-Caputo derivative (see, e.g., \cite{Caputo.1967} for details) given by
\begin{align*}
		\prescript{C}{a}{\mathbf{D}}^{\alpha}_t h(t) &= \frac{1}{\Gamma(m-\alpha)}\int_{a}^{t}(t-\tau)^{m-1-\alpha} h^{(m)}(\tau) \ d \tau,\\
		\prescript{C}{t}{\mathbf{D}}^{\alpha}_b h(t) &= \frac{(-1)^m}{\Gamma(m-\alpha)}\int_{t}^{b}(\tau-t)^{m-1-\alpha} h^{(m)}(\tau) \ d \tau.
  \end{align*}
  For multivariate functions $h:\mathbb{R}^{k+l} \to \mathbb{R}$, we use the same notation to denote the fractional gradient with respect to $\mathbf{x} \in \mathbb{R}^k$, i.e.,
  \begin{align} \label{eq:vecfrac}
    \prescript{RC}{\mathbf{a}}{\mathbf{D}}^{\alpha}_\mathbf{b} h(\mathbf{x}, \mathbf{y}) = \Big(\prescript{RC}{a_1}{\mathbf{D}}^{\alpha}_{b_1}h(\mathbf{x}, \mathbf{y}), \ldots, 	\prescript{RC}{a_k}{\mathbf{D}}^{\alpha}_{b_k}h(\mathbf{x}, \mathbf{y})\Big).
  \end{align}
  Here $\mathbf{a}, \mathbf{b} \in \mathbb{R}^k$ and $\prescript{RC}{a_m}{\mathbf{D}}^{\alpha}_{b_m} h(\mathbf{x}, \mathbf{y})$, $ 1 \leq m \leq k$ applies the definition \eqref{eq:fracdef} to the $m$-th component of $\mathbf{x}$.
  
  In the following, we will choose the interval symmetrically around the differentiation variable and apply the fractional derivative to the von-Mises yield function $f(\boldsymbol{\sigma},\boldsymbol{\chi})$ defined by \eqref{eq:vMises}. More precisely, we set $\mathbf{a} = \boldsymbol{\sigma}- \boldsymbol{\Delta}$ and $\mathbf{b} = \boldsymbol{\sigma}+ \boldsymbol{\Delta}$ for some matrix $\boldsymbol{\Delta}\in \mathbb{R}^{d\times d}_+$, i.e., $k=d^2$. Therefore, we introduce the shorthand notation
  \begin{align} \label{eq:normfrac}
     \D{\boldsymbol{\sigma}}{{\alpha}}f(\boldsymbol{\sigma},\boldsymbol{\chi}):=\prescript{RC}{\boldsymbol{\sigma}-\boldsymbol{\Delta}}{\mathbf{D}}^{\alpha}_{\boldsymbol{\sigma}+\boldsymbol{\Delta}}f(\boldsymbol{\sigma},\boldsymbol{\chi})\quad\text{and}\quad
      \Dhat{\boldsymbol{\sigma}}{{\alpha}}f(\boldsymbol{\sigma},\boldsymbol{\chi}):=\frac{\D{\boldsymbol{\sigma}}{{\alpha}}f(\boldsymbol{\sigma},\boldsymbol{\chi})}{|\D{\boldsymbol{\sigma}}{{\alpha}}f(\boldsymbol{\sigma},\boldsymbol{\chi})|}
  \end{align}
  and note that $\Dhat{\boldsymbol{\sigma}}{{\alpha}}$ stands for a normalized 
  fractional derivative.

  \subsubsection{The fractional flow rule}
We observe that the partial integer order derivative of the von-Mises yield function with respect to $\boldsymbol{\sigma}$ gives automatically a normalized expression. This does not hold for its fractional counterpart. This observation motivates, the choice of the fractional flow rule as a normalized fractional derivative of $f$ of order $\alpha \in (0,1)$. Given the design parameter $\boldsymbol{\Delta} \in \mathbb{R}^{d \times d}_+$, this reads
\begin{align}
\label{eq:flow1}
	\dot{\boldsymbol{\eps}^p} &=\gamma   \Dhat{\boldsymbol{\sigma}}{{\alpha}} f(\boldsymbol{\sigma},\boldsymbol{\chi}) \\
	\dot{\boldsymbol{\xi}}_1 &= \gamma \partial_{\boldsymbol{\chi}^1} f(\boldsymbol{\sigma},\boldsymbol{\chi})=\gamma\frac{{\rm dev}(\boldsymbol{\sigma}+\boldsymbol{\chi}^1)}{|{\rm dev}(\boldsymbol{\sigma}+\boldsymbol{\chi}^1)|}, \label{eq:flow2}\\
        \dot{\xi}_2 &= \gamma \partial_{\chi^2} f(\boldsymbol{\sigma},\boldsymbol{\chi})=\gamma. \label{eq:flow3}
\end{align}
The function $\gamma : \Omega \times \mathbb{R} \to \mathbb{R}$ with $\gamma \geq 0$ is called the \textit{plastic multiplier}. Such approaches were introduced and studied in \cite{Sumelka.2014,Sumelka.2014b,Sumelka.2016,Sumelka.2018}.
Note that $\boldsymbol{\Delta}$ needs to be chosen in such a way that the fractional gradient above is  well-defined. We point out that the fractional derivative only enters into \eqref{eq:flow1} but not in \eqref{eq:flow2} and \eqref{eq:flow3}.

To complete the model problem, we introduce the complementarity conditions for $f$ and $\gamma$, i.e.,
\begin{align}\label{eq:complementarity}
    f(\boldsymbol{\sigma},\boldsymbol{\chi}) &\leq 0,\quad 
    \gamma \geq 0,\quad 
    f(\boldsymbol{\sigma},\boldsymbol{\chi}) \gamma =0,
\end{align}
i.e., $f<0$, $\gamma=0$ in case of purely elastic deformation and $f=0$, $\gamma>0$ for plastic deformation.
Moreover, we require the balance equation for the stress
\begin{align*}
    \text{div}(\boldsymbol{\sigma})+ \mathbf{b}=0
\end{align*}
given some body force $\mathbf{b}: \Omega \times \mathbb{R} \to \mathbb{R}^d$ as well as standard initial and boundary conditions of the form
\begin{align*}
    \mathbf{u}(\mathbf{x},t)&=\mathbf{u}_D(\mathbf{x},t) \quad \text{on } \Gamma_D, \\
	\boldsymbol{\sigma}(\mathbf{x},t)\mathbf{n}(\mathbf{x}) &= \mathbf{t}_N(\mathbf{x},t) \quad \text{on } \Gamma_N,\\
 	\mathbf{u}(\mathbf{x},0) &=
	\boldsymbol{\sigma}(\mathbf{x},0) =
	\boldsymbol{\chi}^1(\mathbf{x},0) =
    \chi^2(\mathbf{x},0)=
    \gamma(\mathbf{x},0)=0\quad \text{in }\Omega,
\end{align*}
$\mathbf{u}_D$ and $\mathbf{t}_N$ denote the Dirichlet and Neumann boundary data, respectively.
\section{Auxiliary results}
The fractional derivative is consistent with integer order derivatives in the sense
\begin{align*}
	\lim_{\alpha \nearrow m} \prescript{RC}{a}{\mathbf{D}}^{\alpha}_b h(t) = 
    h^{(m)}(t) \quad \text{for} \quad h \in C^{m+1}([a,b]),
\end{align*}
and $\prescript{RC}{a}{\mathbf{D}}^{\alpha}_b h(t)=0$ for constant $h$.

The following lemma shows that the angle between fractional and classical derivatives is bounded if applied to the von-Mises yield function.

\begin{lemma}
	\label{lem:positive}
	For $\alpha \in (0,1)$, $\boldsymbol{\sigma},\boldsymbol{\chi}^1 \in M^d$,  $\boldsymbol{\Delta} \in \mathbb{R}^{d \times d}_+$ and
    ${\rm dev}(\boldsymbol{\sigma}+\boldsymbol{\chi}^1) \neq 0$, we find
    that the fractional derivative $\D{\boldsymbol{\sigma}}{{\alpha}}f(\boldsymbol{\sigma},\boldsymbol{\chi})$ is well-defined and
	\begin{align*}
		\partial_{\boldsymbol{\sigma}}f(\boldsymbol{\sigma},\boldsymbol{\chi}):\D{\boldsymbol{\sigma}}{{\alpha}}f(\boldsymbol{\sigma},\boldsymbol{\chi}) > 0 .
	\end{align*} 
    \end{lemma}
\begin{proof}
We fix $1\leq i,j\leq d$ and note that $r_{ij}:= \sum_{k,l} ({\rm dev}(\boldsymbol{\sigma}+\boldsymbol{\chi}^1)_{kl}^2 - {\rm dev}(\boldsymbol{\sigma}+\boldsymbol{\chi}^1)_{ij}^2 \geq 0 $. We treat only the case $r_{ij}>0$, the case $r_{ij}=0$ follows with similar arguments.
The goal is to show that each component of $\D{\boldsymbol{\sigma}}{{\alpha}}f(\boldsymbol{\sigma},\boldsymbol{\chi})$ has the same sign as the corresponding component of $\partial_{\boldsymbol{\sigma}}f(\boldsymbol{\sigma},\boldsymbol{\chi})$, which immediately implies the statement. To that end, we use \eqref{eq:fracdef}--\eqref{eq:vecfrac} and the explicit representation of $\partial_{\boldsymbol{\sigma}} f(\boldsymbol{\sigma},\boldsymbol{\chi})
    = {\rm dev}(\boldsymbol{\sigma}+\boldsymbol{\chi}^1)  / |{\rm dev}(\boldsymbol{\sigma}+\boldsymbol{\chi}^1) |$,
	to get for $i \neq j$ that
	\begin{align}
		\label{eq:ijfrac}
		\Big(\D{\boldsymbol{\sigma}}{{\alpha}}f(\boldsymbol{\sigma},\boldsymbol{\chi})\Big)_{ij}=\frac{1}{2\Gamma(1-\alpha)}\int_{\sigma_{ij}-\Delta_{ij}}^{\sigma_{ij}+\Delta_{ij}} \frac{\tau + \chi^1_{ij}}{|\sigma_{ij}-\tau|^{\alpha}\sqrt{(\tau+\chi^1_{ij})^2+r_{ij}}} \ d\tau.
	\end{align}
    As $r_{ij} >0$, the integral is bounded. If $\sigma_{ij}-\Delta_{ij}+\chi^1_{ij} \geq 0$, the integral is obviously positive and therefore  also $\Big(\partial_{\boldsymbol{\sigma}}f(\boldsymbol{\sigma},\boldsymbol{\chi})\Big)_{ij} \Big(\D{\boldsymbol{\sigma}}{{\alpha}}f(\boldsymbol{\sigma},\boldsymbol{\chi})\Big)_{ij} > 0$. The same is true if $\sigma_{ij}+\Delta_{ij}+\chi^1_{ij} \leq 0$. It remains to consider the case $\sigma_{ij}-\Delta_{ij}+\chi^1_{ij}<0<\sigma_{ij}+\Delta_{ij}+\chi^1_{ij}$. By substitution in~\eqref{eq:ijfrac}, we obtain
	\begin{align}
    \label{eq:subsitution}
			\Big(\D{\boldsymbol{\sigma}}{{\alpha}}f(\boldsymbol{\sigma},\boldsymbol{\chi})\Big)_{ij} = \frac{1}{2 \Gamma(1-\alpha)}\int_{\sigma_{ij}+\chi^1_{ij}-\Delta_{ij}}^{\sigma_{ij}+\chi^1_{ij}+\Delta_{ij}}\frac{\tau}{|\sigma_{ij}+\chi^1_{ij}-\tau|^{\alpha}\sqrt{\tau^2+r_{ij}}} \ d \tau.
	\end{align}
    We aim to argue that the integral has a definite sign (and in particular is non-zero) by exploiting symmetries of the integrand. To that end, we introduce $g(\tau) = \frac{\tau}{\sqrt{\tau^2+r_{ij}}}$ and note that it is symmetric w.r.t the origin. Moreover, ${|\sigma_{ij}+\chi^1_{ij}-\tau|^{-\alpha}}$ is symmetric around $\tau=\sigma_{ij}+\chi^1_{ij}$, increasing for $\tau<\sigma_{ij}+\chi^1_{ij}$ and decreasing otherwise. In case that $\sigma_{ij}+\chi^1_{ij}<0$, we note $\sigma_{ij}+\chi^1_{ij}+\Delta_{ij}<-\sigma_{ij}-\chi^1_{ij}+\Delta_{ij}$ and hence
	\begin{align*}
		\Big(\D{\boldsymbol{\sigma}}{{\alpha}}f(\boldsymbol{\sigma},\boldsymbol{\chi})\Big)_{ij}
  &<\frac{1}{2 \Gamma(1-\alpha)}\int_{\sigma_{ij}+\chi^1_{ij}-\Delta_{ij}}^{-\sigma_{ij}-\chi^1_{ij}+\Delta_{ij}}\frac{g(\tau) }{|\sigma_{ij}+\chi^1_{ij}-\tau|^\alpha } d \tau.
	\end{align*}
 As $|\sigma_{ij}+\chi^1_{ij}-\tau|^\alpha < |\sigma_{ij}+\chi^1_{ij}+\tau|^\alpha$ for $\tau<0$, the symmetric domain of integration and the symmetry of $g(\tau)$ imply
 \begin{align*}
     \Big(\D{\boldsymbol{\sigma}}{{\alpha}}f(\boldsymbol{\sigma},\boldsymbol{\chi})\Big)_{ij}<0.
 \end{align*}
Analogous arguments show for the case $\sigma_{ij}+\chi^1_{ij}>0$ that
	 \begin{align*}
     \Big(\D{\boldsymbol{\sigma}}{{\alpha}}f(\boldsymbol{\sigma},\boldsymbol{\chi})\Big)_{ij}>0.
 \end{align*}
	This gives us the desired result for $i \neq j$. What is now left is the case $i=j$. We obtain
	\begin{align}
		\Big(\D{\boldsymbol{\sigma}}{{\alpha}}f(\boldsymbol{\sigma},\boldsymbol{\chi})\Big)_{ii}=\frac{1}{2\Gamma(1-\alpha)}\int_{\sigma_{ii}-\Delta_{ii}}^{\sigma_{ii}+\Delta_{ii}} \frac{\frac{d-1}{d}\tau + \widetilde{\chi}^1_{ii}}{|\sigma_{ii}-\tau|^{\alpha}\sqrt{(\frac{d-1}{d}\tau+\widetilde{\chi}^1_{ii})^2+\Tilde{r}_{ii}(\tau)}} \ d\tau,
	\end{align}
	with $\widetilde{\chi}^1_{ii}=\chi^1_{ii}-\frac{1}{d}(\text{tr}(\boldsymbol{\sigma}+\boldsymbol{\chi}^1)-\sigma_{ii})$ and $\Tilde{r}_{ii}(\tau)>0$ sums up the remaining terms, where the other diagonal elements also depend on $\tau$ because of the dev. A tedious calculation shows that this can still be transformed such that it has a similar form as \eqref{eq:subsitution} if $\boldsymbol{\Delta}$ was chosen small enough. Here, the same arguments apply verbatim and conclude the proof.
\end{proof}
The following lemma quantifies the difference between standard integer  and fractional order derivative.
\begin{lemma}\label{lem:alphaconv}
    Let $h: I \to \mathbb{R} $ be a twice continuously differentiable function on the compact interval $I \subset \mathbb{R}$. Let $\infty > \delta_1 \geq \delta \geq \delta_0 >0$ (without loss of generality, we set $\delta_0 \leq 1$) such that the set $\widetilde{I}:=\{x \in I: x-\delta,x+\delta \in I\}$ is non-empty, then for $ 0 < \alpha < 1$ there exists a constant $C < \infty$ independent of $\delta$  and $\alpha$ such that
    \begin{align}\label{eq:fracconv1}
        \sup_{x \in \widetilde{I}}|\prescript{RC}{x-\delta}{\mathbf{D}}_{x+\delta}^\alpha h(x)-h'(x)| \leq C (\|h' \|_{L^\infty(I)}+\|h'' \|_{L^\infty(I)})(1-\alpha).
    \end{align}
\end{lemma}
\begin{proof}
We first show for the left-Caputo derivative by partial integration
    \begin{align*}
    \Big|\prescript{C}{x-\delta}{\mathbf{D}}_x^\alpha h(x)-h'(x)\Big| = \Big|\frac{1}{\Gamma(2-\alpha)}\Big(\delta^{1-\alpha} h'(x-\delta) + \int_{x-\delta}^xh''(y)(x-y)^{1-\alpha}\,dy\Big) - h'(x)\Big|,
    \end{align*}
    which can be written as
    \begin{align*}
        &\Big|\frac{1}{\Gamma(2-\alpha)}\Big(\delta^{1-\alpha} h'(x-\delta)+\int_{0}^\delta h''(x-y)(y^{1-\alpha}-1)\,dy + h'(x)-h'(x-\delta)\Big)-h'(x)\Big|\\
        &\leq \|h'\|_{L^\infty(I)} \Big(\frac{2\max\{|\delta^{1-\alpha}-1|,|1-\Gamma(2-\alpha)|\}}{\Gamma(2-\alpha)} \Big)+\|h''\|_{L^\infty(I)} \frac{\int_{0}^\delta | y^{1-\alpha}-1  | \,dy }{\Gamma(2-\alpha)}.
    \end{align*}
    For $0 < \alpha < 1 $, we have obviously $1 < 2-\alpha < 2$. Moreover,  $ \Gamma \in C^1([1,2] )$ and $\Gamma \geq 1/2$ on $[1,2]$. Taylor expansion shows that
    \begin{align*}
       | \Gamma (2 - \alpha ) - 1 | = | \Gamma( 1 + 1- \alpha) -  \Gamma (1) | \leq C (1- \alpha) .
    \end{align*}
    The rule of L´H\^opital allows us to bound   $g(\alpha,y) = |y^{1 - \alpha } -1 |/ (1- \alpha)$ from above uniformly  on $[0,1] \times [\delta_0, \delta_1]$. Thus, it is sufficient to consider $\int_0^{\delta_0} 1 -y^{1-\alpha} \,dy $ in more detail
    \begin{align*}
        \int_{0}^{\delta_0} 1-y^{1-\alpha} \,dy = \frac{1}{2-\alpha} (\delta_0 -\delta_0^{2-\alpha} )  = \frac{\delta_0}{2-\alpha} (1-\delta_0^{1-\alpha}) = \frac{\delta_0}{2-\alpha} g(\alpha,\delta_0) (1-\alpha).
    \end{align*}
Using the bounds obtained so far, we find
  \begin{align*}
    \Big|\prescript{C}{x-\delta}{\mathbf{D}}_x^\alpha h(x)-h'(x)\Big| \leq
    C (1 - \alpha) \left( \|h'\|_{L^\infty(I)} +\|h''\|_{L^\infty(I)} \right) .
    \end{align*}
    Similar arguments yield the same upper bound for the  negative right-Caputo derivative.  Now \eqref{eq:fracconv1} follows from the definition \eqref{eq:fracdef}. 
\end{proof}
In the following, we will need Lemma~\ref{lem:alphaconv} only for the yield function $f$.
\begin{corollary}\label{cor:alphaconv}
Let $\boldsymbol{\sigma}\in M^d$, $\boldsymbol{\chi} \in M^d\times \mathbb{R}$ such that $f(\boldsymbol{\sigma},\boldsymbol{\chi})\geq 0$.
Let $|{\Delta}_{ij}| \leq  Y_0/2$. Then, the fractional derivative $\Dhat{\boldsymbol{\sigma}}{\alpha}f(\boldsymbol{\sigma},\boldsymbol{\chi})$ is well-defined and
there exists a constant $C < \infty$ independent of $\boldsymbol{\Delta}$
    and $\alpha$ such that
    \begin{align}\label{eq:fracconv2}
        \Big|\Dhat{\boldsymbol{\sigma}}{\alpha}f(\boldsymbol{\sigma},\boldsymbol{\chi})
        -\partial_{\boldsymbol{\sigma}}f(\boldsymbol{\sigma},\boldsymbol{\chi})\Big|\leq C(1-\alpha).
    \end{align}
\end{corollary}
\begin{proof}
   A straightforward computation of the  Hessian with respect to the first argument of $f(\boldsymbol{\tau},\boldsymbol{\chi})$ shows that it is bounded by $C | \rm{dev } (\boldsymbol{\tau}+ \boldsymbol{\chi}^1) |^{-1}$. 
   The fact $f(\boldsymbol{\sigma},\boldsymbol{\chi})\geq 0$ implies $|{\rm dev}(\boldsymbol{\sigma}+\boldsymbol{\chi}^1)|> Y_0$ and hence $|{\rm dev}(\boldsymbol{\sigma}+\boldsymbol{\chi}^1+x\Delta_{ij} \boldsymbol{e}^{ij})|>Y_0/2$ for all $-1\leq x\leq 1$ and $\boldsymbol{e}^{ij}\in\mathbb{R}^{d\times d}$ with $(\boldsymbol{e}^{ij})_{k\ell}=1$ for $(i,j)=(k,\ell)$ and zero else. This bounds the Hessian of $f$ uniformly, and we may apply Lemma~\ref{lem:alphaconv} to
   $h\colon x\mapsto f(\boldsymbol{\sigma} + x\boldsymbol{e}^{ij},\boldsymbol{\chi})$ for all $1\leq i,j\leq d$. This shows~\eqref{eq:fracconv2} for the non-normalized fractional derivative  $\D{\boldsymbol{\sigma}}{\alpha}f(\boldsymbol{\sigma},\boldsymbol{\chi})$.

To prove~\eqref{eq:fracconv2} with normalized fractional derivative, we employ Lemma \ref{lem:positive} to show $|\D{\boldsymbol{\sigma}}{\alpha}f(\boldsymbol{\sigma},\boldsymbol{\chi})| \neq 0$ and ensure that $\Dhat{\boldsymbol{\sigma}}{\alpha}f(\boldsymbol{\sigma},\boldsymbol{\chi})$ is well-defined. We recall that $ |\Dhat{\boldsymbol{\sigma}}{\alpha}f(\boldsymbol{\sigma},\boldsymbol{\chi}) | =1 = | \partial_{\boldsymbol{\sigma}} f(\boldsymbol{\sigma},\boldsymbol{\chi})|$. Using the triangle inequality, we find
    \begin{align*}
        \Big|\Dhat{\boldsymbol{\sigma}}{\alpha}&f(\boldsymbol{\sigma},\boldsymbol{\chi})
        -\partial_{\boldsymbol{\sigma}}f(\boldsymbol{\sigma},\boldsymbol{\chi})\Big|
       \\& \leq
       \Big|\Dhat{\boldsymbol{\sigma}}{\alpha}f(\boldsymbol{\sigma},\boldsymbol{\chi})
       - \D{\boldsymbol{\sigma}}{\alpha}f(\boldsymbol{\sigma},\boldsymbol{\chi})
      \Big| +
      \Big|\D{\boldsymbol{\sigma}}{\alpha}f(\boldsymbol{\sigma},\boldsymbol{\chi})
        -\partial_{\boldsymbol{\sigma}}f(\boldsymbol{\sigma},\boldsymbol{\chi})\Big| \\
        & \leq  \Big|\Dhat{\boldsymbol{\sigma}}{\alpha}f(\boldsymbol{\sigma},\boldsymbol{\chi}) \Big| \, \, \Big| 1-
       | \D{\boldsymbol{\sigma}}{\alpha}f(\boldsymbol{\sigma},\boldsymbol{\chi})
     | \Big| +
      \Big|\D{\boldsymbol{\sigma}}{\alpha}f(\boldsymbol{\sigma},\boldsymbol{\chi})
        -\partial_{\boldsymbol{\sigma}}f(\boldsymbol{\sigma},\boldsymbol{\chi})\Big|
           \\ & =\Big| 1-
       | \D{\boldsymbol{\sigma}}{\alpha}f(\boldsymbol{\sigma},\boldsymbol{\chi})
     | \Big| +
      \Big|\D{\boldsymbol{\sigma}}{\alpha}f(\boldsymbol{\sigma},\boldsymbol{\chi})
        -\partial_{\boldsymbol{\sigma}}f(\boldsymbol{\sigma},\boldsymbol{\chi})\Big| \leq 2
      \Big|\D{\boldsymbol{\sigma}}{\alpha}f(\boldsymbol{\sigma},\boldsymbol{\chi})
        -\partial_{\boldsymbol{\sigma}}f(\boldsymbol{\sigma},\boldsymbol{\chi})\Big| .
    \end{align*}
    This concludes the proof.
\end{proof}
\section{Discretization}\label{sec:disc}
We move to the incremental plasticity setting by discretizing the problem in time. Introducing a partition  $0=t_0<t_1< \ldots < t_N=T$, we set $\Delta t_n=t_n-t_{n-1}$. Functions evaluated at $t_n$ are denoted by a subscript, e.g., $\boldsymbol{\sigma}(\mathbf{x},t_n):=\boldsymbol{\sigma}_n(\mathbf{x})$. We discretize the problem in space by introducing 
\begin{align*}
    V_h=\Big\{ \mathbf{v} \in W^{1, \infty}(\Omega)^d: \mathbf{v}_{|T} \in P^1(T), \forall T \in \mathcal{T}, \mathbf{v}_{|\Gamma_D}=0\Big\}, \quad
    M_h=P^0(\mathcal{T},M^d),\quad
    M_h^s=P^0(\mathcal{T}),
\end{align*}
for some triangulation $\mathcal{T}$ of $\Omega$, i.e for some time step $t_n$ we have that $\mathbf{u}_n \in V_h$, $\boldsymbol{\sigma}_n \in M_h$ and $\chi^2_n \in M_h^s.$ Usually~\cite{Sauter}, temporal derivatives are discreticed by means of implicit methods, which, particularly in the fractional setting, is mathematically harder to access. To circumvent this difficulty,  we use an explicit time-discretization similar to the one introduced  in~\cite{Halilovic.2009, Zhou.2020} and refer to Section~\ref{sec:semi} for some ideas on the implicit discretization. The  explicit Euler scheme results in
\begin{align*}
		\boldsymbol{\eps}^p_n &= \boldsymbol{\eps}^p_{n-1} + \Delta t_n \gamma_{n-1} \Dhat{\boldsymbol{\sigma}_{n-1}}{{\alpha}}f(\boldsymbol{\sigma}_{n-1},\boldsymbol{\chi}_{n-1}),\\
	\boldsymbol{\xi}^1_n &= \boldsymbol{\xi}^1_{n-1}+\Delta t_n \gamma_{n-1} \frac{{\rm dev}(\boldsymbol{\sigma}_{n-1}+\boldsymbol{\chi}^1_{n-1})}{|{\rm dev}(\boldsymbol{\sigma}_{n-1}+\boldsymbol{\chi}^1_{n-1})|},\\
	\xi^2_n&=\xi^2_{n-1} + \Delta t_n \gamma_{n-1}.
\end{align*}
Using the additive strain decomposition, the linear elasticity and hardening laws and introducing the abbreviated notation $\Delta \gamma_{n-1}:=\Delta t_n \gamma_{n-1}$ yield
\begin{subequations}\label{eq:explicit_disc}
\begin{align}
	\boldsymbol{\sigma}_n &= \boldsymbol{\sigma}_{tr}-\Delta \gamma_{n-1}\mathbf{C}  \Dhat{\boldsymbol{\sigma}_{n-1}}{{\alpha}}f(\boldsymbol{\sigma}_{n-1},\boldsymbol{\chi}_{n-1}), \label{eq:stressreturn} \\
	\boldsymbol{\chi}^1_n &= \boldsymbol{\chi}^1_{n-1}- k_1 \Delta \gamma_{n-1} \frac{{\rm dev}(\boldsymbol{\sigma}_{n-1}+\boldsymbol{\chi}^1_{n-1})}{|{\rm dev}(\boldsymbol{\sigma}_{n-1}+\boldsymbol{\chi}^1_{n-1})|},\label{eq:chi1return}\\
	\chi^2_n &= \chi^2_{n-1}-k_2\Delta \gamma_{n-1}, \label{eq:chi2return}
\end{align}
\end{subequations}
where $\boldsymbol{\sigma}_{tr}=\mathbf{C}(\boldsymbol{\eps}(\mathbf{u}_n)-\boldsymbol{\eps}^p_{n-1})$. 
Since the complementarity conditions
\begin{align}\label{eq:disc_complementarity}
    \Delta \gamma_{n-1}\geq0,\quad 
    f(\boldsymbol{\sigma}_{n},\boldsymbol{\chi}_{n}) \leq0,\quad f(\boldsymbol{\sigma}_{n},\boldsymbol{\chi}_{n})\Delta \gamma_{n-1}=0,
\end{align}
can, in general, not be satisfied exactly for an explicit scheme, we distinguish two cases: 

\emph{Case~1}: If $f(\boldsymbol{\sigma}_{tr},\boldsymbol{\chi}_{n-1}) \leq 0$, we set $\Delta \gamma_{n-1}=0$ to ensure $(\boldsymbol{\sigma}_n,\boldsymbol{\chi}_{n})=(\boldsymbol{\sigma}_{tr},\boldsymbol{\chi}_{n-1})$ and hence~\eqref{eq:disc_complementarity} holds.

\emph{Case~2}: If $f(\boldsymbol{\sigma}_{tr},\boldsymbol{\chi}_{n-1})>0$, we linearize $f\approx f_{\rm lin}$ around $(\boldsymbol{\sigma}_{tr},\boldsymbol{\chi}_{n-1})$ and solve $f_{\rm lin}(\boldsymbol{\sigma}_{n},\boldsymbol{\chi}_{n})=0$ for $\Delta\gamma_{n-1}$ using~\eqref{eq:explicit_disc}, i.e.,
\begin{align*}
\Delta \gamma_{n-1}= \frac{f(\boldsymbol{\sigma}_{tr},\boldsymbol{\chi}_{n-1})}{2 \mu \partial_{\boldsymbol{\sigma}}f(\boldsymbol{\sigma}_{tr},\boldsymbol{\chi}_{n-1}):\Dhat{\boldsymbol{\sigma}_{n-1}}{{\alpha}}f^{n-1}+k_1 \partial_{\boldsymbol{\chi}^1} f(\boldsymbol{\sigma}_{tr},\boldsymbol{\chi}_{n-1}):\partial_{\boldsymbol{\chi}^1} f^{n-1}+k_2},
\end{align*}
where $f^{n-1}=f(\boldsymbol{\sigma}_{n-1},\boldsymbol{\chi}_{n-1})$. In both cases, $\boldsymbol{\sigma}_{tr}$ serves as a trial candidate for the stress tensor, which is then updated to $\boldsymbol{\sigma}_{n}$ in order to (approximately) satisfy the complementarity conditions.

The mapping $R_n\colon M^d\to M^d$, $\boldsymbol{\sigma}_{tr}\mapsto \boldsymbol{\sigma}_n$ is called return-mapping $R_n$ (see, e.g.,~\cite{Sauter} or \cite{Simo.1998}) and can be written as
\begin{align}
\label{eq:returnmapping}
    \begin{split} 
        R_n(\boldsymbol{\sigma}) = \boldsymbol{\sigma}-\frac{\max\{0,f(\boldsymbol{\sigma},\boldsymbol{\chi}_{n-1})\}\mathbf{C}\Dhat{\boldsymbol{\sigma}_{n-1}}{\alpha}f^{n-1}}{2 \mu \partial_{\boldsymbol{\sigma}}f(\boldsymbol{\sigma},\boldsymbol{\chi}_{n-1}):\Dhat{\boldsymbol{\sigma}_{n-1}}{\alpha}f^{n-1} +k_1\partial_{\boldsymbol{\sigma}}f(\boldsymbol{\sigma},\boldsymbol{\chi}_{n-1}):\partial_{\boldsymbol{\sigma}} f^{n-1}+k_2}.
    \end{split}
\end{align}
The subscript denotes the hidden dependence on the previous state $t_{n-1}$.
Note that it is not obvious that the denominator is non-zero and hence we will require some assumptions that will be discussed in the following section. 
By considering $\mathbf{u}_n^D=0$ in the following, we can get the discrete solution at $t_n$ by solving the following weak fomulation. For given $\boldsymbol{\sigma}_{n-1},\boldsymbol{\eps}^p_{n-1},\boldsymbol{\chi}_{n-1} \in M_h \times M_h \times M_h \times M_h^s$, we want to find $\mathbf{u}_n \in V_h$ such that 
\begin{align}\label{eq:modelproblem}
    \int_{\Omega} R_n\Big(\mathbf{C}(\boldsymbol{\eps}(\mathbf{u}_n)-\boldsymbol{\eps}^p_{n-1})\Big) : \boldsymbol{\eps}(\mathbf{v}) \ dx= \int_{\Omega} \mathbf{b}_n \cdot \mathbf{v} \ dx+ \int_{\Gamma_N}\mathbf{t}^N_n \cdot \mathbf{v} \ dS, \quad \forall \mathbf{v} \in V_h.
\end{align}
Note, that once a solution $\mathbf{u}_n$ is found, stress and hardening variables have to be updated according to~\eqref{eq:explicit_disc}. The plastic strain is given by $\boldsymbol{\eps}^p_{n}=\boldsymbol{\eps}(\mathbf{u}_n)- \mathbf{C}^{-1}\boldsymbol{\sigma}_n$.
\begin{remark}
    By choosing $V_h,M_h$ and $M_h^s$ as done here we ensure that $\boldsymbol{\sigma}_{n-1},\boldsymbol{\eps}^p_{n-1},\boldsymbol{\chi}_{n-1} \in M_h \times M_h \times M_h \times M_h^s$ yields $\boldsymbol{\sigma}_{n},\boldsymbol{\eps}^p_{n},\boldsymbol{\chi}_{n} \in M_h \times M_h \times M_h \times M_h^s$, which is not true for arbitrary discrete spaces. The following discussion applies to all discrete subspaces of bounded functions with this property and can even be generalized to arbitrary ones as seen in Section \ref{sec:gen}.
\end{remark}
\section{Analysis of the Return-Mapping}\label{sec:return}
This section collects some results on $R_n$ as defined in the previous section. Note that $R_n$ is continuous but not necessarily differentiable everywhere. We will investigate $R_n$ according to the weaker notion of \textit{subdifferentials} and \textit{semismoothness}  introduced in \cite{clarke} and \cite{Facchinei.2003b} respectively. 
A locally Lipschitz continuous function $R: \mathbb{R}^{m \times n} \supset A \to \mathbb{R}^{k \times l}$ is differentiable in a dense set $D\subseteq A$ according to a result by Rademacher~\cite{Evans.2010}. Thus, we may consider the set-valued limiting Jacobian at $\boldsymbol{\tau}\in A$, defined via
	\begin{align*}
		\partial^B R(\boldsymbol{\tau})= \Bigl\{ \mathcal{S} \in \mathbb{R}^{k \times l \times m \times n}: \ \textup{for} \ D \supset \boldsymbol{\tau}_p \to \boldsymbol{\tau},  \ \mathcal{S}=\lim_{p \to \infty}\partial_{\boldsymbol{\tau}} R(\boldsymbol{\tau}_p) \}  \Bigr\}.
	\end{align*}
With this, Clarke's subdifferential $\sD R(\boldsymbol{\tau})$ is defined as the convex hull of the limiting Jacobian, i.e.
	\begin{align*}
		\sD R(\boldsymbol{\tau})=\textup{conv}\Big(\partial^B R(\boldsymbol{\tau})\Big)\subseteq  \mathbb{R}^{k \times l \times m \times n}.
	\end{align*}
Note that if $\sD R(\boldsymbol{\tau})$ contains only a single element, this coincides with the classical derivative and we identify $\sD R(\boldsymbol{\tau}) = \partial_{\boldsymbol{\tau}} R(\boldsymbol{\tau})$.
Finally, the function $R: \mathbb{R}^{m \times n} \supset A \to \mathbb{R}^{k \times l}$, is called semi-smooth at $\boldsymbol{\tau}$ if it is locally Lipschitz, directionally differentiable in a neighborhood of $\boldsymbol{\tau}$, and additionally if any $\mathcal{S} \in \sD R(\boldsymbol{\tau}+\boldsymbol{\theta})$ satisfies
		\begin{align*}
			|R(\boldsymbol{\tau}+\boldsymbol{\theta})-R(\boldsymbol{\tau})- \mathcal{S}\boldsymbol{\theta}|=o(|\boldsymbol{\theta}|), \ \ \textup{as} \ \ \boldsymbol{\theta} \to 0.
		\end{align*}
		If $o(|\boldsymbol{\theta}|)$ can be replaced by $\mathcal{O}(|\mathbf{\theta}|^{1+s})$, for $s \in (0,1]$ we say R is semismooth of order $s$ at $\boldsymbol{\tau}$.
We note that many results for smooth functions transfer to semi-smooth functions, e.g., the chain rule~\cite{Facchinei.2003b}.
The first result is summarized in the following Theorem.
\begin{theorem}
\label{thm:semismooth}
    For given $\boldsymbol{\sigma}_{n-1} \in M^d,\boldsymbol{\chi}^1_{n-1} \in M^d,\chi^2_{n-1}  \leq 0$, $\alpha \in (0,1)$, there exists $\eps>0$ such that $R_n$ defined in~\eqref{eq:returnmapping} is semismooth for all $|\boldsymbol{\sigma}-\boldsymbol{\sigma}_{n-1}|< \eps$  and $|\boldsymbol{\Delta}|<Y_0- 2\eps$.
\end{theorem}
\begin{proof}
  Let us consider two cases:
   
  \emph{Case~1, $f(\boldsymbol{\sigma},\boldsymbol{\chi}_{n-1})<0$:} Then we have $R_n(\boldsymbol{\sigma})=\boldsymbol{\sigma}$ in a neighborhood of $\boldsymbol{\sigma}$ and the assertion follows immediately.

  \emph{Case~2, $f(\boldsymbol{\sigma},\boldsymbol{\chi}_{n-1}) \geq 0$:} Choose $\eps>0$ such that $|\boldsymbol{\Delta}|<Y_0- 2\eps$. This yields the well-definedness of the involved fractional gradients if $|\boldsymbol{\sigma}-\boldsymbol{\sigma}_{n-1}|< \eps$. Furthermore, by the continuity of the denominator in $\boldsymbol{\sigma}$ we have a possibly smaller bound, again denoted by $\eps>0$ such that the denominator is positive for $|\boldsymbol{\sigma}-\boldsymbol{\sigma}_{n-1}|< \eps$, because of Lemma \ref{lem:positive}. Moreover, it is continuously differentiable in $\boldsymbol{\sigma}$ because $|\boldsymbol{\sigma}+\boldsymbol{\chi}^1_{n-1}|$ is sufficiently large. Semismoothness now follows from the Chain-Rule for semismooth functions, since $\max\{0,\cdot\}$, although semismooth, is the only non-differentiable component.

\end{proof}
\begin{remark} Note that in our setting the involved quantities depend on $\boldsymbol{x}$. We can use uniform continuity to show that there is a number $\eps>0$ independent of $\mathbf{x}$ such that for $|\boldsymbol{\tau}-\boldsymbol{\sigma}_{n-1}(\mathbf{x})|< \eps$ we have that $R_n$ is semismooth in $\Omega$, meaning that for all $\mathbf{x} \in \Omega$ the function
  \begin{align*}
      \begin{split}
          &R_n(\boldsymbol{\tau},\mathbf{x})= \boldsymbol{\tau}-\\
          &\frac{\max\{0,f(\boldsymbol{\tau},\boldsymbol{\chi}_{n-1}(\mathbf{x}))\}\mathbf{C}\Dhat{\boldsymbol{\sigma}_{n-1}(\mathbf{x})}{\alpha}f^{n-1}(\mathbf{x})}{2 \mu \partial_{\boldsymbol{\sigma}}f(\boldsymbol{\tau},\boldsymbol{\chi}_{n-1}(\mathbf{x})):\Dhat{\boldsymbol{\sigma}_{n-1}(\mathbf{x})}{\alpha}f^{n-1}(\mathbf{x}) +k_1\partial_{\boldsymbol{\sigma}}f(\boldsymbol{\tau},\boldsymbol{\chi}_{n-1}(\mathbf{x})):\partial_{ \boldsymbol{\sigma}} f^{n-1}(\mathbf{x})+k_2}
      \end{split}
  \end{align*}
  is semismooth at $\boldsymbol{\tau}$.
  \end{remark}
  Let us now turn to the sub-differential of $R_n$ (which exists since $R_n$ is semismooth as shown in the previous discussion). $R_n$ is continuously differentiable if $f(\boldsymbol{\sigma},\boldsymbol{\chi}_{n-1})\neq 0$, i.e., everywhere but for a set of measure zero in $\mathbb{R}^{d\times d}$.
By definition of the sub-differential, we may ignore sets of measure zero, see also~\cite[Section 7.1]{Facchinei.2003b}. Hence, we compute
\begin{subequations}
\begin{align}
\label{eq:s1}
    \sD R_n (\boldsymbol{\sigma}) = \Id\ 
\end{align}
if $f(\boldsymbol{\sigma},\boldsymbol{\chi}_{n-1})<0$ and
\begin{align}\label{eq:s2}
    \sD R_n(\boldsymbol{\sigma}) &= \Id- \frac{\Big(\mathbf{C}\Dhat{\boldsymbol{\sigma}_{n-1}}{{\alpha}}f^{n-1}\Big) \otimes \partial_{\boldsymbol{\sigma}}f(\boldsymbol{\sigma},\boldsymbol{\chi}_{n-1})}{2 \mu \partial_{\boldsymbol{\sigma}}f(\boldsymbol{\sigma},\boldsymbol{\chi}_{n-1}):\Dhat{\boldsymbol{\sigma}_{n-1}}{{\alpha}}f^{n-1}+k_1 \partial_{\boldsymbol{\chi}^1} f(\boldsymbol{\sigma},\boldsymbol{\chi}_{n-1}):\partial_{\boldsymbol{\chi}^1} f^{n-1}+k_2}\notag\\
    &\quad + \Bigg(
    \Big(\mathbf{C}\Dhat{\boldsymbol{\sigma}_{n-1}}{{\alpha}}f^{n-1}\Big) \\
    &\qquad \otimes
    \frac{  \Big( f(\boldsymbol{\sigma},\boldsymbol{\chi}_{n-1})\partial^2_{\boldsymbol{\sigma}^2} f(\boldsymbol{\sigma},\boldsymbol{\chi}_{n-1})\Big(2 \mu\Dhat{\boldsymbol{\sigma}_{n-1}}{{\alpha}}f^{n-1}+ k_1 \partial_{\boldsymbol{\chi}^1} f^{n-1}\Big)\Big)}{\Big(2 \mu \partial_{\boldsymbol{\sigma}}f(\boldsymbol{\sigma},\boldsymbol{\chi}_{n-1}):\Dhat{\boldsymbol{\sigma}_{n-1}}{{\alpha}}f^{n-1}+k_1 \partial_{\boldsymbol{\chi}^1} f(\boldsymbol{\sigma},\boldsymbol{\chi}_{n-1}):\partial_{\boldsymbol{\chi}^1} f^{n-1}+k_2\Big)^2}\Bigg),  \notag
\end{align}
\end{subequations}
if  $f(\boldsymbol{\sigma},\boldsymbol{\chi}_{n-1}) >0$.
If $f(\boldsymbol{\sigma},\boldsymbol{\chi}_{n-1}) =0$, $\sD R_n(\boldsymbol{\sigma})$ is the convex hull of the terms in equations~\eqref{eq:s1}--\eqref{eq:s2}. The second derivative with respect to stress of $f$ can be computed as
\begin{align*}
	\partial^2_{\boldsymbol{\sigma}^2} f(\boldsymbol{\sigma},\boldsymbol{\chi})=\frac{\Id-\frac{1}{d}(\id \otimes \id)}{|{\rm dev}(\boldsymbol{\sigma}+\boldsymbol{\chi}^1)|}-\frac{{\rm dev}(\boldsymbol{\sigma}+\boldsymbol{\chi}^1) \otimes {\rm dev}(\boldsymbol{\sigma}+\boldsymbol{\chi}^1)}{|{\rm dev}(\boldsymbol{\sigma}+\boldsymbol{\chi}^1)|^3}.
\end{align*}
We want to investigate the regularity of the operators in $\sD R_n(\boldsymbol{\sigma})$, which will be crucial for arguments involving the implicit function theorem below.
\begin{theorem}
  \label{thm:rposdef2}
   For given $\boldsymbol{\sigma}_{n-1} \in M^d,\boldsymbol{\chi}^1_{n-1} \in M^d,\chi^2_{n-1}\leq 0, \alpha \in (0,1)$, there exists $\eps>0$ such that for $|\boldsymbol{\sigma}-\boldsymbol{\sigma}_{n-1}|< \eps$ and $|\boldsymbol{\Delta}|<Y_0- 2\eps$ every element of $\sD R_n(\boldsymbol{\sigma})$ is positive-definite if at least one of the following cases is satisfied:
      \begin{itemize}
          \item[(i)] There holds $f(\boldsymbol{\sigma},\boldsymbol{\chi}_{n-1})<0$.
          \item[(ii)] The material constants satisfy $\frac{\max\{ 2 \mu,\kappa d\}}{k_1+k_2}<\frac{-1+\sqrt{5}}{2}$, or
          \item[(iii)] there holds $1 -\alpha \leq \mathcal{O}\Big( \frac{k_1+k_2}{\kappa d}\Big)$, where the hidden constant is independent of $k_1$, $k_2$, and $\kappa d$.
      \end{itemize}
  \end{theorem}
  \begin{remark}
      Note that the constraints on the constants in Theorem \ref{thm:rposdef2} can be interpreted physically. If ($k_1+k_2) \to 0$ we have vanishing hardening behavior which would also result in ill-posedness of the associated plasticity problem, see e.g. \cite{Wieners.1999}. On the other hand if $\kappa \to \infty$, the model approaches the incompressible limit case, where it is well known that numerical instabilities (known as locking) occur, even in simpler, purely elastic problems.
  \end{remark}
  \begin{proof}
      Semismoothness and well-definedness of $R_n$ around $\boldsymbol{\sigma}_{n-1}$ was already shown in Theorem~\ref{thm:semismooth}. This also shows that the terms in~\eqref{eq:s1} and~\eqref{eq:s2} are well-defined. For case~(i), we note that~\eqref{eq:s1} is obviously positive definite. In the other cases~(ii)--(iii), we argue that $\sD R_n(\boldsymbol{\sigma})$ consists of convex combinations of terms in~\eqref{eq:s1}--\eqref{eq:s2}. Since convex combinations of positive operators are positive themselves, it remains to show that all terms~\eqref{eq:s2} are positive definite. To that end, we may assume $f(\boldsymbol{\sigma},\boldsymbol{\chi}_{n-1})\geq 0$. The third term in \eqref{eq:s2} satisfies
      \begin{align*}
          \partial^2_{\boldsymbol{\sigma}^2} f(\boldsymbol{\sigma},\boldsymbol{\chi})\partial_{\boldsymbol{\sigma}}f(\boldsymbol{\sigma},\boldsymbol{\chi})=\frac{{\rm dev}(\boldsymbol{\sigma}+\boldsymbol{\chi}^1)}{|{\rm dev}(\boldsymbol{\sigma}+\boldsymbol{\chi}^1)|^2}-\frac{{\rm dev}(\boldsymbol{\sigma}+\boldsymbol{\chi}^1)}{|{\rm dev}(\boldsymbol{\sigma}+\boldsymbol{\chi}^1)|^2}=0.
      \end{align*}
      This allows us to find $\eps>0$ such that for $|\boldsymbol{\sigma}-\boldsymbol{\sigma}_{n-1}|<\eps$, the third term is arbitrarily close to the term
      \begin{align}
      \label{eq:s2_simple}
          \Bigg(
      \frac{ \Big(\mathbf{C}\Dhat{\boldsymbol{\sigma}_{n-1}}{{\alpha}}f^{n-1}\Big)
       \otimes \Big( 2 \mu f(\boldsymbol{\sigma},\boldsymbol{\chi}_{n-1})\partial^2_{\boldsymbol{\sigma}^2} f(\boldsymbol{\sigma},\boldsymbol{\chi}_{n-1})\Dhat{\boldsymbol{\sigma}_{n-1}}{{\alpha}}f^{n-1}\Big)}{\Big(2 \mu \partial_{\boldsymbol{\sigma}}f(\boldsymbol{\sigma},\boldsymbol{\chi}_{n-1}):\Dhat{\boldsymbol{\sigma}_{n-1}}{{\alpha}}f^{n-1}+k_1 \partial_{\boldsymbol{\chi}^1} f(\boldsymbol{\sigma},\boldsymbol{\chi}_{n-1}):\partial_{\boldsymbol{\chi}^1} f^{n-1}+k_2\Big)^2}\Bigg).
      \end{align}
      For arbitrary $\boldsymbol{\tau} \in M^d$ and using Cauchy-Schwartz as well as the definition of $\partial^2_{\boldsymbol{\sigma}^2} f$, the numerator in \eqref{eq:s2_simple} satisfies
      \begin{align}
      \label{eq:s2third}
              &\boldsymbol{\tau}:\Bigg( \Big(\mathbf{C}\Dhat{\boldsymbol{\sigma}_{n-1}}{{\alpha}}f^{n-1}\Big)
       \otimes  \Big( 2 \mu f(\boldsymbol{\sigma},\boldsymbol{\chi}_{n-1})\partial^2_{\boldsymbol{\sigma}^2} f(\boldsymbol{\sigma},\boldsymbol{\chi}_{n-1})\Dhat{\boldsymbol{\sigma}_{n-1}}{{\alpha}}f^{n-1}\Big)\Bigg) \boldsymbol{\tau} \\
       &\geq -|\boldsymbol{\tau}|\Big|\mathbf{C}\Dhat{\boldsymbol{\sigma}_{n-1}}{{\alpha}}f^{n-1}\Big|\frac{2 \mu f(\boldsymbol{\sigma},\boldsymbol{\chi}_{n-1})  |\text{dev}(\boldsymbol{\tau})|}{|\text{dev}(\boldsymbol{\sigma}+\boldsymbol{\chi}^1_{n-1})|}\Bigg|\Dhat{\boldsymbol{\sigma}_{n-1}}{{\alpha}}f^{n-1}-\partial_{\boldsymbol{\sigma}}f(\boldsymbol{\sigma},\boldsymbol{\chi}_{n-1})\Big(\partial_{\boldsymbol{\sigma}}f(\boldsymbol{\sigma},\boldsymbol{\chi}_{n-1}):\Dhat{\boldsymbol{\sigma}_{n-1}}{{\alpha}}f^{n-1}\Big) \Bigg|.\notag
         \end{align}
      Since $|\partial_{\boldsymbol{\sigma}}f(\boldsymbol{\sigma},\boldsymbol{\chi}_{n-1})|=|\Dhat{\boldsymbol{\sigma}_{n-1}}{{\alpha}}f^{n-1}|=1$, the last expression in~\eqref{eq:s2third} is the difference of a unit length vector and its projection onto another vector and hence satisfies
      \begin{align*}
          \Bigg|\Dhat{\boldsymbol{\sigma}_{n-1}}{{\alpha}}f^{n-1}-\partial_{\boldsymbol{\sigma}} f(\boldsymbol{\sigma},\boldsymbol{\chi}_{n-1})\Big(\partial_{\boldsymbol{\sigma}}f(\boldsymbol{\sigma},\boldsymbol{\chi}_{n-1}):\Dhat{\boldsymbol{\sigma}_{n-1}}{{\alpha}}f^{n-1}\Big) \Bigg| \leq 1.
      \end{align*}
      Combined with the fact that $\frac{f(\boldsymbol{\sigma},\boldsymbol{\chi}_{n-1})}{|\text{dev}(\boldsymbol{\sigma}+\boldsymbol{\chi}^1_{n-1})|} \leq 1$, this shows
      \begin{align}
      \label{eq:3rd_estimate}
      \begin{split}
          &\boldsymbol{\tau}: \Bigg(
      \frac{\Big(\mathbf{C}\Dhat{\boldsymbol{\sigma}_{n-1}}{{\alpha}}f^{n-1}\Big)
       \otimes  \Big( 2 \mu f(\boldsymbol{\sigma},\boldsymbol{\chi}_{n-1})\partial^2_{\boldsymbol{\sigma}^2} f(\boldsymbol{\sigma},\boldsymbol{\chi}_{n-1})\Dhat{\boldsymbol{\sigma}_{n-1}}{{\alpha}}f^{n-1}\Big)}{\Big(2 \mu \partial_{\boldsymbol{\sigma}}f(\boldsymbol{\sigma},\boldsymbol{\chi}_{n-1}):\Dhat{\boldsymbol{\sigma}_{n-1}}{{\alpha}}f^{n-1}+k_1 \partial_{\boldsymbol{\chi}^1} f(\boldsymbol{\sigma},\boldsymbol{\chi}_{n-1}):\partial_{\boldsymbol{\chi}^1} f^{n-1}+k_2\Big)^2}\Bigg) \boldsymbol{\tau}\\
      &\qquad\geq \frac{-2 \mu \Big|\mathbf{C}\Dhat{\boldsymbol{\sigma}_{n-1}}{{\alpha}}f^{n-1}\Big| |\boldsymbol{\tau}|^2}{\Big(2 \mu \partial_{\boldsymbol{\sigma}}f(\boldsymbol{\sigma},\boldsymbol{\chi}_{n-1}):\Dhat{\boldsymbol{\sigma}_{n-1}}{{\alpha}}f^{n-1}+k_1 \partial_{\boldsymbol{\chi}^1} f(\boldsymbol{\sigma},\boldsymbol{\chi}_{n-1}):\partial_{\boldsymbol{\chi}^1} f^{n-1}+k_2\Big)^2}.
      \end{split}
      \end{align}
      By a similar estimate using Cauchy-Schwartz for the second summand in \eqref{eq:s2}, we get for $\mathcal{S}\in \sD R_n(\boldsymbol{\sigma})$ that
      \begin{align}
      \label{eq:fullestimate}
      \begin{split}
          \boldsymbol{\tau}:\mathcal{S} \boldsymbol{\tau} &\geq |\boldsymbol{\tau}|^2 \Bigg(1- \frac{\Big|\mathbf{C} \Dhat{\boldsymbol{\sigma}_{n-1}}{{\alpha}}f^{n-1}\Big|}{2 \mu \partial_{\boldsymbol{\sigma}}f(\boldsymbol{\sigma},\boldsymbol{\chi}_{n-1}):\Dhat{\boldsymbol{\sigma}_{n-1}}{{\alpha}}f^{n-1}+k_1 \partial_{\boldsymbol{\chi}^1} f(\boldsymbol{\sigma},\boldsymbol{\chi}_{n-1}):\partial_{\boldsymbol{\chi}^1} f^{n-1}+k_2}\\
      &\qquad -\frac{2 \mu \Big|\mathbf{C}\Dhat{\boldsymbol{\sigma}_{n-1}}{{\alpha}}f^{n-1}\Big|}{\Big(2 \mu \partial_{\boldsymbol{\sigma}}f(\boldsymbol{\sigma},\boldsymbol{\chi}_{n-1}):\Dhat{\boldsymbol{\sigma}_{n-1}}{{\alpha}}f^{n-1}+k_1 \partial_{\boldsymbol{\chi}^1} f(\boldsymbol{\sigma},\boldsymbol{\chi}_{n-1}):\partial_{\boldsymbol{\chi}^1} f^{n-1}+k_2\Big)^2}\Bigg).
      \end{split}
      \end{align}
      Note that for sufficiently small $\eps>0$ the norms of the denominators in \eqref{eq:fullestimate} are arbitrarily close to $k_1+k_2$ and $(k_1+k_2)^2$ respectively. Furthermore, we have that $|\mathbf{C}\Dhat{\boldsymbol{\sigma}_{n-1}}{{\alpha}}f^{n-1}| \leq \max\{\kappa d, 2 \mu\}$. So the term in the parenthesis in \eqref{eq:fullestimate} is positive if $\frac{\max\{ 2 \mu,\kappa d\}}{k_1+k_2}<\frac{-1+\sqrt{5}}{2}$. Now factoring in the convergence properties for $\alpha$, we first recall that Corollary~\ref{cor:alphaconv} shows $\Big|\partial_{\boldsymbol{\sigma}}f(\boldsymbol{\sigma},\boldsymbol{\chi})- \Dhat{\boldsymbol{\sigma}}{{\alpha}}f(\boldsymbol{\sigma},\boldsymbol{\chi})\Big| := C(\alpha)= \mathcal{O}(1-\alpha)$ as $\alpha \nearrow 1$. This allows us to rewrite \eqref{eq:3rd_estimate} as
      \begin{align*}
      \begin{split}
          &\boldsymbol{\tau}: \Bigg(
      \frac{\Big(\mathbf{C}\Dhat{\boldsymbol{\sigma}_{n-1}}{{\alpha}}f^{n-1}\Big)
       \otimes  \Big( 2 \mu f(\boldsymbol{\sigma},\boldsymbol{\chi}_{n-1})\partial^2_{\boldsymbol{\sigma}^2} f(\boldsymbol{\sigma},\boldsymbol{\chi}_{n-1})\Dhat{\boldsymbol{\sigma}_{n-1}}{{\alpha}}f^{n-1}\Big)}{\Big(2 \mu \partial_{\boldsymbol{\sigma}}f(\boldsymbol{\sigma},\boldsymbol{\chi}_{n-1}):\Dhat{\boldsymbol{\sigma}_{n-1}}{{\alpha}}f^{n-1}+k_1 \partial_{\boldsymbol{\chi}^1} f(\boldsymbol{\sigma},\boldsymbol{\chi}_{n-1}):\partial_{\boldsymbol{\chi}^1} f^{n-1}+k_2\Big)^2}\Bigg) \boldsymbol{\tau}\\
      &\qquad \geq \frac{-2 \mu C(\alpha)\Big|\mathbf{C}\Dhat{\boldsymbol{\sigma}_{n-1}}{{\alpha}}f^{n-1}\Big| |\boldsymbol{\tau}|^2}{\Big(2 \mu \partial_{\boldsymbol{\sigma}}f(\boldsymbol{\sigma},\boldsymbol{\chi}_{n-1}):\Dhat{\boldsymbol{\sigma}_{n-1}}{{\alpha}}f^{n-1}+k_1 \partial_{\boldsymbol{\chi}^1} f(\boldsymbol{\sigma},\boldsymbol{\chi}_{n-1}):\partial_{\boldsymbol{\chi}^1} f^{n-1}+k_2\Big)^2}.
      \end{split}
      \end{align*}
      Subsequently, by estimating the remaining fractional gradients with the help of $C(\alpha)$, we get
      \begin{align}
      \label{eq:fullestimate_alpha}
      \begin{split}
          \boldsymbol{\tau}:\mathcal{S} \boldsymbol{\tau} &\geq |\boldsymbol{\tau}|^2 \Bigg(1- \frac{2 \mu + \max\{2 \mu, \kappa d\} C(\alpha)}{(2 \mu- C(\alpha) +k_1)  \partial_{\boldsymbol{\chi}^1} f(\boldsymbol{\sigma},\boldsymbol{\chi}_{n-1}):\partial_{\boldsymbol{\chi}^1} f^{n-1}+k_2}\\
      &\qquad -\frac{2 \mu C(\alpha)\Big(2 \mu + \max\{ 2 \mu , \kappa d\} C(\alpha) \Big)}{\Big((2 \mu- C(\alpha) +k_1)  \partial_{\boldsymbol{\chi}^1} f(\boldsymbol{\sigma},\boldsymbol{\chi}_{n-1}):\partial_{\boldsymbol{\chi}^1} f^{n-1}+k_2\Big)^2}\Bigg).
      \end{split}
      \end{align}
      We will find $\eps>0$ such that the term in the parenthesis of \eqref{eq:fullestimate_alpha} is positive if
      \begin{align}
      \label{eq:c_alpha_pos}
          \Bigg(1- \frac{2 \mu + \max\{2 \mu, \kappa d\} C(\alpha)}{2 \mu- C(\alpha) +k_1+k_2}
      -\frac{2 \mu C(\alpha)\Big(2 \mu + \max\{ 2 \mu , \kappa d\} C(\alpha) \Big)}{\Big(2 \mu- C(\alpha) +k_1+k_2\Big)^2}\Bigg)>0.
      \end{align}
      Denoting $k_1+k_1:=N$ and $2 \mu < \kappa d := M$, \eqref{eq:c_alpha_pos} will be satisfied if
      \begin{align*}
          C(\alpha)&< \frac{-(M+2)(2 \mu +N)-2 \mu(2\mu -1)}{2 (M(2 \mu-1)-1)}\\
          &\qquad + \frac{\sqrt{\Big((M+2)(2 \mu +N)+2 \mu(2\mu -1)\Big)^2+4\Big(N(2 \mu+N)(M(2 \mu-1)-1)\Big)}}{2 (M(2 \mu-1)-1)}.
      \end{align*}
      By simplifying the numerator we see, that
      \begin{align*}
          \begin{split}
              &-(M+2)(2 \mu +N)-2 \mu(2\mu -1)\\
              &\qquad + \sqrt{\Big((M+2)(2 \mu +N)+2 \mu(2\mu -1)\Big)^2+4\Big(N(2 \mu+N)(M(2 \mu-1)-1)\Big)}\\
              &\leq \frac{4N(2 \mu +N)(M(2 \mu-1)-1)}{(M+2)(2 \mu +N)+2 \mu(2\mu -1)+ \sqrt{\Big((M+2)(2 \mu +N)+2 \mu(2\mu -1)\Big)^2+4\Big(N(2 \mu+N)(M(2 \mu-1)-1)\Big)}}\\
              &\leq \frac{4N(2 \mu +N)(M(2 \mu-1)-1)}{(M+2)(2 \mu +N)}.
          \end{split}
      \end{align*}
      Therefore, $C(\alpha)=1-\alpha$ has to be bounded from above by a term of order $\mathcal{O}(\frac{N}{M})$ for $M/N   \to  \infty$.
      \end{proof}
  \begin{remark}\label{rm:rposdef2} Lemma~\ref{lem:alphaconv}, uniform continuity of~\eqref{eq:s1}--\eqref{eq:s2} in $\boldsymbol{\sigma}$, and the compactness of $\Omega$ show that there exists $\eps>0$ such that $\|\boldsymbol{\sigma}-\boldsymbol{\sigma}_{n-1}\|_\infty< \eps$ implies
  \begin{align*}
      \boldsymbol{\tau}:\mathcal{S}\boldsymbol{\tau} \geq C |\boldsymbol{\tau}|^2\quad\text{for all }   \mathbf{x} \in \Omega, \text{ all }\boldsymbol{\tau} \in \mathbb{R}^{d \times d}, \text{ and all }\mathcal{S} \in \sD R_n(\boldsymbol{\sigma}(\mathbf{x}),\mathbf{x})
  \end{align*}
  with a uniform constant $C>0$, if the assumptions on the parameters of Theorem~\ref{thm:rposdef2} hold.
  \end{remark}
  
  Because the positive-definiteness of $\mathcal{S}$ does not directly imply the positive definiteness of $\mathcal{S} \mathbf{C}$, we require another property of the subdifferential.
  \begin{lemma}
  \label{lem:scposdef}
      If $\kappa \geq \frac{2 \mu}{d}$, $\mathcal{S}\mathbf{C}$ is positive definite for all $\mathcal{S} \in \sD R_n(\boldsymbol{\sigma})$.
  \end{lemma}
  \begin{proof}
    The definitions of $\mathcal{S}$ and $\mathbf{C}$ imply
    \begin{align*}
      \boldsymbol{\eps} : \mathcal{S}\mathbf{C} \boldsymbol{\eps}&=\sum_{i,j=1}^d \sum_{k,l=1}^d \sum_{p,g=1}^d \eps_{ij}S_{ijkl}C_{klpq}\eps_{pq}\\
      &=\sum_{i,j=1}^n \eps_{ij}\Big(\sum_{k,l=1}^d 2\mu S_{ijkl}\eps_{kl}+\sum_{k=1}^dS_{ijkk}(\kappa-\frac{2\mu}{d})\text{tr}(\boldsymbol{\eps})\Big).
    \end{align*}
    Together with $\sum_{i,j,k=1}^d\eps_{ij}S_{ijkk}=\text{tr}(\boldsymbol{\eps})$, this shows
    \begin{align*}
      \boldsymbol{\eps} : \mathcal{S}\mathbf{C} \boldsymbol{\eps}=2 \mu \boldsymbol{\eps} \mathcal{S} \boldsymbol{\eps}+ (\kappa-\frac{2\mu}{d})\text{tr}(\boldsymbol{\eps})^2
    \end{align*}
    and concludes the proof.
  \end{proof}
  \begin{remark}
  Note that for $\alpha \nearrow 1$, $\mathcal{S}$ is positive definite by definition and Lemma \ref{lem:scposdef} is not required and hence the restrictions on the material parameters can be dropped.
  \end{remark}
  Finally, we want to investigate the semismoothness of $R_n$ with respect to all involved variables, i.e.
  \begin{align*}
      \begin{split}
  R(\boldsymbol{\sigma}&,\boldsymbol{\sigma}_{n-1},\boldsymbol{\chi}_{n-1})\\
  &=\boldsymbol{\sigma}-\frac{\max\{0,f(\boldsymbol{\sigma},\boldsymbol{\chi}_{n-1})\}\mathbf{C}\Dhat{\boldsymbol{\sigma}_{n-1}}{\alpha}f^{n-1}}{2 \mu \partial_{\boldsymbol{\sigma}}f(\boldsymbol{\sigma},\boldsymbol{\chi}_{n-1}):\Dhat{\boldsymbol{\sigma}_{n-1}}{\alpha}f^{n-1} +k_1\partial_{\boldsymbol{\sigma}}f(\boldsymbol{\sigma},\boldsymbol{\chi}_{n-1}):\partial_{\boldsymbol{\sigma}} f^{n-1}+k_2}.
  \end{split}
  \end{align*}
  First, an auxiliary lemma about the fractional gradient is needed.
  \begin{lemma}
  \label{lem:fracdiff}
    The function $\D{\boldsymbol{\sigma}}{\alpha}f(\boldsymbol{\sigma},\boldsymbol{\chi})$ is continuously differentiable in $\boldsymbol{\chi}^1$ and $\boldsymbol{\sigma}$, as long as $|\text{dev}(\boldsymbol{\sigma}_{ij}+\boldsymbol{\chi}^1_{ij})|>\text{dev}(\boldsymbol{\Delta})_{ij}$ for all $1\leq i,j\leq d$.
  \end{lemma}
  \begin{proof}
    The $ij$ component of said fractional derivative is defined as
    \begin{align*}
      \frac{1}{\Gamma(1-\alpha)}\Big(\int_{\sigma_{ij}-\Delta_{ij}}^{\sigma_{ij}}\partial_{\sigma_{ij}}f(\boldsymbol{\sigma}_{\tau_{ij}},\boldsymbol{\chi}^1)(\sigma_{ij}-\tau)^{-\alpha} d \tau+\int_{\sigma_{ij}}^{\sigma_{ij}+\Delta_{ij}}\partial_{\sigma_{ij}} f(\boldsymbol{\sigma}_{\tau_{ij}},\boldsymbol{\chi}^1)(\tau-\sigma_{ij})^{-\alpha} d \tau \Big),
    \end{align*}
    where $\boldsymbol{\sigma}_{\tau_{ij}}$ is $\boldsymbol{\sigma}$ with the $ij$ component replaced by $\tau$. By a variable transformation we get
      \begin{align*}
        \frac{1}{\Gamma(1-\alpha)}\Big(\int_0^{\Delta_{ij}}\partial_{\sigma_{ij}} f(\boldsymbol{\sigma}_{(\sigma_{ij}-\tau)_{ij}},\boldsymbol{\chi}^1) \tau^{-\alpha}d \tau+\int_0^{\Delta_{ij}}\partial_{\sigma_{ij}} f(\boldsymbol{\sigma}_{(\sigma_{ij}+\tau)_{ij}},\boldsymbol{\chi}^1) \tau^{-\alpha}d \tau\Big).
      \end{align*}
    The assumptions ensure continuous differentiability of the integrand in $(\boldsymbol{\sigma}, \boldsymbol{\chi}^1)$ in a neighborhood of the integration domain. Therefore, dominated convergence yields the continuous differentiability and
    \begin{align*}
   \begin{split}
       &\partial_{\sigma_{kl}}\D{\boldsymbol{\sigma}}{\alpha}f(\boldsymbol{\sigma},\boldsymbol{\chi})_{ij}\\
       &=\frac{1}{\Gamma(1-\alpha)}\Big(\int_0^{\Delta_{ij}}\partial_{\sigma_{ij}\sigma_{kl}}^2 f(\boldsymbol{\sigma}_{(\sigma_{ij}-\tau)_{ij}},\boldsymbol{\chi}^1) \tau^{-\alpha}d \tau+\int_0^{\Delta_{ij}}\partial_{\sigma_{ij}\sigma_{kl}}^2 f(\boldsymbol{\sigma}_{(\sigma_{ij}+\tau)_{ij}},\boldsymbol{\chi}^1) \tau^{-\alpha}d \tau\Big).
   \end{split}
    \end{align*}
  \end{proof}
  Now we are ready to assess semismoothness of the return mapping as a function of all involved variables.
  \begin{theorem}
  \label{thm:rsemi}
  For given $\alpha \in (0,1)$ and $(\boldsymbol{\sigma}_{n-1},\boldsymbol{\chi}_{n-1}) \in  M^d \times M^d \times \mathbb{R}^-_0$, there exists $\eps>0$ such that if $|\boldsymbol{\sigma}-\boldsymbol{\sigma}_{n-1}|< \eps$ and if $|{\rm dev}(\boldsymbol{\Delta})|<Y_0- 2\eps$, $R$ is semismooth around $(\boldsymbol{\sigma},\boldsymbol{\sigma}_{n-1},\boldsymbol{\chi}_{n-1})$.
  \end{theorem}
  \begin{proof}
     \emph{Case~1, $f(\boldsymbol{\sigma},\boldsymbol{\chi}_{n-1})<0$:} Then we have $R(\boldsymbol{\sigma},\boldsymbol{\sigma}_{n-1},\boldsymbol{\chi}_{n-1})=\boldsymbol{\sigma}$ in a neighborhood of $(\boldsymbol{\sigma},\boldsymbol{\sigma}_{n-1},\boldsymbol{\chi}_{n-1})$ and the assertion follows immediately.
  
     \emph{Case~2, $f(\boldsymbol{\sigma},\boldsymbol{\chi}_{n-1}) \geq 0$:} Choose $\eps>0$ such that $|{\rm dev}(\boldsymbol{\Delta})|<Y_0- 2\eps$. This yields the well-definedness of the involved fractional gradients if $|\boldsymbol{\sigma}-\boldsymbol{\sigma}_{n-1}|< \eps$. Furthermore, by the continuity of the denominator in $\boldsymbol{\sigma}$ we have a possibly smaller bound, again denoted by $\eps>0$ such that the denominator is positive for $|\boldsymbol{\sigma}-\boldsymbol{\sigma}_{n-1}|< \eps$. Lemma~\ref{lem:fracdiff} shows continuous differentiability in $(\boldsymbol{\sigma},\boldsymbol{\sigma}_{n-1},\boldsymbol{\chi}_{n-1})$ because $|\boldsymbol{\sigma}+\boldsymbol{\chi}^1_{n-1}|$ is sufficiently large. Hence, semismoothness now follows by the chain-rule for semismooth functions as $\max\{0,\cdot\}$ is the only non-differentiable component.
  \end{proof}
  \begin{remark}\label{rm:rsemi}
  Since all involved quantities depend on $\boldsymbol{x}$, we can generalize Theorem~\ref{thm:rsemi} such that $R$ is semismooth around $(\boldsymbol{\sigma}(\mathbf{x}),\boldsymbol{\sigma}_{n-1}(\mathbf{x}),\boldsymbol{\chi}_{n-1}(\mathbf{x}))$ for all $\mathbf{x} \in \Omega$ if $\|\boldsymbol{\sigma}-\boldsymbol{\sigma}_{n-1}\|_\infty<\eps$.
  \end{remark}
  
  \section{Well-posedness of the space-time discretizations}\label{sec:wellposed}
   Before going into that in full detail, we want to justify our choice of explicit time-discretization, by considering the limit-case $\alpha \nearrow 1$ in our return-mapping $R_n$. This translates to the explicitly discretized von-Mises model for which~\cite{Sauter} studies implicit discretizations.
  \subsection{Well-posedness of explicit (non-fractional) space-time-discretization}
   For $\alpha \nearrow 1$ the return-mapping dependent on the previous state $(\boldsymbol{\sigma}_{n-1},\boldsymbol{\chi}_{n-1})$ in~\eqref{eq:returnmapping} turns into
   \begin{align}
   \label{eq:return_exp}
      R^e_n(\boldsymbol{\sigma}) = \boldsymbol{\sigma}-\frac{2 \mu \max\{0,f(\boldsymbol{\sigma},\boldsymbol{\chi}_{n-1})\}}{2 \mu \partial_{\boldsymbol{\sigma}}f(\boldsymbol{\sigma},\boldsymbol{\chi}_{n-1}):\partial_{\boldsymbol{\sigma}} f^{n-1} +k_1\partial_{\boldsymbol{\sigma}}f(\boldsymbol{\sigma},\boldsymbol{\chi}_{n-1}):\partial_{\boldsymbol{\sigma}} f^{n-1}+k_2}\partial_{\boldsymbol{\sigma}} f^{n-1}.
    \end{align}
    It will be useful to compare this to the return-mapping resulting from the implicit-discretization of the associated von-Mises problem~\cite{Sauter}
    \begin{align*}
      R^i_n(\boldsymbol{\sigma})=\boldsymbol{\sigma}-\frac{2 \mu \max\{0,f(\boldsymbol{\sigma},\boldsymbol{\chi}_{n-1})\}}{2 \mu +k_1+k_2}\frac{{\rm dev}(\boldsymbol{\sigma}+\boldsymbol{\chi}^1_{n-1})}{|{\rm dev}(\boldsymbol{\sigma}+\boldsymbol{\chi}^1_{n-1})|}
    \end{align*}
      as there holds $R^e_n(\boldsymbol{\sigma}) \to R^i_n(\boldsymbol{\sigma})$ as $\boldsymbol{\sigma}_{n-1} \to \boldsymbol{\sigma}$. 
    We denote $R^{\boldsymbol{\sigma}_{n-1}}_n:=R^e_n$ to show the dependence on the previous stress and recall the finite-dimensional subspaces:
    \begin{align*}
      V_h &\subset (W^{1,\infty}_0)^d,\quad M_h \subset L^\infty(\Omega)^{d \times d}_{\text{sym}}
    \end{align*}
  note that $\boldsymbol{\eps}(V_h) \subseteq M_h$. We select a basis $\boldsymbol{\phi}_1, \ldots, \boldsymbol{\phi}_k$ of $V_h$ and $\boldsymbol{\psi}_1, \ldots, \boldsymbol{\psi}_l$ of $M_h$ and define the map (dependent on $\boldsymbol{\eps}^p_{n-1} \in M_h$, $\boldsymbol{\chi}^1_{n-1} \in L^{\infty}(\Omega)^{d\times d}_{\text{sym}}$ and $\chi^2_{n-1} \in L^\infty(\Omega)$)
  \begin{align}\label{eq:defTnonfrac}
  \begin{split}
    T&\colon V_h \times M_h \to (V_h)^*\\
     \langle T(\mathbf{u}_n,\boldsymbol{\sigma}_{n-1}), \mathbf{v} \rangle &= \int_{\Omega}R^{\boldsymbol{\sigma}_{n-1}}_n\Big(\mathbf{C}(\boldsymbol{\eps}(\mathbf{u}_n)-\boldsymbol{\eps}^p_{n-1})\Big):\boldsymbol{\eps}(\mathbf{v})dx-\int_{\Omega}\mathbf{b}_n \cdot \mathbf{v}dx-\int_{\Gamma_D} \mathbf{t}_n \cdot \mathbf{v}ds,
  \end{split}
  \end{align}
  and its implicit version
  \begin{align*}
  \begin{split}
    T^i&: V_h  \to (V_h)^*\\
  \langle T^i(\mathbf{u}_n), \mathbf{v} \rangle &= \int_{\Omega}R^i_n\Big(\mathbf{C}(\boldsymbol{\eps}(\mathbf{u}_n)-\boldsymbol{\eps}^p_{n-1})\Big):\boldsymbol{\eps}(\mathbf{v})dx-\int_{\Omega}\mathbf{b}_n \cdot \mathbf{v}dx-\int_{\Gamma_D} \mathbf{t}_n \cdot \mathbf{v}ds.
  \end{split}
  \end{align*}
  With this, we can state the main result of the section on well-posedness of~\eqref{eq:modelproblem} for the non-fractional return mapping.
  \begin{theorem}
  \label{thm:solexp}
      Given $\mathbf{b}_n \in L^1(\Omega)^d, \mathbf{t}_n \in L^1(\Gamma_D)^d$, $\boldsymbol{\eps}^p_{n-1} \in M_h$, $\boldsymbol{\chi}^1_{n-1} \in L^{\infty}(\Omega)^{d\times d}_{\text{sym}}$ and $\chi^2_{n-1} \in L^\infty(\Omega)$ such that $\chi^2_{n-1} \leq0$ almost everywhere, let $\hat{\mathbf{u}}_n \in V_h$ satisfy $T^i(\hat{\mathbf{u}}_n)=0.$ Then there exists $\eps>0$ and a semismooth function $\mathbf{u}:B^\infty_\eps(\hat{\boldsymbol{\sigma}}_{n-1}) \to V_h$ such that 
      \begin{align*}
          T(\mathbf{u}(\boldsymbol{\sigma}_{n-1}),\boldsymbol{\sigma}_{n-1})=0, \quad\text{for all } \boldsymbol{\sigma}_{n-1} \in B^\infty_\eps(\hat{\boldsymbol{\sigma}}_{n-1}),
      \end{align*}
      where $\hat{\boldsymbol{\sigma}}_{n-1}=\mathbf{C}(\boldsymbol{\eps}(\hat{\mathbf{u}}_n)-\boldsymbol{\eps}^p_{n-1}),$ and $B^\infty_\eps$ is the ball with radius $\eps$ in $M_h$ with respect to the $L^\infty$-norm.
  \end{theorem}
  We postpone the proof of this result to the end of this section, where we aim to apply the implicit function theorem for semismooth functions~\cite[Proposition 10.3 ]{Sauter} and use $T(\hat{\mathbf{u}}_n,\hat{\boldsymbol{\sigma}}_{n-1})=0$ as an anchor point. Note that existence results of $\hat{\mathbf{u}}_n$ satisfying $T^i(\hat{\mathbf{u}}_n)=0$, are well-known and can be found in \cite{Han.1999}. Following this strategy, it is required to proof assumptions, similar to certain smoothness and invertibility of Jacobians at the anchor point for the classical implicit function theorem. Because the involved spaces are finite-dimensional we can write $T$ as
  \begin{align*}
  \begin{split}
    &T: \mathbb{R}^k \times \mathbb{R}^l \to \mathbb{R}^k\\
    & T(\mathbf{c},\widetilde{\mathbf
    {c}})_i =\int_{\Omega}R^{\boldsymbol{\sigma}_{n-1}}_n\Big(\mathbf{C}(\boldsymbol{\eps}(\mathbf{u}_n)-\boldsymbol{\eps}^p_{n-1})\Big):\boldsymbol{\eps}(\boldsymbol{\phi}_i)dx-\int_{\Omega}\mathbf{b}_n \cdot \boldsymbol{\phi}_idx-\int_{\Gamma_D} \mathbf{t}_n \cdot \boldsymbol{\phi}_ids,
  \end{split}
  \end{align*}
  with $\mathbf{u}_n=\sum_{i=1}^{k}c_i \boldsymbol{\phi}_i$ and $\boldsymbol{\sigma}_{n-1}=\sum_{i=1}^{l}\widetilde{c}_i \boldsymbol{\psi}_i$. The first step in the discussion is to show that $T(\cdot,\cdot)$ is a semismooth mapping. Since we established semismoothness for the return mapping $R_n$, semismoothness of $T$ can be established with the help of the following two lemmas, concerning a Lipschitz bound of the integrand in $T$ and semismoothness of integral operators.
  \begin{lemma}
  \label{lem:lip}
      There exists $\eps>0$ and an integrable function $K$ such that for all $\boldsymbol{\sigma}^1,\boldsymbol{\sigma}^2, \boldsymbol{\sigma}^1_{n-1}, \boldsymbol{\sigma}^2_{n-1}  \in B^\infty_\eps(\hat{\boldsymbol{\sigma}}_{n-1})$ there holds pointwise almost everywhere in $\Omega$ that
      \begin{align*}
          \Big|R^{\boldsymbol{\sigma}^1_{n-1}}_n(\boldsymbol{\sigma}^1):\boldsymbol{\eps}(\boldsymbol{\phi}_i)-R^{\boldsymbol{\sigma}^2_{n-1}}_n(\boldsymbol{\sigma}^2):\boldsymbol{\eps}(\boldsymbol{\phi}_i)\Big|<K\Big(|\boldsymbol{\sigma}^1-\boldsymbol{\sigma}^2|+|\boldsymbol{\sigma}^1_{n-1}-\boldsymbol{\sigma}^2_{n-1}|\Big).
      \end{align*}
  \end{lemma}
  \begin{proof}
      This follows from Lipschitz continuity of $R^{\boldsymbol{\sigma}^1_{n-1}}_n(\boldsymbol{\sigma}^1)$ and uniform boundedness of numerator and denominator, i.e it comes down to calculating the Lipschitz-constants and check their boundedness in $\Omega$.
      Let us partition our domain $\Omega=\Omega_1\cup \Omega_2$ with
      \begin{align*}
          \Omega_1 &= \Big\{\mathbf{x} \in \Omega: |{\rm dev}(\hat{\boldsymbol{\sigma}}_{n-1}(\mathbf{x})+\boldsymbol{\chi}^1_{n-1}(\mathbf{x}))|\leq\frac{-\chi^2_{n-1}(\mathbf{x})+Y_0}{2}\Big\}, \\
          \Omega_2 &= \Big\{\mathbf{x} \in \Omega: |{\rm dev}(\hat{\boldsymbol{\sigma}}_{n-1}(\mathbf{x})+\boldsymbol{\chi}^1_{n-1}(\mathbf{x}))|>\frac{-\chi^2_{n-1}(\mathbf{x})+Y_0}{2}\Big\}.
      \end{align*}
      This allows to choose $\eps_1>0$ such that for $\boldsymbol{\sigma} \in B^\infty_{\eps_1}(\hat{\boldsymbol{\sigma}}_{n-1})$ it holds
      \begin{align*}
          &|{\rm dev}(\boldsymbol{\sigma}(\mathbf{x})+\boldsymbol{\chi}^1_{n-1}(\mathbf{x}))|<-\chi^2_{n-1}(\mathbf{x})+Y_0 \ \text{in} \ \Omega_1,\\
          &|{\rm dev}(\boldsymbol{\sigma}(\mathbf{x})+\boldsymbol{\chi}^1_{n-1}(\mathbf{x}))|>\delta>0 \ \text{in} \ \Omega_2.
      \end{align*}
      We have now for $\mathbf{x} \in \Omega_1$ that $R_n^{\boldsymbol{\sigma}_{n-1}}(\boldsymbol{\sigma})=\boldsymbol{\sigma}$ as long as $\boldsymbol{\sigma}, \boldsymbol{\sigma}_{n-1} \in B^\infty_{\eps_1}(\hat{\boldsymbol{\sigma}}_{n-1}).$ Therefore, the estimate follows immediately with $K=1$. For $\mathbf{x} \in \Omega_2$ we need to employ rules for calculating Lipschitz constants of products and quotients. It holds, always under the assumption $\boldsymbol{\sigma}^1,\boldsymbol{\sigma}^2, \boldsymbol{\sigma}^1_{n-1}, \boldsymbol{\sigma}^2_{n-1}  \in B^\infty_{\eps_1}(\hat{\boldsymbol{\sigma}}_{n-1})$, that
      \begin{align*}
          \Big|\frac{{\rm dev}(\boldsymbol{\sigma}^1+\boldsymbol{\chi}^1_{n-1})}{|{\rm dev}(\boldsymbol{\sigma}^1+\boldsymbol{\chi}^1_{n-1})|}-\frac{{\rm dev}(\boldsymbol{\sigma}^2+\boldsymbol{\chi}^1_{n-1})}{|{\rm dev}(\boldsymbol{\sigma}^2+\boldsymbol{\chi}^1_{n-1})|}\Big| \leq C\frac{|\boldsymbol{\sigma}^1-\boldsymbol{\sigma}^2|}{\delta^2},
      \end{align*}
      where $C>0$ is a constant. Choosing a smaller $0<\eps< \eps_1$ such that
      \begin{align*}
          \partial_{\boldsymbol{\sigma}}f(\boldsymbol{\sigma}^1,\boldsymbol{\chi}^1_{n-1}):\partial_{\boldsymbol{\sigma}}f(\boldsymbol{\sigma}^1_{n-1},\boldsymbol{\chi}^1_{n-1})>0,
      \end{align*}
      in $\Omega_2$ yields by continuity on a bounded domain, that the denominator in $R_n^{\boldsymbol{\sigma}_{n-1}}$, i.e. \eqref{eq:return_exp}, is bounded from below by $k_2$, and from above by $k_2+k_1+2\mu$. Furthermore,
      \begin{align*}
          |\max\{0,f(\boldsymbol{\sigma}^1,\boldsymbol{\chi}_{n-1}\}-\max\{0,f(\boldsymbol{\sigma}^2,\boldsymbol{\chi}_{n-1}\}| \leq |\boldsymbol{\sigma}^1-\boldsymbol{\sigma}^2|.
      \end{align*}
      Since $|{\rm dev}(\boldsymbol{\sigma}^1+\boldsymbol{\chi}^1_{n-1})|\leq|{\rm dev}(\hat{\boldsymbol{\sigma}}_{n-1}+\boldsymbol{\chi}^1_{n-1})|+ \eps <\infty$, because also $\boldsymbol{\chi}^1_{n-1}$ is bounded, the rules for Lipschitz constants of products and quotients gives the result with a constant $K>0$, because $|\boldsymbol{\phi}_i|$ is bounded by assumption.
  \end{proof}
  \begin{lemma}
  \label{lem:ssint}
    Let $\Omega \subset \mathbb{R}^d$. Let $\widetilde{f}: \mathbb{R}^n \times \Omega \to \mathbb{R}$ such that $f(\mathbf{x},\cdot) \in L^1(\Omega)$ for all $\mathbf{x}\in \mathbb{R}^n$ and define $f:\mathbb{R}^n \to \mathbb{R}$ by
    \begin{align*}
      f(\mathbf{x} )=\int_{\Omega}\widetilde{f}(\mathbf{x} ,\mathbf{y} )dy .
    \end{align*}
    Suppose there exists an integrable local Lipschitz constant $k\colon \Omega\to [0,\infty)$ such that almost everywhere $\mathbf{y}\in\Omega$ 
    \begin{align}
        \label{eq:lipae}
      |\widetilde{f}(\mathbf{x}_1,\mathbf{y})-\widetilde{f}(\mathbf{x}_2,\mathbf{y} )|<k(\mathbf{y} )|\mathbf{x}_1-\mathbf{x}_2|, \quad \text{for sufficiently close } \mathbf{x}_1,\mathbf{x}_2 \in \mathbb{R}^n.
    \end{align}
    Then $f$ is semismooth at $\mathbf{x} $ whenever $\widetilde{f}(\cdot,\mathbf{y} )$ is semismooth at $\mathbf{x} $ for almost all $\mathbf{y}  \in \Omega$.
  \end{lemma}
  \begin{proof}
    We first show that $f$ is locally Lipschitz. From \eqref{eq:lipae} we get
    \begin{align*}
      |f(\mathbf{x}_1)-f(\mathbf{x}_2)|\leq \int_{\Omega}|\widetilde{f}(\mathbf{x}_1,\mathbf{y})-\widetilde{f}(\mathbf{x}_2,\mathbf{y})|dy\leq |\mathbf{x} _1-\mathbf{x} _2|\int_\Omega k(\mathbf{y} ) dy ,
    \end{align*}
    which yields the claim.
    We continue by showing that $f$ is directionally differentiable around $\mathbf{x} \in \mathbb{R}^n$ where $\widetilde{f}$ is semismooth for almost all $\mathbf{y} \in \Omega$. Let $\mathbf{v} \neq 0$ $\in$ $\mathbb{R}^n$ and $\mathbf{x} \in \mathbb{R}^n$ be fixed. It holds, again using \eqref{eq:lipae}, for $t$ small enough that
    \begin{align*}
      \Big|\frac{\widetilde{f}(\mathbf{x}+t\mathbf{v},\mathbf{y})-\widetilde{f}(\mathbf{x},\mathbf{y})}{t}\Big| \leq \frac{k(\mathbf{y}) t|\mathbf{v}|}{t}=k(\mathbf{y})|\mathbf{v}|,
    \end{align*}
    which is integrable by our assumptions.
     Dominated convergence yields
    \begin{align*}
      \lim_{t \searrow 0} \frac{f(\mathbf{x}+t\mathbf{v})-f(\mathbf{x})}{t}=\int_{\Omega}\lim_{t \searrow 0}\frac{\widetilde{f}(\mathbf{x}+t\mathbf{v},\mathbf{y})-\widetilde{f}(\mathbf{x},\mathbf{y})}{t}dy
    \end{align*}
  We conclude the proof, by finally showing that $f$ is semismooth at $\mathbf{x}$.
    Note that it holds by (\cite[Theorem 2.7.2]{clarke}) (since $\mathbb{R}^n$ is separable)  that
    \begin{align*}
      \sD f(\mathbf{x}) \subset \int_{\Omega} \sD \widetilde{f}(\mathbf{x},\mathbf{y})dy.
    \end{align*}
  This means, that for every $\mathbf{H} \in \sD f(\mathbf{x}) \subset \mathbb{R}^n$ there exists a map $\mathbf{H} (\cdot): \Omega \to \mathbb{R}^n$ such that, $\mathbf{H} (\mathbf{y}) \in \sD \widetilde{f}(\mathbf{x},\mathbf{y})$.
  Such a map is also called a selection of $\sD \widetilde{f}(\mathbf{x},\cdot)$, and for $\mathbf{v}  \in \mathbb{R}^n$, $\mathbf{H} (\cdot)\cdot \mathbf{v} $ is integrable and
  \begin{align*}
    \mathbf{H}  \cdot \mathbf{v}  =\int_{\Omega}\mathbf{H} (\mathbf{y}) \cdot \mathbf{v}  \ dy.
  \end{align*}
  This shows that every component of $\mathbf{H} (\cdot)$ is integrable and therefore also $|\mathbf{H} (\cdot)|$. Local Lipschitz continuity~\eqref{eq:lipae} implies for $\mathbf{x}'$ being in a neighborhood of $\mathbf{x}$ and $\mathbf{H} \in \sD f(\mathbf{x}')$ that
  \begin{align*}
    \Big|\frac{\widetilde{f}(\mathbf{x}',\mathbf{y})-\widetilde{f}(\mathbf{x},\mathbf{y})-\mathbf{H} (\mathbf{y})(\mathbf{x}'-\mathbf{x})}{|\mathbf{x}'-\mathbf{x}|}\Big|\leq \frac{k(\mathbf{y})|\mathbf{x}'-\mathbf{x}|+|\mathbf{H} (\mathbf{y})\|\mathbf{x}'-\mathbf{x}|}{|\mathbf{x}'-\mathbf{x}|}=k(\mathbf{y})+|\mathbf{H} (\mathbf{y})|.
  \end{align*}
  Finally, we have $|\mathbf{H} (\mathbf{y})|\leq k(\mathbf{y})$ by \cite[Proposition 2.1.2]{clarke} and therefore integrability.
  Thus, dominated convergence implies
  \begin{align*}
    \lim_{\mathbf{x}'\to \mathbf{x}\atop \mathbf{H} \in \sD f(\mathbf{x}')} \frac{f(\mathbf{x}')-f(\mathbf{x})-\mathbf{H} \cdot(\mathbf{x}'-\mathbf{x})}{|\mathbf{x}'-\mathbf{x}|}=\int_{\Omega}\lim_{\mathbf{x}'\to \mathbf{x}\atop \mathbf{H} (\mathbf{y}) \in \sD \widetilde{f}(\mathbf{x}',\mathbf{y})} \frac{\widetilde{f}(\mathbf{x}',\mathbf{y})-\widetilde{f}(\mathbf{x},\mathbf{y})-\mathbf{H} (\mathbf{y})\cdot(\mathbf{x}'-\mathbf{x})}{|\mathbf{x}'-\mathbf{x}|}dy=0,
  \end{align*}
  due to the almost everywhere semismoothness of $\widetilde{f}$, at $\mathbf{x}$.
  \end{proof}
  Lemmas~\ref{lem:lip} and~\ref{lem:ssint} show that if $R^{\boldsymbol{\sigma}_{n-1}}_n(\boldsymbol{\sigma})$ is semismooth in $\boldsymbol{\sigma}_{n-1}$ and $\boldsymbol{\sigma}$ at almost all $\mathbf{x} \in \Omega$, then $T$ is semismooth in $\mathbf{c}$ and $\widetilde{\mathbf{c}}$, because
  \begin{align*}
      &R^{\sum_{p}\widetilde{c}^1_p\boldsymbol{\psi}_p}_n\Big(\mathbf{C}(\boldsymbol{\eps}(\sum_{q} c^1_q \boldsymbol{\phi}_q)-\boldsymbol{\eps}^p_{n-1})\Big):\boldsymbol{\eps}(\boldsymbol{\phi}_i)-R^{\sum_p\widetilde{c}^2_p\boldsymbol{\psi}_p}_n\Big(\mathbf{C}(\boldsymbol{\eps}(\sum_q c^2_q \boldsymbol{\phi}_q)-\boldsymbol{\eps}^p_{n-1})\Big):\boldsymbol{\eps}(\boldsymbol{\phi}_i)\\
      &\qquad\qquad\leq 
      K\Big(\Big|\mathbf{C}\boldsymbol{\eps}(\sum_q c^1_q \boldsymbol{\phi}_q)-\mathbf{C}\boldsymbol{\eps}(\sum_q c^2_q \boldsymbol{\phi}_q)\Big|+\Big|\sum_p\widetilde{c}^1_p\boldsymbol{\psi}_p-\sum_p\widetilde{c}^2_p\boldsymbol{\psi}_p\Big|\Big)\\
      &\qquad\qquad \leq 
      K \Big(\|\mathbf{C}\| (\sum_{q=1}^k|\boldsymbol{\phi}_q|^2)^{1/2} |\mathbf{c}^1-\mathbf{c}^2|+(\sum_{p=1}^l|\boldsymbol{\psi}_p|^2)^{1/2} |\widetilde{\mathbf{c}}^1-\widetilde{\mathbf{c}}^2|\Big).
  \end{align*}
   Since $R^{\boldsymbol{\sigma}_{n-1}}_n(\boldsymbol{\sigma})$ is semismooth for almost all $\mathbf{x} \in \Omega$ in a $L^\infty$ neighborhood of $\hat{\boldsymbol{\sigma}}_{n-1}$, Remark~\ref{rm:rsemi} shows all the assumptions of Lemma~\ref{lem:ssint} are met. Thus, recalling $\hat{\mathbf{u}}_n=\sum_{i=1}^{k}\hat{c}_i \boldsymbol{\phi}_i$ and $\hat{\boldsymbol{\sigma}}_{n-1}=\sum_{i=1}^{l}\hat{\widetilde{c}}_i \boldsymbol{\psi}_i$, $T$ is semismooth in a neighborhood of $(\hat{\mathbf{c}}, \hat{\widetilde{\mathbf{c}}})$.
  
   The next step concerns the subdifferential of $T$ with respect to $\boldsymbol{c}$. Because there is no straightforward connection between partial subdifferentials and full subdifferentials~\cite{clarke}, we require the following result.
  \begin{lemma}
  \label{lem:partgrad}
    For $\mathcal{S} \in \mathbb{R}^{d \times d\times d\times d} $ and $\mathcal{P} \in \mathbb{R}^{d \times d\times d\times d}$, let $[\mathcal{S} \ \mathcal{P}]\in\mathbb{R}^{d \times d \times 2d \times d}$ denote the concatenation in the third dimension. Then, any $[\mathcal{S} \ \mathcal{P}] \in \sD R_n^{\boldsymbol{\sigma}_{n-1}}(\boldsymbol{\sigma})$ satisfies $ \mathcal{S} \in \sD_{\boldsymbol{\sigma}} R_n^{\boldsymbol{\sigma}_{n-1}}(\boldsymbol{\sigma})$ under the assumption that $|\boldsymbol{\sigma}-\boldsymbol{\sigma}_{n-1}|$ is sufficiently small to ensure well-definedness.
  \end{lemma}
  \begin{proof}
      Remember
  \begin{align*}
    R^{\boldsymbol{\sigma}_{n-1}}_n(\boldsymbol{\sigma})=\boldsymbol{\sigma}-\frac{2 \mu \max\{0,f(\boldsymbol{\sigma},\boldsymbol{\chi}_{n-1})\}}{2 \mu \partial_{\boldsymbol{\sigma}}f(\boldsymbol{\sigma},\boldsymbol{\chi}_{n-1}):\partial_{\boldsymbol{\sigma}} f^{n-1} +k_1\partial_{\boldsymbol{\sigma}}f(\boldsymbol{\sigma},\boldsymbol{\chi}_{n-1}):\partial_{\boldsymbol{\sigma}} f^{n-1}+k_2}\partial_{\boldsymbol{\sigma}} f^{n-1}
  \end{align*}
  We have continuous differentiability of $R^{\boldsymbol{\sigma}_{n-1}}_n(\boldsymbol{\sigma})$ everywhere (if $\boldsymbol{\sigma}$ and $\boldsymbol{\sigma}_{n-1}$ are close enough) but in the set
  \begin{align*}
    \mathcal{A}=\Big\{(\boldsymbol{\sigma},\boldsymbol{\sigma}_{n-1}) \in M^{d \times d} \times M^{d \times d}: f(\boldsymbol{\sigma},\boldsymbol{\chi}_{n-1})=0\Big\}
  \end{align*}
  which is a set of measure 0 in $M^{d \times d} \times M^{d \times d}$. The claim holds everywhere outside $\mathcal{A}$. Let now $(\boldsymbol{\sigma},\boldsymbol{\sigma}_{n-1}) \in \mathcal{A}$ and $[\mathcal{S} \ \mathcal{P}] \in \partial^B R^{\boldsymbol{\sigma}_{n-1}}_n(\boldsymbol{\sigma}).$ By definition, there exist sequences $(\boldsymbol{\sigma}^k, \boldsymbol{\sigma}^k_{n-1}) \notin \mathcal{A} \to (\boldsymbol{\sigma}, \boldsymbol{\sigma}_{n-1})$ such that 
  \begin{align*}
    \Big[ \partial_{\boldsymbol{\sigma}} R^{\boldsymbol{\sigma}^k_{n-1}}_n(\boldsymbol{\sigma}^k) \ \partial_{\boldsymbol{\sigma}_{n-1}} R^{\boldsymbol{\sigma}^k_{n-1}}_n(\boldsymbol{\sigma}^k)\Big] \to [\mathcal{S} \ \mathcal{P}].
  \end{align*}
  There are only two possibilities (because otherwise we do not have convergence of the Jacobian due to a discontinuity at the yield surface as seen in formulas~\eqref{eq:s1}--\eqref{eq:s2}), either we have $f(\boldsymbol{\sigma}^k,\boldsymbol{\chi}_{n-1})>0$ or $f(\boldsymbol{\sigma}^k,\boldsymbol{\chi}_{n-1})<0$ for all sufficiently large $k$. Because of the continuity of the Jacobians outside $\mathcal{A}$ we have therefore that $\mathcal{S} \in \partial^B_{\boldsymbol{\sigma}}R^{\boldsymbol{\sigma}_{n-1}}_n(\boldsymbol{\sigma})$. Since $\sD R^{\boldsymbol{\sigma}_{n-1}}_n(\boldsymbol{\sigma})$ only consists of finite convex combinations of $\partial^B R^{\boldsymbol{\sigma}_{n-1}}_n(\boldsymbol{\sigma})$, this yields the claim.
  \end{proof}
  The positive definiteness of the partial subdifferentials of $T$ is addressed in the next lemma. Note that we already use an implication of Lemma \ref{lem:partgrad} in the statement.    
\begin{lemma}
\label{lem:posdef}
	Let $[\sD_{\mathbf{c}}T(\hat{\mathbf{c}},\hat{\widetilde{\mathbf{c}}}) \ \mathcal{P}] \in \sD T(\hat{\mathbf{c}},\hat{\widetilde{\mathbf{c}}})$, then $\sD_{\mathbf{c}}T(\hat{\mathbf{c}},\hat{\widetilde{\mathbf{c}}})$ is well-defined and positive definite.
\end{lemma}
\begin{proof}
By~\cite[Corollary below Thm. 2.6.6]{clarke} and the chain-rule it holds for $\mathbf{b} \in \mathbb{R}^k$ that
\begin{align*}
		\mathbf{b}^T\sD_{\mathbf{c}}T(\hat{\mathbf{c}},\hat{\widetilde{\mathbf{c}}})\mathbf{b}&=\sD_{\mathbf{c}}\Big(\mathbf{b}^TT(\hat{\mathbf{c}},\hat{\widetilde{\mathbf{c}}})\Big)\mathbf{b} \subset \int_{\Omega} \sD_{\mathbf{c}} \Big(\sum_{i=1}^k b_i \boldsymbol{\eps}(\boldsymbol{\phi}_i) : R^{\hat{\boldsymbol{\sigma}}_{n-1}}_n\Big(\mathbf{C}(\boldsymbol{\eps}(\hat{\mathbf{u}}_n)-\boldsymbol{\eps}^p_{n-1})\Big)\Big) \mathbf{b}dx \\
		&= 
        \sum_{i=1}^k b_i \int_\Omega \boldsymbol{\eps}(\boldsymbol{\phi}_i) : \sD_{\mathbf{c}}R^{\hat{\boldsymbol{\sigma}}_{n-1}}_n\Big(\mathbf{C}(\boldsymbol{\eps}(\hat{\mathbf{u}}_n)-\boldsymbol{\eps}^p_{n-1})\Big) \mathbf{b}dx\\
		 & \subset 
        \mathbf{b}^T 
        \left\{\begin{pmatrix}
			\int_{\Omega} \mathcal{S}\mathbf{C}\boldsymbol{\eps}(\boldsymbol{\phi}_1):\boldsymbol{\eps}(\boldsymbol{\phi}_1)dx&
			\cdots & \int_{\Omega} \mathcal{S}\mathbf{C}\boldsymbol{\eps}(\boldsymbol{\phi}_k):\boldsymbol{\eps}(\boldsymbol{\phi}_1)dx\\
			\vdots & \ddots & \vdots\\
            \int_{\Omega} \mathcal{S}\mathbf{C}\boldsymbol{\eps}(\boldsymbol{\phi}_1):\boldsymbol{\eps}(\boldsymbol{\phi}_k)dx & \cdots &\int_{\Omega} \mathcal{S}\mathbf{C}\boldsymbol{\eps}(\boldsymbol{\phi}_k):\boldsymbol{\eps}(\boldsymbol{\phi}_k)dx
		\end{pmatrix}\right\}\mathbf{b},
\end{align*}
where $\mathcal{S}$ denotes selections of $\sD_{\boldsymbol{\sigma}} R_n^{\hat{\boldsymbol{\sigma}}_{n-1}}(\mathbf{C}(\boldsymbol{\eps}(\hat{\mathbf{u}}_{n})-\boldsymbol{\eps}^p_{n-1})),$ such that $\mathcal{S}\mathbf{C}\boldsymbol{\eps}(\boldsymbol{\phi}_i):\boldsymbol{\eps}(\boldsymbol{\phi}_j)$ is integrable for every $1 \leq i,j  \leq k$.
Because $\mathcal{S}\mathbf{C}$ is uniformly (in $\mathbf{x}$) positive definite by Remark~\ref{rm:rposdef2} and Lemma~\ref{lem:scposdef} (independently of the condition on the material-parameters), the result follows.
\end{proof}
\begin{proof}[Proof of Theorem~\ref{thm:solexp}]
  Combining Lemmas~\ref{lem:lip},~\ref{lem:partgrad} and~\ref{lem:posdef} we can apply the implicit function theorem for semismooth functions \cite[Proposition 10.3 ]{Sauter}, which directly proves Theorem~\ref{thm:solexp}.
  \end{proof}
  \subsection{Well-posedness of fractional explicit space-time-discretization}
  Now we turn to the case $\alpha< 1$. We recall and introduce finite-dimensional subspaces
  \begin{align*}
    V_h \subset W^{1,\infty}(\Omega)^d, \quad
    M_h \subset  L^\infty(\Omega)^{d \times d}_{\text{sym}}, \quad
    M_h^{s} \subset  L^\infty(\Omega), \quad
    B_h \subset L^1(\Omega)^d, \quad
    T_h \subset L^1(\Omega)^d,
  \end{align*}
  select a basis $\boldsymbol{\phi}_1, \ldots, \boldsymbol{\phi}_k$ of $V_h$, and define the map
  \begin{align*}
      \begin{split}
    T:V_h &\times M_h \times M_h \times M_h \times M_h^s \times B_h \times T_h \to (V_h)^*\\
    \langle T(\mathbf{u}_n,&\boldsymbol{\eps}^p_{n-1},\boldsymbol{\sigma}_{n-1},\boldsymbol{\chi}_{n-1},\mathbf{b}^h_n,\mathbf{t}^h_n), \mathbf{v}_h \rangle\\
    =&\int_{\Omega}R\Big(\mathbf{C}(\boldsymbol{\eps}(\mathbf{u}_n)-\boldsymbol{\eps}^p_{n-1}),\boldsymbol{\sigma}_{n-1},\boldsymbol{\chi}_{n-1}\Big):\boldsymbol{\eps}(\mathbf{v})dx-\int_{\Omega}\mathbf{b}^h_n \cdot \mathbf{v}dx-\int_{\Gamma_D} \mathbf{t}^h_n \cdot \mathbf{v}ds.
  \end{split}
  \end{align*}
  
  The main result of this section is the well-posedness of~\eqref{eq:modelproblem} for the fractional return mapping.
  \begin{theorem}
  \label{thm:solfrac}
      Consider $(\mathbf{u}_{n-1},\boldsymbol{\eps}^p_{n-2},\boldsymbol{\sigma}_{n-2},\boldsymbol{\chi}_{n-2},\mathbf{b}^h_{n-1},\mathbf{t}^h_{n-1}) \in V_h \times M_h^3\times M_h^s \times B_h \times T_h $ such that $T(\mathbf{u}_{n-1},\boldsymbol{\eps}^p_{n-2},\boldsymbol{\sigma}_{n-2},\boldsymbol{\chi}_{n-2},\mathbf{b}^h_{n-1},\mathbf{t}^h_{n-1})=0$ and $\|\mathbf{C}(\boldsymbol{\eps}(\mathbf{u}_{n-1})-\boldsymbol{\eps}^p_{n-2})-\boldsymbol{\sigma}_{n-2}\|_{\infty}, \boldsymbol{\Delta}>0$ is sufficiently small as well as  $\kappa\geq \frac{2 \mu}{d}$ and $\chi^2_{n-2}\leq 0$ almost everywhere. Suppose furthermore, that the assumptions on the parameters $k_1, k_2, \kappa d, \alpha$ from Theorem~\ref{thm:rposdef2} hold. Then, there exists $\eps>0$,  and a semismooth function $\mathbf{u}: B_\eps^\infty(\boldsymbol{\eps}^p_{n-2},\boldsymbol{\sigma}_{n-2},\boldsymbol{\chi}_{n-2},\mathbf{b}^h_{n-1},\mathbf{t}^h_{n-1}) \to V_h$ such that
      \begin{align*}
          &T(\mathbf{u}(\boldsymbol{\eps}^p,\boldsymbol{\sigma},\boldsymbol{\chi},\mathbf{b},\mathbf{t}),\boldsymbol{\eps}^p,\boldsymbol{\sigma},\boldsymbol{\chi},\mathbf{b},\mathbf{t}) = 0, \\
          &\forall (\boldsymbol{\eps}^p,\boldsymbol{\sigma},\boldsymbol{\chi},\mathbf{b},\mathbf{t}) \in B_\eps^\infty(\boldsymbol{\eps}^p_{n-2},\boldsymbol{\sigma}_{n-2},\boldsymbol{\chi}_{n-2},\mathbf{b}^h_{n-1},\mathbf{t}^h_{n-1}),
      \end{align*}
      where $B_\eps^\infty(\cdot)$ denotes the ball of radius $\eps$ in $M_h^3\times M_h^s \times B_h \times T_h $.
  \end{theorem}
  The proof follows the same steps as the one for Theorem~\ref{thm:solexp}, meaning we aim to apply the implicit function theorem for semismooth functions. In the following we state and prove the pendants of Lemmas~\ref{lem:lip}~\&~\ref{lem:partgrad}--\ref{lem:posdef}. Because of the finite-dimensionality we can write $T$ as
  \begin{align*}
  \begin{split}
      T\colon V_h \times M_h \times M_h \times M_h \times M_h^s \times B_h \times T_h &\to \mathbb{R}^k\\
     T(\mathbf{u}_n,\boldsymbol{\eps}^p_{n-1},\boldsymbol{\sigma}_{n-1},\boldsymbol{\chi}_{n-1},\mathbf{b}^h_n,\mathbf{t}^h_n)_i&=\langle T(\mathbf{u}_n,\boldsymbol{\eps}^p_{n-1},\boldsymbol{\sigma}_{n-1},\boldsymbol{\chi}_{n-1},\mathbf{b}^h_n,\mathbf{t}^h_n), \boldsymbol{\phi}_i \rangle\quad\text{for }i=1,\ldots, k.
  \end{split}
  \end{align*}
  The following lemma gives a Lipschitz type estimate on the integrand in the definition of $T$.
  \begin{lemma}\label{lem:lipfrac}
  Suppose $ \|\mathbf{C}(\boldsymbol{\eps}(\mathbf{u}_{n-1})-\boldsymbol{\eps}^p_{n-2})-\boldsymbol{\sigma}_{n-2}\|_{\infty}$ and $\boldsymbol{\Delta}>0$ sufficiently small, then there exists $\eps>0$ and a constant $K>0$ such that for all $(\boldsymbol{\sigma}^1,\boldsymbol{\sigma}^1_{n-1},\boldsymbol{\chi}_1), (\boldsymbol{\sigma}^2,\boldsymbol{\sigma}^2_{n-1},\boldsymbol{\chi}_2) \in B_{\eps}^\infty\Big(\mathbf{C}(\boldsymbol{\eps}(\mathbf{u}_{n-1})-\boldsymbol{\eps}^p_{n-2}),\boldsymbol{\sigma}_{n-2},\boldsymbol{\chi}_{n-2}\Big) $ it holds pointwise almost everywhere in $\Omega$ that
  \begin{align*}
      \Big|R\Big(\boldsymbol{\sigma}^1,\boldsymbol{\sigma}^1_{n-1},\boldsymbol{\chi}_1\Big)&:\boldsymbol{\eps}(\boldsymbol{\phi}_i)-R\Big(\boldsymbol{\sigma}^2,\boldsymbol{\sigma}^2_{n-1},\boldsymbol{\chi}_2\Big):\boldsymbol{\eps}(\boldsymbol{\phi}_i)\Big|\\
      &<K\Big(|\boldsymbol{\sigma}^1-\boldsymbol{\sigma}^2|+|\boldsymbol{\sigma}^1_{n-1}-\boldsymbol{\sigma}^2_{n-1}|+|\boldsymbol{\chi}^1_1-\boldsymbol{\chi}^1_2|+|\chi^2_1-\chi^2_2|\Big).
  \end{align*}
  \end{lemma}
  \begin{proof}
      Again, we partition $\Omega$ into two subdomains. First, we choose $0<\zeta <1$ such that $|{\rm dev}(\mathbf{C}(\boldsymbol{\eps}(\mathbf{u}_{n-1})-\boldsymbol{\eps}^p_{n-2})+\boldsymbol{\chi}^1_{n-2})|\geq\zeta Y_0$ implies
  \begin{align}\label{eq:distest}
          \text{dist}\Big(0, [{\rm dev}(\mathbf{C}(\boldsymbol{\eps}(\mathbf{u}_{n-1})-\boldsymbol{\eps}^p_{n-2})+\boldsymbol{\chi}^1_{n-2}-\boldsymbol{\Delta}),{\rm dev}(\mathbf{C}(\boldsymbol{\eps}(\mathbf{u}_{n-1})-\boldsymbol{\eps}^p_{n-2})+\boldsymbol{\chi}^1_{n-2}+\boldsymbol{\Delta})]\Big)>c>0,
      \end{align}
      where $c>0$ is fixed. Obviously, this puts restrictions on $\boldsymbol{\Delta}$.
      The domain is then partitioned in
      \begin{align*}
          \Omega_1&=\Big\{\mathbf{x} \in \Omega: |{\rm dev}(\mathbf{C}(\boldsymbol{\eps}(\mathbf{u}_{n-1}(\mathbf{x}))-\boldsymbol{\eps}^p_{n-2}(\mathbf{x}))+\boldsymbol{\chi}^1_{n-2}(\mathbf{x}))|\leq\zeta(-\chi^2_{n-2}(\mathbf{x})+Y_0)\Big\}, \\
          \Omega_2&=\Big\{\mathbf{x} \in \Omega: |{\rm dev}(\mathbf{C}(\boldsymbol{\eps}(\mathbf{u}_{n-1}(\mathbf{x}))-\boldsymbol{\eps}^p_{n-2}(\mathbf{x}))+\boldsymbol{\chi}^1_{n-2}(\mathbf{x}))|>\zeta(-\chi^2_{n-2}(\mathbf{x})+Y_0)\Big\}.
      \end{align*}
      Now choose $\eps>0$ such that for $(\boldsymbol{\sigma},\boldsymbol{\sigma}_{n-1},\boldsymbol{\chi}) \in B_{\eps}^\infty\Big(\mathbf{C}(\boldsymbol{\eps}(\mathbf{u}_{n-1})-\boldsymbol{\eps}^p_{n-2}),\boldsymbol{\sigma}_{n-2},\boldsymbol{\chi}_{n-2}\Big)$ it holds
      \begin{align}\label{eq:uniformfrom0}
      \begin{split}
          &|{\rm dev}(\boldsymbol{\sigma}(\mathbf{x})+ \boldsymbol{\chi}^1(\mathbf{x})|\leq -\chi^2(\mathbf{x})+Y_0,\ \mathbf{x} \in \Omega_1\quad\text{ as well as}\\
          &\text{dist}\Big(0,[{\rm dev}(\boldsymbol{\sigma}_{n-1}(\mathbf{x})+\boldsymbol{\chi}^1(\mathbf{x})-\boldsymbol{\Delta}),{\rm dev}(\boldsymbol{\sigma}_{n-1}(\mathbf{x})+\boldsymbol{\chi}^1(\mathbf{x})+\boldsymbol{\Delta})]\Big) > c_2>0, \ \mathbf{x} \in \Omega_2, 
          \end{split}
      \end{align}
      for some $c_2>0$, which follows from~\eqref{eq:distest} if $\|\mathbf{C}(\boldsymbol{\eps}(\mathbf{u}_{n-1})-\boldsymbol{\eps}^p_{n-2})-\boldsymbol{\sigma}_{n-2}\|_{\infty}$ sufficiently small relative to $\eps$. In the following let $(\boldsymbol{\sigma}^1,\boldsymbol{\sigma}^1_{n-1},\boldsymbol{\chi}_1), (\boldsymbol{\sigma}^2,\boldsymbol{\sigma}^2_{n-1},\boldsymbol{\chi}_2)\in B_{\eps}^\infty\Big(\mathbf{C}(\boldsymbol{\eps}(\mathbf{u}_{n-1})-\boldsymbol{\eps}^p_{n-2}),\boldsymbol{\sigma}_{n-2},\boldsymbol{\chi}_{n-2}\Big)$.
      In $\Omega_1$ we have that $R(\boldsymbol{\sigma}^1,\boldsymbol{\sigma}^1_{n-1},\boldsymbol{\chi}_1)=\boldsymbol{\sigma}_1$, therefore the estimate follows immediately with $K=1$. A more thorough investigation is needed in $\Omega_2$.  Because of~\eqref{eq:uniformfrom0}, we show analogously to Lemma~\ref{lem:lip} that $\Dhat{\boldsymbol{\sigma}^1_{n-1}}{\alpha}f(\boldsymbol{\sigma}^1_{n-1},\boldsymbol{\chi}_1)$ is continuous in $(\boldsymbol{\sigma}^1_{n-1},\boldsymbol{\chi}_1)$ at all $\mathbf{x} \in \Omega_2$.
      Furthermore, there obviously holds
  \begin{align}
      \label{eq:fracbounded}
      \Big|\Dhat{\boldsymbol{\sigma}^1_{n-1}}{\alpha}f(\boldsymbol{\sigma}^1_{n-1},\boldsymbol{\chi}_1)\Big| =1.
  \end{align}
  This implies that $|\max\{0,f(\boldsymbol{\sigma}^1,\boldsymbol{\chi}_1)\}\Dhat{\boldsymbol{\sigma}^1_{n-1}}{\alpha}f(\boldsymbol{\sigma}^1_{n-1},\boldsymbol{\chi}_1)|$ is bounded from above uniformly in $\Omega_2$, independent of $(\boldsymbol{\sigma}^1,\boldsymbol{\sigma}^1_{n-1},\boldsymbol{\chi}_1).$
  Let us turn our attention to the denominator, i.e. 
  \begin{align*}
      2 \mu \partial_{\boldsymbol{\sigma}}f(\boldsymbol{\sigma}^1,\boldsymbol{\chi}_1):\Dhat{\boldsymbol{\sigma}^1_{n-1}}{\alpha}f(\boldsymbol{\sigma}^1_{n-1},\boldsymbol{\chi}_1) +k_1\partial_{\boldsymbol{\sigma}}f(\boldsymbol{\sigma}^1,\boldsymbol{\chi}_1):\partial_{\boldsymbol{\sigma}}f(\boldsymbol{\sigma}^1_{n-1},\boldsymbol{\chi}_1)+k_2.
  \end{align*}
  Lemma~\ref{lem:positive} implies the existence of (a possibly smaller) $\eps>0$ such that the inner products are positive, therefore the denominator is bounded from below by $k_2$ uniformly in $\Omega_2$ and independent of $(\boldsymbol{\sigma}^1,\boldsymbol{\sigma}^1_{n-1},\boldsymbol{\chi}_1)$. It is also bounded from above because of Cauchy--Schwartz inequality and~\eqref{eq:fracbounded}.
  The remaining step is to calculate the Lipschitz constants for numerator and denominator. For the fractional gradient we get via  a similar calculation as in Lemma~\ref{lem:alphaconv} that
  \begin{align*}
  \begin{split}
      \Big|\D{\boldsymbol{\sigma}^1_{n-1}}{\alpha}&f(\boldsymbol{\sigma}^1_{n-1},\boldsymbol{\chi}_1)-\D{\boldsymbol{\sigma}^2_{n-1}}{\alpha}f(\boldsymbol{\sigma}^2_{n-1},\boldsymbol{\chi}_2)\Big| \\
      &\lesssim \frac{ \sum_{ij}\Delta_{ij}^{1-\alpha}}{c_2^2\Gamma(2-\alpha)}\Big(|\boldsymbol{\sigma}^1_{n-1}-\boldsymbol{\sigma}^2_{n-1}|+|\boldsymbol{\chi}^1_1-\boldsymbol{\chi}^1_2|\Big).
  \end{split}
  \end{align*} 
  A similar estimate can be calculated  for $\Dhat{\boldsymbol{\sigma}^1_{n-1}}{\alpha}f(\boldsymbol{\sigma}^1_{n-1},\boldsymbol{\chi}_1)$, because it is continuous on a compact domain and therefore bounded from above and below. For the classical gradient we already obtained an estimate in Lemma~\ref{lem:lip}. Bringing everything together yields the result, where $K>0$ depends only on $\alpha$ and $\boldsymbol{\Delta}$.
  \end{proof}
  It follows from Lemma~\ref{lem:ssint} that $T$ is semismooth in $(\mathbf{u}_n,\boldsymbol{\eps}^p_{n-1},\boldsymbol{\sigma}_{n-1},\boldsymbol{\chi}_{n-1},\mathbf{b}^h_n,\mathbf{t}^h_n)$ around the given solution, whenever $R$ is semismooth in ($\mathbf{C}(\boldsymbol{\eps}(\mathbf{u}_n)-\boldsymbol{\eps}^p_{n-1}),\boldsymbol{\sigma}_{n-1},\boldsymbol{\chi}_{n-1}$) in an $L^\infty$ neighborhood of $(\mathbf{C}(\boldsymbol{\eps}(\mathbf{u}_{n-1})-\boldsymbol{\eps}^p_{n-2}),\boldsymbol{\sigma}_{n-2},\boldsymbol{\chi}_{n-2})$ for almost all $\mathbf{x} \in \Omega$, which is ensured by Theorem~\ref{thm:rsemi}.
  The next lemma states the relationship between partial subdifferentials and full subdifferentials as already used in the non-fractional case.
  \begin{lemma}\label{lem:partgradfrac}
  For any $[\mathcal{S} \ \mathcal{P}] \in \sD R(\boldsymbol{\sigma},\boldsymbol{\sigma}_{n-1},\boldsymbol{\chi}_{n-1})$, there holds $ \mathcal{S} \in \sD_{\boldsymbol{\sigma}} R(\boldsymbol{\sigma},\boldsymbol{\sigma}_{n-1},\boldsymbol{\chi}_{n-1})$.
  \end{lemma}
  \begin{proof}
     The proof transfers verbatim from Lemma~\ref{lem:partgrad}.
  \end{proof}
  For the invertibility of the elements of the subgradient, we identify $\mathbf{u}_{n-1}$ with the corresponding vector $\mathbf{c}\in \mathbb{R}^k$ Then, statement and proof are very similar to the previous case in Lemma~\ref{lem:posdef}.
  \begin{lemma}\label{lem:posdeffrac}
    Assume that the subderivative $[\widetilde{\mathcal{S}}\ \mathcal{P}]\in \sD T(\mathbf{u}_{n-1},\boldsymbol{\eps}^p_{n-2},\boldsymbol{\sigma}_{n-2},\boldsymbol{\chi}_{n-2},\mathbf{b}^h_{n-1},\mathbf{t}^h_{n-1})$ satisfies $\widetilde{\mathcal{S}}=\sD_{\mathbf{c}}T(\mathbf{u}_{n-1},\boldsymbol{\eps}^p_{n-2},\boldsymbol{\sigma}_{n-2},\boldsymbol{\chi}_{n-2},\mathbf{b}^h_{n-1},\mathbf{t}^h_{n-1})$. Then all matrices in $\widetilde{\mathcal{S}}$ are positive definite.
  \end{lemma}
  \begin{proof}
  Denote by $\mathcal{S}$ the selections of $\sD_{\boldsymbol{\sigma}} R(\mathbf{C}(\boldsymbol{\eps}(\mathbf{u}_{n-1})-\boldsymbol{\eps}^p_{n-1},\boldsymbol{\sigma}_{n-2},\boldsymbol{\chi}_{n-2})$, such that $\mathcal{S}\mathbf{C}\boldsymbol{\eps}(\boldsymbol{\phi}_i):\boldsymbol{\eps}(\boldsymbol{\phi}_j)$ is integrable for every $1 \leq i,j  \leq k$, we get by using chain-rule and as in Lemma \ref{lem:posdef} by \cite[Corollary below Thm. 2.6.6]{clarke}  that for $\mathbf{b} \in \mathbb{R}^k$, we have
  \begin{align*}
      \begin{split}
          \mathbf{b}^T\sD_{\mathbf{c}}&T(\mathbf{u}_{n-1},\boldsymbol{\eps}^p_{n-2},\boldsymbol{\sigma}_{n-2},\boldsymbol{\chi}_{n-2},\mathbf{b}^h_{n-1},\mathbf{t}^h_{n-1})\mathbf{b}\\
          &\subset \mathbf{b}^T 
          \left\{\begin{pmatrix}
        \int_{\Omega} \mathcal{S}\mathbf{C}\boldsymbol{\eps}(\boldsymbol{\phi}_1):\boldsymbol{\eps}(\boldsymbol{\phi}_1)dx&
        \cdots & \int_{\Omega} \mathcal{S}\mathbf{C}\boldsymbol{\eps}(\boldsymbol{\phi}_k):\boldsymbol{\eps}(\boldsymbol{\phi}_1)dx\\
        \vdots & \ddots & \vdots\\
              \int_{\Omega} \mathcal{S}\mathbf{C}\boldsymbol{\eps}(\boldsymbol{\phi}_1):\boldsymbol{\eps}(\boldsymbol{\phi}_k)dx & \cdots &\int_{\Omega} \mathcal{S}\mathbf{C}\boldsymbol{\eps}(\boldsymbol{\phi}_k):\boldsymbol{\eps}(\boldsymbol{\phi}_k)dx
      \end{pmatrix}\right\}\mathbf{b}.
      \end{split}
  \end{align*}
  Since Theorem~\ref{thm:rposdef2} implies uniform positive definiteness of $\mathcal{S} \mathbf{C}$  if $\|\mathbf{C}(\boldsymbol{\eps}(\mathbf{u}_{n-1})-\boldsymbol{\eps}^p_{n-2})-\boldsymbol{\sigma}_{n-2}\|_{\infty}$ is sufficiently small, the result follows.
  \end{proof}
  
  \begin{proof}[Proof of Theorem~\ref{thm:solfrac}]
  Following the proof of Theorem~\ref{thm:solexp} but replacing Lemmas~\ref{lem:lip}~\&~\ref{lem:partgrad}--\ref{lem:posdef} with Lemmas~\ref{lem:lipfrac}~\&~\ref{lem:partgradfrac}--\ref{lem:posdeffrac}, we conclude the proof. 
  \end{proof}
  \subsection{Generalizations}\label{sec:gen}
  In this section we want to generalize Theorems~\ref{thm:solexp} and~\ref{thm:solfrac} to the case of infinite dimensional inputs, which would arise if we do not use the space discretization introduced above, i.e. piecewise affine displacements and piecewise constant stresses. We start with the non-fractional setting of Theorem~\ref{thm:solexp} and redefine $T$ from~\eqref{eq:defTnonfrac} as
  \begin{align*}
      T: V_h \times L^\infty(\Omega)^{d \times d}_{\text{sym}} \to (V_h)^*
  \end{align*}
  to prove the following theorem. Note that $V_h$ can be any finite-dimensional subspace of bounded $H^1_0(\Omega)^d$ functions.
  \begin{theorem}
  \label{thm:solexpinf}
      Given, $\kappa \geq \frac{2 \mu}{d}$, $\mathbf{b}_n \in L^1(\Omega)^d, \mathbf{t}_n \in L^1(\Gamma_D)^d$, $\boldsymbol{\eps}^p_{n-1} \in L^\infty(\Omega)^{d \times d}_{\text{sym}}, \boldsymbol{\chi}^1_{n-1} \in L^{\infty}(\Omega)^{d\times d}_{\text{sym}}$, and $\chi^2_{n-1} \in L^\infty(\Omega)$ such that $\chi^2_{n-1}\leq 0$ almost everywhere, let $\hat{\mathbf{u}}_n \in V_h$ satisfy $T^i(\hat{\mathbf{u}}_n)=0.$ Then there exists $\eps>0$ and a function $\mathbf{u}:B^\infty_\eps(\hat{\boldsymbol{\sigma}}_{n-1}) \to V_h$ such that 
      \begin{align*}
          T(\mathbf{u}(\boldsymbol{\sigma}_{n-1}),\boldsymbol{\sigma}_{n-1})=0, \ \forall \boldsymbol{\sigma}_{n-1} \in B^\infty_\eps(\hat{\boldsymbol{\sigma}}_{n-1}),
      \end{align*}
      where $\hat{\boldsymbol{\sigma}}_{n-1}=\mathbf{C}(\boldsymbol{\eps}(\hat{\mathbf{u}}_n)-\boldsymbol{\eps}^p_{n-1})$ and $B^\infty_\eps$ is the Ball with radius $\eps$ in $L^\infty(\Omega)^{d \times d}_{\text{sym}}$ with respect to the $L^\infty$-norm.
  \end{theorem}
  For the proof we invoke \cite[Thm. 4]{Kruse}. Let us state each step in a separate lemma.
  \begin{lemma}
      $T(\hat{\mathbf{u}}_n, \cdot)$ is continuous at $\hat{\boldsymbol{\sigma}}_{n-1}$ and $T(\cdot, \hat{\boldsymbol{\sigma}}_{n-1})$ is Lipschitz continuous around $\hat{\mathbf{u}}_n.$
  \end{lemma}
  \begin{proof}
      Let $\boldsymbol{\sigma}_{n-1}^k$ be a sequence in $B^\infty_\eps(\hat{\boldsymbol{\sigma}}_{n-1})$ such that $\|\boldsymbol{\sigma}_{n-1}^k-\hat{\boldsymbol{\sigma}}_{n-1}\|_{\infty}\to 0$ as $k \to \infty$. We have then
      \begin{align*}
          |T(\hat{\mathbf{u}}_n, &\boldsymbol{\sigma}_{n-1}^k)_i-T(\hat{\mathbf{u}}_n, \hat{\boldsymbol{\sigma}}_{n-1})_i| \\
          &\leq \int_{\Omega}|\Big(R_n^{\boldsymbol{\sigma}_{n-1}^k}\Big(\mathbf{C}(\boldsymbol{\eps}(\hat{\mathbf{u}}_n)- \boldsymbol{\eps}^p_{n-1})\Big)-R_n^{\hat{\boldsymbol{\sigma}}_{n-1}}\Big(\mathbf{C}(\boldsymbol{\eps}(\hat{\mathbf{u}}_n)- \boldsymbol{\eps}^p_{n-1})\Big)\Big):\boldsymbol{\eps}(\boldsymbol{\phi}_i)|dx  \\
          &\leq\int_{\Omega}K|\boldsymbol{\sigma}_{n-1}^k-\hat{\boldsymbol{\sigma}}_{n-1}| dx \to 0\quad\text{as }k\to \infty,
      \end{align*}
      where the last estimate uses Lemma~\ref{lem:lip}. The  Lipschitz continuity also follows from Lemma~\ref{lem:lip}.
  \end{proof}
  For the following considerations, define 
  \begin{align*}
      \mathcal{A}=\left\{\begin{pmatrix}
        \int_{\Omega} \mathcal{S}\mathbf{C}\boldsymbol{\eps}(\boldsymbol{\phi}_1):\boldsymbol{\eps}(\boldsymbol{\phi}_1)dx&
        \cdots & \int_{\Omega} \mathcal{S}\mathbf{C}\boldsymbol{\eps}(\boldsymbol{\phi}_k):\boldsymbol{\eps}(\boldsymbol{\phi}_1)dx\\
        \vdots & \ddots & \vdots\\
              \int_{\Omega} \mathcal{S}\mathbf{C}\boldsymbol{\eps}(\boldsymbol{\phi}_1):\boldsymbol{\eps}(\boldsymbol{\phi}_k)dx & \cdots &\int_{\Omega} \mathcal{S}\mathbf{C}\boldsymbol{\eps}(\boldsymbol{\phi}_k):\boldsymbol{\eps}(\boldsymbol{\phi}_k)dx
      \end{pmatrix}\right\}\subseteq \mathbb{R}^{k\times k},
  \end{align*}
  where $\mathcal{S}$ are all selections of $\sD_{\boldsymbol{\sigma}} R_n^{\hat{\boldsymbol{\sigma}}_{n-1}}(\hat{\boldsymbol{\sigma}}_{n-1})$, such that $\mathcal{S}\mathbf{C}\boldsymbol{\eps}(\boldsymbol{\phi}_i):\boldsymbol{\eps}(\boldsymbol{\phi}_j)$ is integrable for every $1 \leq i,j  \leq k$. Lemma~\ref{lem:posdef} gives the invertibility of each element in $\mathcal{A}$ and below we will derive geometric and topological properties of $\mathcal{A}$.
  To that end, we require the following result.
  \begin{lemma}
      Let $\eps>0$, $\Tilde{\mathbf{y}} \in \mathbb{R}^n$, and $\mathbf{f}: \Omega \times B_\eps(\Tilde{\mathbf{y}}) \subset \mathbb{R}^n \to \mathbb{R}^m$ such that $\mathbf{f}(\cdot,\mathbf{y})$ is integrable for every fixed $\mathbf{y} \in B_\eps(\Tilde{\mathbf{y}})$. Furthermore, suppose that $|\mathbf{f}(\mathbf{x},\mathbf{y}_1)-\mathbf{f}(\mathbf{x},\mathbf{y}_2)| \leq l(\mathbf{x}) |\mathbf{y}_1-\mathbf{y}_2|$, for an integrable function $l$ and every $\mathbf{y}_1,\mathbf{y}_2 \in B_\eps(\Tilde{\mathbf{y}})$. Then we have that
      \begin{align*}
          \int_{\Omega}\sD_\mathbf{y} \mathbf{f}(\mathbf{x},\mathbf{y}) dx
      \end{align*}
      is well-defined as the set of integrals of integrable selections of $\sD_\mathbf{y} \mathbf{f}(\cdot,\mathbf{y})$ for every $\mathbf{y} \in B_\eps(\Tilde{\mathbf{y}})$ and compact and convex in $\mathbb{R}^{m \times n}$.
  \end{lemma}
  \begin{proof}
      Because of the Lipschitz condition, $\sD_\mathbf{y} \mathbf{f}(\mathbf{x},\mathbf{y})$ is well-defined and convex and compact for all $\mathbf{x} \in \Omega$ and $\mathbf{y} \in B_\eps (\Tilde{\mathbf{y}})$. From \cite[Thm. 1]{imbert:hal-00176517}, we know the exact form of its support function, i.e.
      \begin{align}
      \label{eq:supp}
      \begin{split}
           h(\mathbf{x},\mathbf{y},\mathbf{M})&:=\max\{\mathbf{A}:\mathbf{M} , \mathbf{A} \in \partial_\mathbf{y} \mathbf{f}(\mathbf{x},\mathbf{y})\}\\
          &=\limsup_{\mathbf{y}'\to \mathbf{y}, \delta \searrow 0} \frac{1}{\delta^n}\int_{\partial P_\delta (\mathbf{y}')} \mathbf{f}(\mathbf{x},\mathbf{s}):\Big( \mathbf{M \mathbf{n}(\mathbf{s})}\Big) d \mathbf{s},
      \end{split}
      \end{align}
      where $P_\delta(\mathbf{y})$ is the hypercube in $\mathbb{R}^n$, starting at $\mathbf{y}$ with edge length $\delta$, $\mathbf{M} \in \mathbb{R}^{m \times n}$ and $\mathbf{n}(\cdot)$ denotes the outward unit normal. Let us assert the measurability of $h(\cdot,\mathbf{y},\mathbf{M})$. First, note that $\mathbf{f}(\cdot.\cdot)$ is a Caratheodory-function as it is measurable in $\mathbf{x}$ and continuous in $\mathbf{y}$, which by \cite[Lemma 4.51]{aliprantis} ensures joint measurability. Clearly, this implies joint measurability of $\mathbf{f}(\cdot,\cdot):\Big( \mathbf{M \mathbf{n}(\cdot)}\Big)$. Following~\cite[Section~2.2]{imbert:hal-00176517}, the integral in \eqref{eq:supp} can be rewritten  as
      \begin{align}
      \label{eq:rewritesupp}
          \frac{1}{\eps^n}\int_{\partial P_\delta (\mathbf{y}')} \mathbf{f}(\mathbf{x},\mathbf{s}):\Big( \mathbf{M \mathbf{n}(\mathbf{s})}\Big) d \mathbf{s}=\sum_{i=1}^n \int_{[0,1]^{n-1}}\frac{\Big(\mathbf{f}(\mathbf{x},\mathbf{y}'+\delta \mathbf{s}+\delta \mathbf{e}_i)-\mathbf{f}(\mathbf{x},\mathbf{y}'+\delta \mathbf{s}\Big):\Big(\mathbf{M}\mathbf{e}_i\Big)}{\delta} d \mathbf{s},
      \end{align}
      where $\mathbf{e}_i$ denotes the $i$-th unit vector in $\mathbb{R}^n$. Because of the Lipschitz condition, \eqref{eq:rewritesupp} can be bounded by an integrable function over $\Omega$. The Fubini-Tonelli theorem shows measurability of $\int_{\partial P_\delta (\mathbf{y}')} \mathbf{f}(\cdot,\mathbf{s}):\Big( \mathbf{M \mathbf{n}(\mathbf{s})}\Big) d \mathbf{s}$ for all $\mathbf{y}', \delta>0$ such that $P_\delta(\mathbf{y}') \subset B_\eps(\Tilde{\mathbf{y}})$ and $\mathbf{M} \in \mathbb{R}^{m \times n}$.
      The Lipschitz condition and dominated convergence show that~\eqref{eq:rewritesupp} is continuous in $\mathbf{y}'$. Moreover, it is also continuous in $\delta>0$. Thus, we can restrict $\boldsymbol{y}'$ and $\delta$ in~\eqref{eq:supp} to rational values. 
      Therefore, $h(\cdot, \mathbf{y},\mathbf{M})$ is the pointwise $\lim\sup$ of measurable functions, and thus measurable itself.
       By \eqref{eq:rewritesupp}, $h$ is integrable as well and thus \cite[Corollary 18.37]{aliprantis}  shows that the Gelfand integral of $\sD_\mathbf{y} \mathbf{f}(\mathbf{x},\mathbf{y})$ is compact, convex, and non-empty, which coincides with our integral definition in finite-dimensional spaces.
  \end{proof}
  Since $\mathcal{A}$ is the integral of a generalized Jacobian this immediately implies the next result.
  \begin{lemma}
  \label{lem:compA}
      $\mathcal{A}$ is convex and compact in $\mathbb{R}^{k \times k}$.
  \end{lemma}
  Note that positive definiteness (Lemma~\ref{lem:scposdef}) together with Lemma~\ref{lem:compA}, also yields compactness of the set $\{A^{-1}: A \in \mathcal{A}\}$.
The notion of subdifferential is generalized to the setting of uniform strict prederivatives (see~\cite[Definition 13]{Kruse} for details) in the following. 
\begin{lemma}
\label{lem:predev}
    The set $\mathcal{A}$ is a \textit{uniform strict prederivative for T at $\hat{\mathbf{u}}_n$ near $\hat{\boldsymbol{\sigma}}_{n-1}$}, i.e., for every $\eps>0$ there exists a neighborhood of $(\hat{\mathbf{c}},\hat{\boldsymbol{\sigma}}_{n-1})$ ($\hat{\mathbf{c}}\in \mathbb{R}^k$ is the vector corresponding to the basis expansion of $\hat{\mathbf{u}}_n$), such that for elements $(\mathbf{c}_1,\boldsymbol{\sigma}_{n-1}),(\mathbf{c}_2,\boldsymbol{\sigma}_{n-1})$ in this neighborhood it holds
    \begin{align*}
        \inf_{A \in \mathcal{A}} \Big|T(\mathbf{c}_1,\boldsymbol{\sigma}_{n-1})-T(\mathbf{c}_2,\boldsymbol{\sigma}_{n-1})-A(\mathbf{c}_1-\mathbf{c}_2)\Big| \leq \eps |\mathbf{c}_1-\mathbf{c}_2|.
    \end{align*}
\end{lemma}
For the proof we need another result that requires the set-valued mapping
\begin{align*}
\begin{split}
    \widetilde{\mathcal{A}}&: \mathbb{R}^k \times L^\infty(\Omega)^{d \times d}_{\text{sym}} \to \{S\,:\, S\subseteq \mathbb{R}^{k \times k}\}\\
    \widetilde{\mathcal{A}}(\mathbf{c},\boldsymbol{\sigma}_{n-1})&= \left\{\begin{pmatrix}
			\int_{\Omega} \mathcal{S}\mathbf{C}\boldsymbol{\eps}(\boldsymbol{\phi}_1):\boldsymbol{\eps}(\boldsymbol{\phi}_1)dx&
			\cdots & \int_{\Omega} \mathcal{S}\mathbf{C}\boldsymbol{\eps}(\boldsymbol{\phi}_k):\boldsymbol{\eps}(\boldsymbol{\phi}_1)dx\\
			\vdots & \ddots & \vdots\\
            \int_{\Omega} \mathcal{S}\mathbf{C}\boldsymbol{\eps}(\boldsymbol{\phi}_1):\boldsymbol{\eps}(\boldsymbol{\phi}_k)dx & \cdots &\int_{\Omega} \mathcal{S}\mathbf{C}\boldsymbol{\eps}(\boldsymbol{\phi}_k):\boldsymbol{\eps}(\boldsymbol{\phi}_k)dx
		\end{pmatrix}\right\},
\end{split}
\end{align*}
where $\mathcal{S}$ are selections of $\sD_{\boldsymbol{\sigma}} R_n^{\boldsymbol{\sigma}_{n-1}}(\mathbf{C}(\boldsymbol{\eps}(c_i \boldsymbol{\phi}_i)-\boldsymbol{\eps}^p_{n-1}))$, such that $\mathcal{S}\mathbf{C}\boldsymbol{\eps}(\boldsymbol{\phi}_i):\boldsymbol{\eps}(\boldsymbol{\phi}_j)$ is integrable for every $1 \leq i,j  \leq k$. Note that $\mathcal{A}=\widetilde{\mathcal{A}}(\hat{\mathbf{c}},\hat{\boldsymbol{\sigma}}_{n-1}).$
\begin{lemma}
  \label{lem:pointwiseusc}
      $\sD_{\boldsymbol{\sigma}} R_n^{\boldsymbol{\sigma}_{n-1}}\Big(\mathbf{C}\Big(\boldsymbol{\eps}\Big(\sum_{i=1}^dc_i \boldsymbol{\phi}_i\Big)-\boldsymbol{\eps}^p_{n-1}\Big)\Big)$ is upper-semicontinuous at $(\hat{\mathbf{c}},\hat{\boldsymbol{\sigma}}_{n-1})$, for every $\mathbf{x} \in \Omega.$
  \end{lemma}
  \begin{proof}
      Let $\mathbf{x} \in \Omega$ be fixed and $\eps>0$. Since all elements of $\sD_{\boldsymbol{\sigma}} R_n^{\boldsymbol{\sigma}_{n-1}(\mathbf{x})}\Big(\mathbf{C}\Big(\boldsymbol{\eps}\Big(\sum_{i=1}^dc_i \boldsymbol{\phi}_i(\mathbf{x})\Big)-\boldsymbol{\eps}^p_{n-1}(\mathbf{x})\Big)\Big)$ are uniformly continuous in $\boldsymbol{\sigma}_{n-1}(\mathbf{x})$ for fixed $\mathbf{c}$, we have the existence of $\delta >0$ such that for $\|\boldsymbol{\sigma}_{n-1}-\hat{\boldsymbol{\sigma}}_{n-1}\|_\infty< \delta$ we have that
      \begin{align*}
          \sD_{\boldsymbol{\sigma}} R_n^{\boldsymbol{\sigma}_{n-1}(\mathbf{x})}\Big(\mathbf{C}\Big(\boldsymbol{\eps}\Big(\sum_{i=1}^dc_i \boldsymbol{\phi}_i(\mathbf{x})\Big)-\boldsymbol{\eps}^p_{n-1}(\mathbf{x})\Big)\Big) \subset B_\eps \Bigg(\sD_{\boldsymbol{\sigma}} R_n^{\hat{\boldsymbol{\sigma}}_{n-1}(\mathbf{x})}\Big(\mathbf{C}\Big(\boldsymbol{\eps}\Big(\sum_{i=1}^dc_i \boldsymbol{\phi}_i(\mathbf{x})\Big)-\boldsymbol{\eps}^p_{n-1}(\mathbf{x})\Big)\Big)\Bigg).
      \end{align*}
      By the upper semi-continuity properties of generalized gradients~\cite[Prop. 2.6.2]{clarke}, we have a possibly smaller $\delta_2>0$ such that for $|\mathbf{c}-\hat{\mathbf{c}}|< \delta_2$ it holds
      \begin{align*}
          \sD_{\boldsymbol{\sigma}} R_n^{\hat{\boldsymbol{\sigma}}_{n-1}(\mathbf{x})}\Big(\mathbf{C}\Big(\boldsymbol{\eps}\Big(\sum_{i=1}^dc_i \boldsymbol{\phi}_i(\mathbf{x})\Big)-\boldsymbol{\eps}^p_{n-1}(\mathbf{x})\Big)\Big) \subset B_\eps \Bigg(\sD_{\boldsymbol{\sigma}} R_n^{\hat{\boldsymbol{\sigma}}_{n-1}(\mathbf{x})}\Big(\mathbf{C}\Big(\boldsymbol{\eps}\Big(\sum_{i=1}^d\hat{c}_i \boldsymbol{\phi}_i(\mathbf{x})\Big)-\boldsymbol{\eps}^p_{n-1}(\mathbf{x})\Big)\Big)\Bigg).
      \end{align*}
      Bringing both results together, yields the claim.
  \end{proof}
  \begin{lemma}
    \label{lem:usc}
        $\widetilde{\mathcal{A}}$ is upper semicontinuous at $(\hat{\mathbf{c}},\hat{\boldsymbol{\sigma}}_{n-1})$.
    \end{lemma}
    \begin{proof}
    We consider the set valued mapping $\widetilde{\mathcal{A}}^\circ\colon \mathbb{R}^k\times  L^\infty(\Omega)^{d\times d}_{\rm sym} \times \Omega \to \{S\,:\,S\subseteq \mathbb{R}^{k\times k}\}$ by
    \begin{align*}
    \begin{split}
     \widetilde{\mathcal{A}}^\circ(\mathbf{c}, \boldsymbol{\sigma}_{n-1},\mathbf{x})=  \left\{\begin{pmatrix}
          \mathcal{S}(\mathbf{x})\mathbf{C}\boldsymbol{\eps}(\boldsymbol{\phi}_1(\mathbf{x})):\boldsymbol{\eps}(\boldsymbol{\phi}_1(\mathbf{x}))&
          \cdots &  \mathcal{S}(\mathbf{x})\mathbf{C}\boldsymbol{\eps}(\boldsymbol{\phi}_k(\mathbf{x})):\boldsymbol{\eps}(\boldsymbol{\phi}_1(\mathbf{x}))\\
          \vdots & \ddots & \vdots\\
                \mathcal{S}(\mathbf{x})\mathbf{C}\boldsymbol{\eps}(\boldsymbol{\phi}_1(\mathbf{x})):\boldsymbol{\eps}(\boldsymbol{\phi}_k(\mathbf{x}))& \cdots &\mathcal{S}(\mathbf{x})\mathbf{C}\boldsymbol{\eps}(\boldsymbol{\phi}_k(\mathbf{x})):\boldsymbol{\eps}(\boldsymbol{\phi}_k(\mathbf{x}))
        \end{pmatrix}\right\}
    \end{split}
    \end{align*}
    and note that its range $\mathbb{R}^{k\times k}$ is a separable metric space. Analogously to the proof of Lemma~\ref{lem:compA}, the images of the above map are compact and convex. Thus,~\cite[18.6]{aliprantis}, together with \cite[Thm. 18.31]{aliprantis}, shows that $\mathbf{x}\mapsto \widetilde{\mathcal{A}}^\circ(\mathbf{c},\boldsymbol{\sigma}_{n-1})$ has a measurable graph for every $(\mathbf{c},\boldsymbol{\sigma}_{n-1})$ (here we used that the images of $\widetilde{\mathcal{A}}^\circ$ are compact and thus, particularly, closed).
    By the chain rule for subdifferentials, the upper semi-continuity of $\widetilde{\mathcal{A}}^\circ$ follows from the upper semi-continuity of  $\sD_{\boldsymbol{\sigma}} R_n^{\boldsymbol{\sigma}_{n-1}}\Big(\mathbf{C}\Big(\boldsymbol{\eps}\Big(\sum_{i=1}^dc_i \boldsymbol{\phi}_i\Big)-\boldsymbol{\eps}^p_{n-1}\Big)\Big)$ and hence Lemma~\ref{lem:pointwiseusc}.
    The existence of an integrable upper bound follows from the integrable boundedness of the components.
    The finite-dimensionality (and thus separability and equivalence of strong and weak topology) of the codomain of $\widetilde{\mathcal{A}}^\circ$ allows us to apply~\cite[Thm. 3.1]{YANNELIS}, which directly yields the result.
    \end{proof}
    \begin{proof}[Proof of Lemma~\ref{lem:predev}]
      We can invoke the mean-value theorem of \cite[Prop. 2.6.5]{clarke} for generalized Jacobian: There exists $t \in [0,1]$ such that
      \begin{align*}
          T(\mathbf{c}_1,\boldsymbol{\sigma}_{n-1})-T(\mathbf{c}_2,\boldsymbol{\sigma}_{n-1}) \in \text{co}\{A(\mathbf{c}_1-\mathbf{c}_2): A \in \sD_\mathbf{c} T(\mathbf{c}_1+t(\mathbf{c}_2-\mathbf{c}_1),\boldsymbol{\sigma}_{n-1})\}.
      \end{align*}
      In Lemma~\ref{lem:posdef} it is shown $\sD_\mathbf{c} T(\mathbf{c}_1+t(\mathbf{c}_2-\mathbf{c}_1),\boldsymbol{\sigma}_{n-1}) \mathbf{b} \subset \widetilde{\mathcal{A}}(\mathbf{c}_1+t(\mathbf{c}_2-\mathbf{c}_1),\boldsymbol{\sigma}_{n-1})\mathbf{b}$, if $\mathbf{c}_1, \mathbf{c}_2, \boldsymbol{\sigma}_{n-1}$ are in a sufficiently small neighborhood of $(\hat{\mathbf{c}},\hat{\boldsymbol{\sigma}}_{n-1})$ to ensure well-definedness of $T$ and its generalized gradients. Therefore, we have that for a finite convex combination we can write
      \begin{align*}
          T(\mathbf{c}_1,\boldsymbol{\sigma}_{n-1})-T(\mathbf{c}_2,\boldsymbol{\sigma}_{n-1})=\sum_i \lambda_i A_i(\mathbf{c}_1-\mathbf{c}_2),
      \end{align*} where $A_i \in \widetilde{\mathcal{A}}(\mathbf{c}_1+t(\mathbf{c}_2-\mathbf{c}_1),\boldsymbol{\sigma}_{n-1})$. For $\eps>0$, we can now choose a neighborhood of $(\hat{\mathbf{c}},\hat{\boldsymbol{\sigma}}_{n-1})$ by Lemma~\ref{lem:usc}, such that for every $ A \in \widetilde{\mathcal{A}}(\mathbf{c},\boldsymbol{\sigma}_{n-1})$ in the neighborhood we have $ \overline{A} \in \mathcal{A}$ such that
      \begin{align*}
          |A-\overline{A}|<\eps.
      \end{align*}
      This shows for points $(\mathbf{c}_1,\boldsymbol{\sigma}_{n-1}), (\mathbf{c}_2,\boldsymbol{\sigma}_{n-1})$ in the neighborhood that
      \begin{align*}
          |T(\mathbf{c}_1,\boldsymbol{\sigma}_{n-1})-T(\mathbf{c}_2,\boldsymbol{\sigma}_{n-1})-\sum_i \lambda_i \overline{A}(\mathbf{c}_1-\mathbf{c}_2)|< \eps |\mathbf{c}_1-\mathbf{c}_2|
      \end{align*}
      and concludes the proof.
  \end{proof}
  \begin{proof}[Proof of Theorem~\ref{thm:solexpinf}]
    With the results above, we can directly apply~\cite[Thm. 4]{Kruse} to get the result.
    \end{proof}
    The same program can be applied in order to generalize Theorem~\ref{thm:solfrac}, i.e we redefine $T$ as
    \begin{align*}
        T:V_h \times L^\infty(\Omega)^{d \times d}_{\text{sym}} \times L^\infty(\Omega)^{d \times d}_{\text{sym}} \times L^\infty(\Omega)^{d \times d}_{\text{sym}} \times L^\infty(\Omega) \times L^1(\Omega)^d \times L^1(\Gamma_D)^d \to (V_h)^*
    \end{align*}
    and aim to prove
    \begin{theorem}
    Let $(\mathbf{u}_{n-1},\boldsymbol{\eps}^p_{n-2},\boldsymbol{\sigma}_{n-2},\boldsymbol{\chi}_{n-2},\mathbf{b}_{n-1},\mathbf{t}_{n-1}) \in V_h \times L^\infty(\Omega)^{d \times d}_{\text{sym}} \times L^\infty(\Omega)^{d \times d}_{\text{sym}} \times L^\infty(\Omega)^{d \times d}_{\text{sym}} \times L^\infty(\Omega) \times L^1(\Omega)^d \times L^1(\Gamma_D)^d $ such that $T(\mathbf{u}_{n-1},\boldsymbol{\eps}^p_{n-2},\boldsymbol{\sigma}_{n-2},\boldsymbol{\chi}_{n-2},\mathbf{b}_{n-1},\mathbf{t}_{n-1})=0$ and $\|\mathbf{C}(\boldsymbol{\eps}(\mathbf{u}_{n-1})-\boldsymbol{\eps}^p_{n-2})-\boldsymbol{\sigma}_{n-2}\|_{\infty}$, $\boldsymbol{\Delta}$ is sufficiently small, $\kappa \geq \frac{2 \mu}{d}$and $\chi^2_{n-2}\leq 0$ almost everywhere. Furthermore, let the assumptions on the parameters $k_1, k_2, \kappa d, \alpha$ from Theorem \ref{thm:rposdef2} hold. Then there exists $\eps>0$,  and a function $\mathbf{u}: B_\eps^\infty(\boldsymbol{\eps}^p_{n-2},\boldsymbol{\sigma}_{n-2},\boldsymbol{\chi}_{n-2},\mathbf{b}_{n-1},\mathbf{t}_{n-1}) \to V_h$ such that
        \begin{align*}
            T(\mathbf{u}(\boldsymbol{\eps}^p,\boldsymbol{\sigma},\boldsymbol{\chi},\mathbf{b},\mathbf{t}),\boldsymbol{\eps}^p,\boldsymbol{\sigma},\boldsymbol{\chi},\mathbf{b},\mathbf{t})=0,
        \end{align*}
        where $B_\eps^\infty(\cdot)$ denotes the Ball of radius $\eps$ in  $L^\infty(\Omega)^{d \times d}_{\text{sym}} \times L^\infty(\Omega)^{d \times d}_{\text{sym}} \times L^\infty(\Omega)^{d \times d}_{\text{sym}} \times L^\infty(\Omega) \times L^1(\Omega)^d \times L^1(\Gamma_D)^d$.  
    \end{theorem}
    \begin{proof}
        We essentially follow the steps of the proof of Theorem~\ref{thm:solexpinf}. The only difference is in the arguments of Lemma~\ref{lem:pointwiseusc}, when we show pointwise upper-semicontinuity of $\sD_{\boldsymbol{\sigma}}R\Big(\mathbf{C}\Big(\boldsymbol{\eps}\Big(\sum_{i=1}^dc_i \boldsymbol{\phi}_i\Big)-\boldsymbol{\eps}^p_{n-1}\Big),\boldsymbol{\sigma}_{n-1},\boldsymbol{\chi}\Big)$  at $(\hat{\mathbf{c}},\boldsymbol{\eps}^p_{n-2},\boldsymbol{\sigma}_{n-2},\boldsymbol{\chi}_{n-2})$, where $\boldsymbol{u}_{n-1}$ is identified with its basis expansion $\hat{\mathbf{c}}$. As before, we fix $\mathbf{x} \in \Omega$ and $\eps>0$. Again, the uniform continuity of the elements of  $\sD_{\boldsymbol{\sigma}}R\Big(\mathbf{C}\Big(\boldsymbol{\eps}\Big(\sum_{i=1}^dc_i \boldsymbol{\phi}_i(\mathbf{x})\Big)-\boldsymbol{\eps}^p_{n-1}(\mathbf{x})\Big),\boldsymbol{\sigma}_{n-1}(\mathbf{x}),\boldsymbol{\chi}(\mathbf{x})\Big)$ in $\boldsymbol{\sigma}_{n-1}(\mathbf{x})$ implies that for $\delta>0$ and $||\boldsymbol{\sigma}_{n-1}-\boldsymbol{\sigma}_{n-2}||_\infty< \delta$ that
        \begin{align*}
            \sD_{\boldsymbol{\sigma}}R\Big(&\mathbf{C}\Big(\boldsymbol{\eps}\Big(\sum_{i=1}^dc_i \boldsymbol{\phi}_i(\mathbf{x})\Big)-\boldsymbol{\eps}^p_{n-1}(\mathbf{x})\Big),\boldsymbol{\sigma}_{n-1}(\mathbf{x}),\boldsymbol{\chi}(\mathbf{x})\Big) \\
            &\subset B_{\eps} \Bigg( \sD_{\boldsymbol{\sigma}}R\Big(\mathbf{C}\Big(\boldsymbol{\eps}\Big(\sum_{i=1}^dc_i \boldsymbol{\phi}_i(\mathbf{x})\Big)-\boldsymbol{\eps}^p_{n-1}(\mathbf{x})\Big),\boldsymbol{\sigma}_{n-2}(\mathbf{x}),\boldsymbol{\chi}(\mathbf{x})\Big)\Bigg).
        \end{align*}
        Differently to the previous proof, we have to distinguish two cases, namely the situation inside and outside the critical region. If $f\Big(\mathbf{C}\Big(\boldsymbol{\eps}\Big(\sum_{i=1}^d\hat{c}_i \boldsymbol{\phi}_i(\mathbf{x})\Big)-\boldsymbol{\eps}^p_{n-2}(\mathbf{x})\Big),\boldsymbol{\chi}_{n-2}(\mathbf{x}) \Big)<0$, it suffices to find $\delta>0$ such that for $||(\mathbf{c}-\hat{\mathbf{c}},\boldsymbol{\eps}^p_{n-1}-\boldsymbol{\eps}^p_{n-2},\boldsymbol{\chi}^1-\boldsymbol{\chi}^1_{n-2}, \chi^2-\chi^2_{n-2})||_\infty< \delta$, we stay inside the critical region. The same arguments work outside the critical region. If we are on the yield surface, i.e. $f\Big(\mathbf{C}\Big(\boldsymbol{\eps}\Big(\sum_{i=1}^d\hat{c}_i \boldsymbol{\phi}_i(\mathbf{x})\Big)-\boldsymbol{\eps}^p_{n-2}(\mathbf{x})\Big),\boldsymbol{\chi}_{n-2}(\mathbf{x}) \Big)= 0$, we  find $\delta>0$ by using the continuity of the elements of $\sD_{\boldsymbol{\sigma}}R$ in the respective variables to get
        \begin{align*}
            \sD_{\boldsymbol{\sigma}}R\Big(&\mathbf{C}\Big(\boldsymbol{\eps}\Big(\sum_{i=1}^dc_i \boldsymbol{\phi}_i(\mathbf{x})\Big)-\boldsymbol{\eps}^p_{n-1}(\mathbf{x})\Big),\boldsymbol{\sigma}_{n-2}(\mathbf{x}),\boldsymbol{\chi}(\mathbf{x})\Big) \\
        \end{align*}
        With this result, the remaining proof steps transfer verbatim.
    \end{proof}
    \section{Semismoothness of fractional implicit return-mapping}\label{sec:semi}
As discussed before, most classical discretizations of flow rules are implicit. This has stability advantages but also ensures exact complementary conditions. In our case, implicitly discretizing both flow rule~\eqref{eq:flow1}--\eqref{eq:flow3} and complementarity conditions~\eqref{eq:complementarity} yields the following equations for the material update ($\boldsymbol{\sigma}_n,\boldsymbol{\chi}_{n},\Delta \gamma_n$), with a given $\boldsymbol{\sigma}_{tr}$
	\begin{align*}
		\boldsymbol{\sigma}_n&=\boldsymbol{\sigma}_{tr}- \Delta \gamma_n \mathbf{C} \Dhat{\boldsymbol{\sigma}_n}{\alpha}f(\boldsymbol{\sigma}_n,\boldsymbol{\chi}_{n})\\
		\boldsymbol{\chi}^1_{n}&=\boldsymbol{\chi}^1_{n-1}-\Delta \gamma_n k_1 \partial_{\boldsymbol{\chi}^1} f(\boldsymbol{\sigma}_n,\boldsymbol{\chi}_{n})\\
		\chi^2_n&=\chi^2_{n-1}-\Delta \gamma_n k_2
	\end{align*}
	subject to the discretized complementarity conditions
	\begin{align}
    \label{eq:implicitcomp}
		\Delta \gamma_n \geq 0 \quad
		f(\boldsymbol{\sigma}_n,\boldsymbol{\chi}_{n}) \leq 0 \quad
		\Delta \gamma_n f(\boldsymbol{\sigma}_n,\boldsymbol{\chi}_{n}) =0.
	\end{align}
	Note that the complementarity conditions \eqref{eq:implicitcomp} can be written via a NCP-function \cite{Sauter} as
	\begin{align*}
		\max\Big\{0,\Delta \gamma_n+f(\boldsymbol{\sigma}_n,\boldsymbol{\chi}_{n})\Big\}- \Delta \gamma_n=0.
	\end{align*}
This motivates the definition of the map
\begin{align*}
\begin{split}
	&T: (M^d)^3 \times \mathbb{R} \times \mathbb{R} \to (M^d)^2 \times \mathbb{R} \times \mathbb{R}, \\
	&T(\boldsymbol{\sigma}_{tr},\boldsymbol{\sigma}_n,\boldsymbol{\chi}_{n},\Delta \gamma_n)=
 \begin{bmatrix}
		\boldsymbol{\sigma}_n -\boldsymbol{\sigma}_{tr}+ \Delta \gamma_n \mathbf{C} \Dhat{\boldsymbol{\sigma}_n}{\alpha}f(\boldsymbol{\sigma}_n,\boldsymbol{\chi}_{n})\\
		\boldsymbol{\chi}^1_{n}-\boldsymbol{\chi}^1_{n-1}+\Delta \gamma_n k_1 \partial_{\boldsymbol{\chi}^1} f(\boldsymbol{\sigma}_n,\boldsymbol{\chi}_{n})\\
		\chi^2_n-\chi^2_{n-1}+\Delta \gamma_n k_2\\
		\max\Big\{0,\Delta \gamma_n+f(\boldsymbol{\sigma}_n,\boldsymbol{\chi}_{n})\Big\}- \Delta \gamma_n
\end{bmatrix}.
\end{split}
\end{align*}
A solution satisfying $T=0$ is a valid material update. We want to show, using the implicit function theorem for semismooth functions, that in a neighborhood of the region enclosed by the yield function $f$ there exists a semismooth return-mapping, $R(\boldsymbol{\sigma}_{tr})=\boldsymbol{\sigma}_n$. While this is a first step towards well-posedness of the implicit scheme, it does not imply that a solution $\mathbf{u}$ of the weak form~\eqref{eq:modelproblem} exists, which is left as an interesting open question.
\begin{theorem}
    Let $\boldsymbol{\chi}^1_{n-1} \in M^d, \chi^2_{n-1}<0$, and $\boldsymbol{\Delta}$ sufficiently small. Then, there exists an open neighborhood $\widetilde{K}$ of  $K=\{\boldsymbol{\sigma} \in M^d: f(\boldsymbol{\sigma},\boldsymbol{\chi}_{n-1}) \leq 0\}$ and a semismooth function
    \begin{align*}
        \mathcal{R}: \widetilde{K} \to (M^d)^2 \times \mathbb{R} \times \mathbb{R}
    \end{align*}
    such that
    \begin{align*}
        T(\boldsymbol{\sigma},\mathcal{R}(\boldsymbol{\sigma}))=0, \quad \text{for all } \boldsymbol{\sigma} \in \widetilde{K}.
    \end{align*}
\end{theorem}
\begin{proof}
  The first step is to show, that $T$ is a semismooth function.
Due to Lemma~\ref{lem:fracdiff}, we have that the first three components of $T$ are continuously differentiable, while the fourth component is semismooth as a concatenation of semismooth functions. Therefore, $T$ is a semismooth function around a point where $\boldsymbol{\sigma}_n+\boldsymbol{\chi}^1_n$ is sufficiently far from 0 or $\Delta \gamma_n=0$, which will be the case for a solution of $T=0$.
Now let $\boldsymbol{\sigma}_{tr} \in K$, which yields that $T(\boldsymbol{\sigma}_{tr},\boldsymbol{\sigma}_{tr},\boldsymbol{\chi}_{n-1},0)=0$. We now want to check the generalized gradient at that point and distinguish two cases:

$\bullet$ $f(\boldsymbol{\sigma}_{tr},\boldsymbol{\chi}_{n-1})<0$:
This implies that $\Delta \gamma_n=0$ and that $T$ is continuously differentiable in a neighborhood of $(\boldsymbol{\sigma}_{tr},\boldsymbol{\sigma}_{tr},\boldsymbol{\chi}_{n-1},0)$. Therefore, we have by classical differentiation and using that $\Delta \gamma_n=0$
\begin{align}
    \label{eq:dt1}
\begin{split}
   \sD T(\boldsymbol{\sigma}_{tr},\boldsymbol{\sigma}_{tr},\boldsymbol{\chi}_{n-1},0)= 
   \begin{bmatrix}
    -\Id & \Id & 0 & 0 & \mathbf{C}\Dhat{\boldsymbol{\sigma}_{tr}}{\alpha}f(\boldsymbol{\sigma}_{tr},\boldsymbol{\chi}_{n-1}) \\
    0& 0 & \Id & 0 & k_1 \partial_{\boldsymbol{\chi}^1} f(\boldsymbol{\sigma}_{tr},\boldsymbol{\chi}_{n-1})\\
    0 & 0 & 0 & 1 & k_2 \\
    0 & 0 & 0 &0 & -1
  \end{bmatrix}.
\end{split}
\end{align}
It is necessary to show that $\partial T$ without the first column is invertible. This follows directly from the upper triangular structure.

$\bullet$ $f(\boldsymbol{\sigma}_{tr},\boldsymbol{\chi}_{n-1})=0$: Now we have to calculate the subgradient elements of $\sD T$ because of the lack of differentiability. It consists of convex combinations of \eqref{eq:dt1} and
\begin{align}
\label{eq:dt2}
\begin{split}
\begin{bmatrix}
-\Id & \Id& 0 & 0 & \mathbf{C}\Dhat{\boldsymbol{\sigma}_{tr}}{\alpha}f^{tr} \\
0& 0& \Id& 0 & k_1 \partial_{\boldsymbol{\chi}^1} f^{tr}\\
0 & 0 & 0 & 1 & k_2 \\
0 & \partial_{\boldsymbol{\sigma}} f^{tr} & \partial_{\boldsymbol{\chi}^1} f^{tr} &1 & 0
\end{bmatrix},
\end{split}
\end{align}
where we use the shortcut $f^{tr}:=f(\boldsymbol{\sigma}_{tr},\boldsymbol{\chi}_{n-1})$.
It remains to investigate convex combinations of \eqref{eq:dt1} and \eqref{eq:dt2}. With $\lambda \in (0,1)$, we get for the four trailing columns
\begin{align}\label{eq:lambdamat}
	\begin{bmatrix}
		\Id & 0 & 0 & \mathbf{C}\Dhat{\boldsymbol{\sigma}_{tr}}{\alpha}f^{tr} \\
		0& \Id& 0 & k_1 \partial_{\boldsymbol{\chi}^1} f^{tr}\\
		0 & 0 & 1 & k_2 \\
		\lambda \partial_{\boldsymbol{\sigma}} f^{tr} & \lambda \partial_{\boldsymbol{\chi}^1} f^{tr} & \lambda  & \lambda-1
	\end{bmatrix}.
\end{align}
Using the first three rows to eliminate the first three columns of the fourth row, the remaining entry in the fourth row is given by
\begin{align*}
    \lambda-1  -\lambda \mathbf{C}\Dhat{\boldsymbol{\sigma}_{tr}}{\alpha}f^{tr}: \partial_{\boldsymbol{\sigma}} f^{tr} -  k_1 \lambda \partial_{\boldsymbol{\chi}^1} f^{tr}:  \partial_{\boldsymbol{\chi}^1} f^{tr} -\lambda k_2.
\end{align*}
Clearly, invertibility of~\eqref{eq:lambdamat} is equivalent to this term being non-zero.
Since $| \partial_{\boldsymbol{\chi}^1} f^{tr}|=1$ and $\mathbf{C}\Dhat{\boldsymbol{\sigma}_{tr}}{\alpha}f^{tr}: \partial_{\boldsymbol{\sigma}} f^{tr}\geq 0$, due to the symmetry of $\mathbf{C}$, this follows if
\begin{align*}
     \lambda -1 -\lambda(k_1+k_2)<0,
\end{align*}
which is always satisfied for non-zero hardening variables $k_1, k_2$. 
The implicit function theorem for semismooth functions, i.e. \cite[Proposition 10.3]{Sauter} yields the existence of a semismooth return-mapping $\mathcal{R}(\boldsymbol{\sigma}_{tr})$ in a neighborhood of the feasible set $f(\boldsymbol{\sigma}_{tr},\boldsymbol{\chi}_{n-1})\leq0.$
\end{proof}
\section{Numerical Experiments}\label{sec:numerics}
To highlight the results of the previous sections, we conduct several numerical experiments. Different values for the fractional coefficient $\alpha$, for the interval matrix $\boldsymbol{\Delta}$ and different domains with corresponding loading regimes shall be considered and their influences examined. 
\subsection{Two-dimensional domain}
The domain $\Omega \subset \mathbb{R}^2$ is depicted in Figures~\ref{fig:domain}(A) and~\ref{fig:domain}(B).
\begin{figure}
    \centering
    \begin{subfigure}{0.45 \textwidth}
    \centering
    \begin{tikzpicture}[scale=0.55]
    
        \draw[black,thick,fill=gray!30] (0,0) -- (0,2)--(3,2)--(4,1.5)--(6,1.5)--(7,2)--(10,2)--(10,0)--(7,0)--(6,0.5)--(4,0.5)--(3,0)--cycle;
		\draw (0,0)-- +(-0.5,0.2);
		\draw (0,0.2)-- +(-0.5,0.2);
		\draw (0,0.4)-- +(-0.5,0.2);
		\draw (0,0.6)-- +(-0.5,0.2);
		\draw (0,0.8)-- +(-0.5,0.2);
		\draw (0,1)-- +(-0.5,0.2);
		\draw (0,1.2)-- +(-0.5,0.2);
		\draw (0,1.4)-- +(-0.5,0.2);
		\draw (0,1.6)-- +(-0.5,0.2);
		\draw (0,1.8)-- +(-0.5,0.2);
		\draw (0,2)-- +(-0.5,0.2);
		\draw[->,red, thick] (10,0) -- +(0.5,0);
		\draw[->,red, thick] (10,0.2) -- +(0.5,0);
		\draw[->,red, thick] (10,0.4) -- +(0.5,0);
		\draw[->,red, thick] (10,0.6) -- +(0.5,0);
		\draw[->,red, thick] (10,0.8) -- +(0.5,0);
		\draw[->,red, thick] (10,1) -- +(0.5,0);
		\draw[->,red, thick] (10,1.2) -- +(0.5,0);
		\draw[->,red, thick] (10,1.4) -- +(0.5,0);
		\draw[->,red, thick] (10,1.6) -- +(0.5,0);
		\draw[->,red, thick] (10,1.8) -- +(0.5,0);
		\draw[->,red, thick] (10,2) -- +(0.5,0);
    
	\end{tikzpicture}
    \caption{}
    \label{fig:domain1}
\end{subfigure} 
\begin{subfigure}{0.45\textwidth}
	\centering
	\begin{tikzpicture}[scale=0.55]
			\draw[white,fill=gray!30] (0,0) -- (0,2)--(3,2)--(4,1.5)--(6,1.5)--(7,2)--(10,2)--(10,0)--(7,0)--(6,0.5)--(4,0.5)--(3,0)--cycle;
			\node at (5,1) {$\Omega$};
			\draw[thick, blue] (0,0) -- (0,2);
			\node[fill, red, minimum size=0.8pt, inner sep=0] at (0,2) {};
			\node[fill, red, minimum size=0.8pt, inner sep=0] at (0,0) {};
			\node at (-0.5,1) {\color{blue}$\Gamma_D$};
			\draw[thick, red] (0,2)--(3,2)--(4,1.5)--(6,1.5)--(7,2)--(10,2)--(10,0)--(7,0)--(6,0.5)--(4,0.5)--(3,0)--(0,0);
			\node at (4,2) {\color{red}$\Gamma_N$};
			\node[minimum size=0,inner sep=0] (A) at (10.3,0) {};
			\node[minimum size=0,inner sep=0] (B) at (10.3,2) {};
			\draw[dashed] (A)--(B);
			\draw (10.2,0)-- (10.4,0);
			\draw (10.2,2)-- (10.4,2);
			\node[right] at (10.4,1) {$\mathbf{t}_N \neq 0$};
	\end{tikzpicture}
	\caption{}
 \label{fig:domain2}
\end{subfigure}
\vspace{2em}
\begin{subfigure}{ \textwidth}
    \centering
    	\begin{tikzpicture}[scale=0.7]
			\draw[black,thick,fill=gray!10] (0,0) -- (0,2)--(3,2)--(4,1.5)--(6,1.5)--(7,2)--(10,2)--(10,0)--(7,0)--(6,0.5)--(4,0.5)--(3,0)--cycle;
			\node[fill=red,circle, minimum size=2mm,inner sep=0] (A) at (5,0.5) {};
			\draw[red,->] (A) -- +(0,0.5);
			\node at (5.4,0.75) {\color{red}$d_y$};
            \draw[red,->] (10,1) -- +(0.5,0);
			\node at (10.4,1.25) {\color{red}$d_x$};
            \node[fill=red,circle, minimum size=2mm,inner sep=0] at (10,1) {};
	\end{tikzpicture}
    \caption{}
    \label{fig:dispmeasure}
\end{subfigure}
    \caption{(A): The domain $\Omega \subset \mathbb{R}^2$, is anchored on the left and subject to a pulling force on the right. (B): Depiction of $\Omega$ and $\partial \Omega= \Gamma_D \cup \Gamma_N$. (C): Depiction of locations where displacement measurements $d_x$ and $d_y$ are taken.}
    \label{fig:domain}
\end{figure}
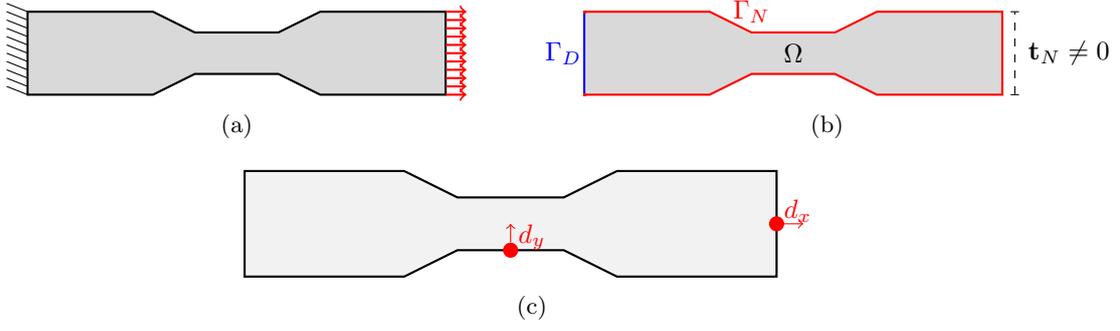
We have $\mathbf{b}=0$ and $\mathbf{t}_N$ vanishing everywhere but on the right edge, as seen in Figure~\ref{fig:domain}(B), where it is constant in space but varying in time as shown in Figure~\ref{fig:load}.

\begin{minipage}{0.45 \linewidth}
  \centering
  \begin{tikzpicture}[scale=0.8]
  \draw[->,thick] (-1,0) --(7,0);
  \draw[->,thick] (0,-1) -- (0,4);
  \draw[red,thick] (0,0) -- (3,3);
  \draw[red,thick] (3,3)--(6,0);
  \node at(7.4,0) {$t$};
  \node at (0,4.4) {$t^N_1$};
  \draw (-0.1,3)--(0.1,3);
  \node[left] at (-0.2,3) {$t^{max}_N$};
  \node[anchor=north east] at (0,0) {$0$};
\end{tikzpicture}
  
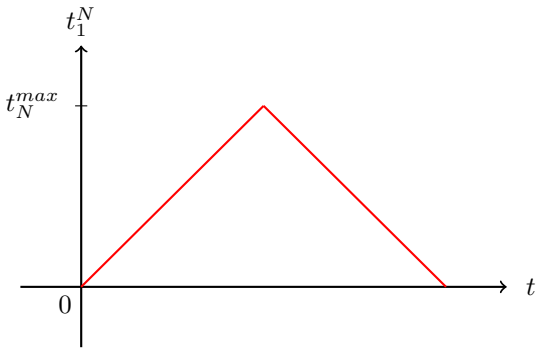
\captionof{figure}{\hbox{Loading $\mathbf{t}_N=(t_1^N,0)$ over time.}}
  \label{fig:load}
\end{minipage}
\begin{minipage}{0.45 \textwidth}
\centering
\begin{tabular}{c |c}
  
  Material parameters & Value \\
  \hline 
  $\mu$ & $55000$\\
  $\kappa$ & $55000$ \\
  $Y_0$ & $10000$ \\
  $k_1$ & $110000$ \\
  $k_2$ & $110000$\\
  $t_N^{max}$ & $15000$\\
          $\boldsymbol{\Delta}$&$\begin{bmatrix}
              100 & 100\\
              100 & 200
          \end{bmatrix}$
\end{tabular}
\bigskip
\bigskip
\bigskip

 \captionof{table}{\hbox{Default material parameters.}}
 \label{tab:paras}
\end{minipage}
\medskip

For most experiments we consider the fixed material parameters in Table~\ref{tab:paras}, chosen similarly to the ones used in the numerical experiments in \cite{Sauter} and adapted to fit the given constraints in our setting.
Figure~\ref{fig:flow} compares the plastic flow-vectors for $\alpha=0.5$ as defined in \eqref{eq:flow1} with the flow-vectors for $\alpha \approx 1$ on the yield-surface $f=0$, and we observe a clear deviation from the non-fractional case.
\begin{figure}
    \centering
    \begin{subfigure}{0.45\textwidth}
        \centering
        \includegraphics[width=\linewidth]{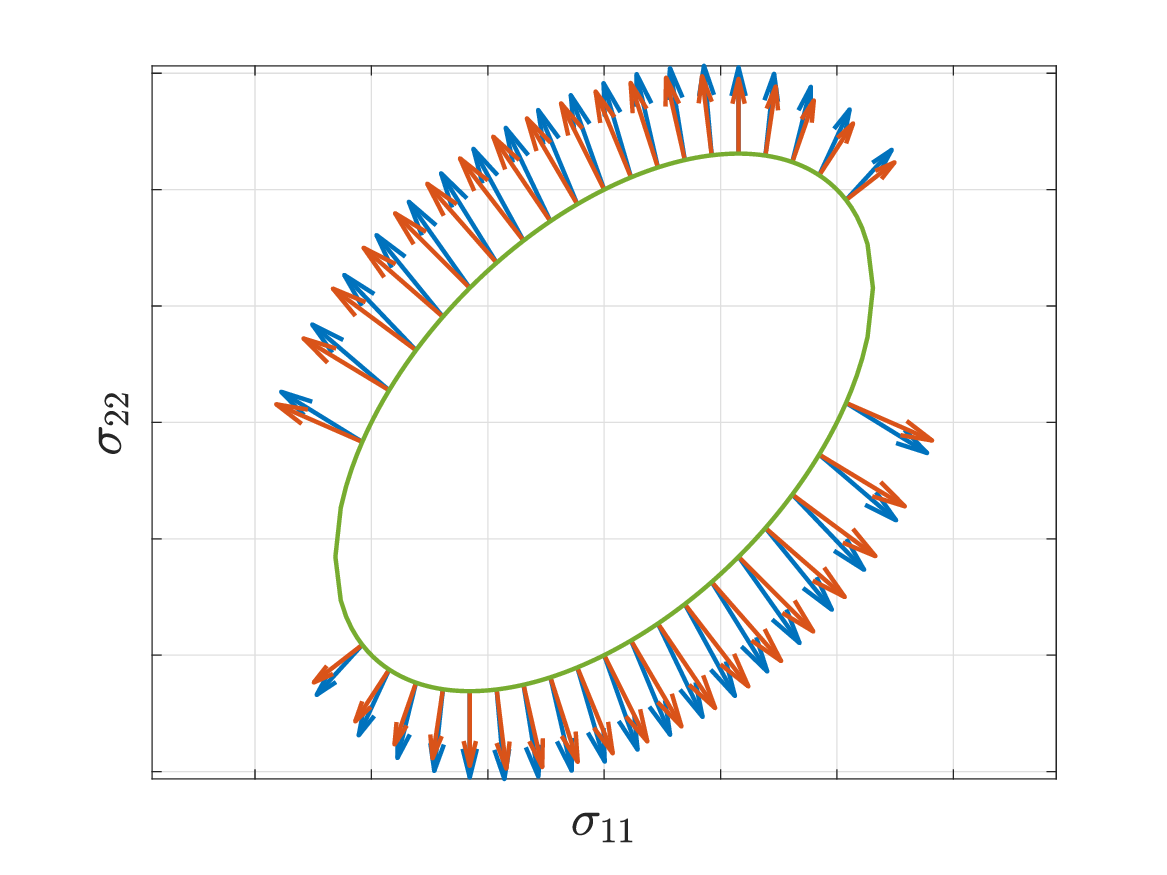}
        \caption{We depict the flow vectors for \textcolor{color1}{$\alpha=0.5$}, \textcolor{color2}{$\alpha=1$} and the \textcolor{color3}{yield surface}. $\boldsymbol{\Delta}=\begin{pmatrix}
            100 & 100 & 100\\
            100 & 200 & 100 \\
            100 & 100 & 100
        \end{pmatrix}$}
    \end{subfigure}
    \begin{subfigure}{0.45\textwidth}
        \centering
        \includegraphics[width=\linewidth]{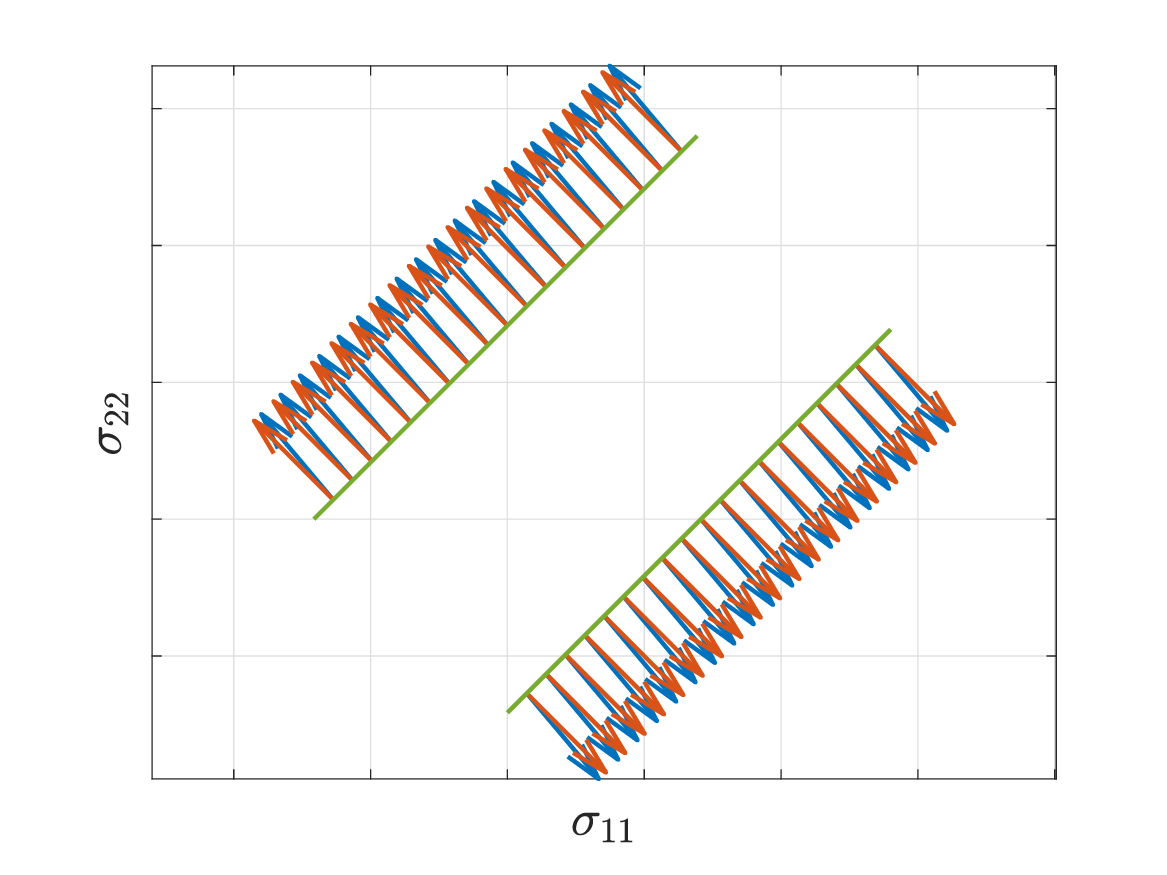}
        \caption{ We depict the flow vectors for \textcolor{color1}{$\alpha=0.5$}, \textcolor{color2}{$\alpha=1$} and the \textcolor{color3}{yield surface}. $\boldsymbol{\Delta}=\begin{pmatrix}
            100 & 100 \\
            100 & 200 
        \end{pmatrix}$.}
    \end{subfigure}
    \caption{The flow-vectors of the fractional flow rule for $\alpha=0.5$, compared to the flow-vectors for $\alpha \approx 1$, projected to the $\sigma_{11}-\sigma_{22}$ plane. (A): $d=3$, (B): $d=2$}
    \label{fig:flow}
\end{figure}
To solve \eqref{eq:modelproblem}, we use a first order spline space $V_h $ on a triangulation $\mathcal{T}$ of $\Omega$, i.e. 
\begin{align*}
    V_h=\Big\{ \mathbf{u} \in W^{1,\infty}_0(\Omega)^2 \cap C(\Omega,\mathbb{R}^2): \mathbf{u}_{|T} \ \text{is affine on all } T \in \mathcal{T}    \Big\}.
\end{align*}
The experiments are done on a mesh with roughly $1.2\cdot 10^6$ degrees of freedom with the finite-element framework \textit{FEniCSx} \cite{FEniCSProject.08.12.2023}.
In order to solve the nonlinear system of equations \eqref{eq:modelproblem}, we use so called \textit{semismooth Newton-methods} with convergence threshold for the residual of $10^{-8}$, as introduced in \cite{Facchinei.2003b} and applied to elasto-plasticity in \cite[Table 1]{Sauter}. Due to Theorems~\ref{thm:semismooth}\&\ref{thm:rposdef2} (semismoothness of  $R_n$ and positive-definiteness of  $\partial R_n$) together with~\cite[Thm. 3.2]{Sauter} we expect local superlinear convergence of the Newton-method. The fractional derivatives are evaluated with the convolutional quadrature equivalent to the implicit Euler method with $n_{\rm conv}=10$ nodes (see~\cite{Lubich.1988,Lubich.1988b, Lubich.1986} for details). To confirm the superlinear convergence, we plot the norm of the residuals in Figure~\ref{fig:suplin1}(A) for $\alpha=0.5$.
\begin{figure}
  \centering
  \begin{subfigure}{0.3 \textwidth}
      \centering
      \includegraphics[width=\textwidth]{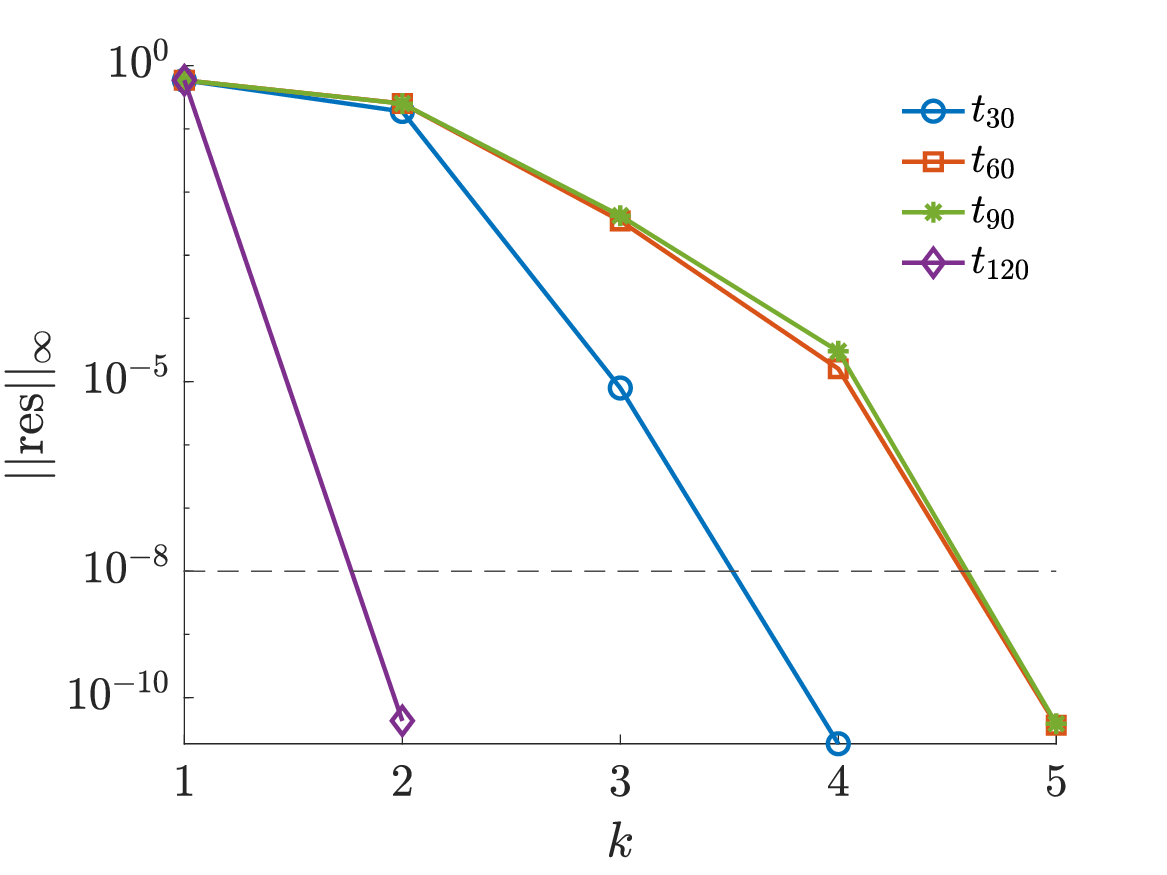}
      \caption{$\alpha=0.5$}
  \end{subfigure}
  \begin{subfigure}{0.3 \textwidth}
      \centering
      \includegraphics[width=\textwidth]{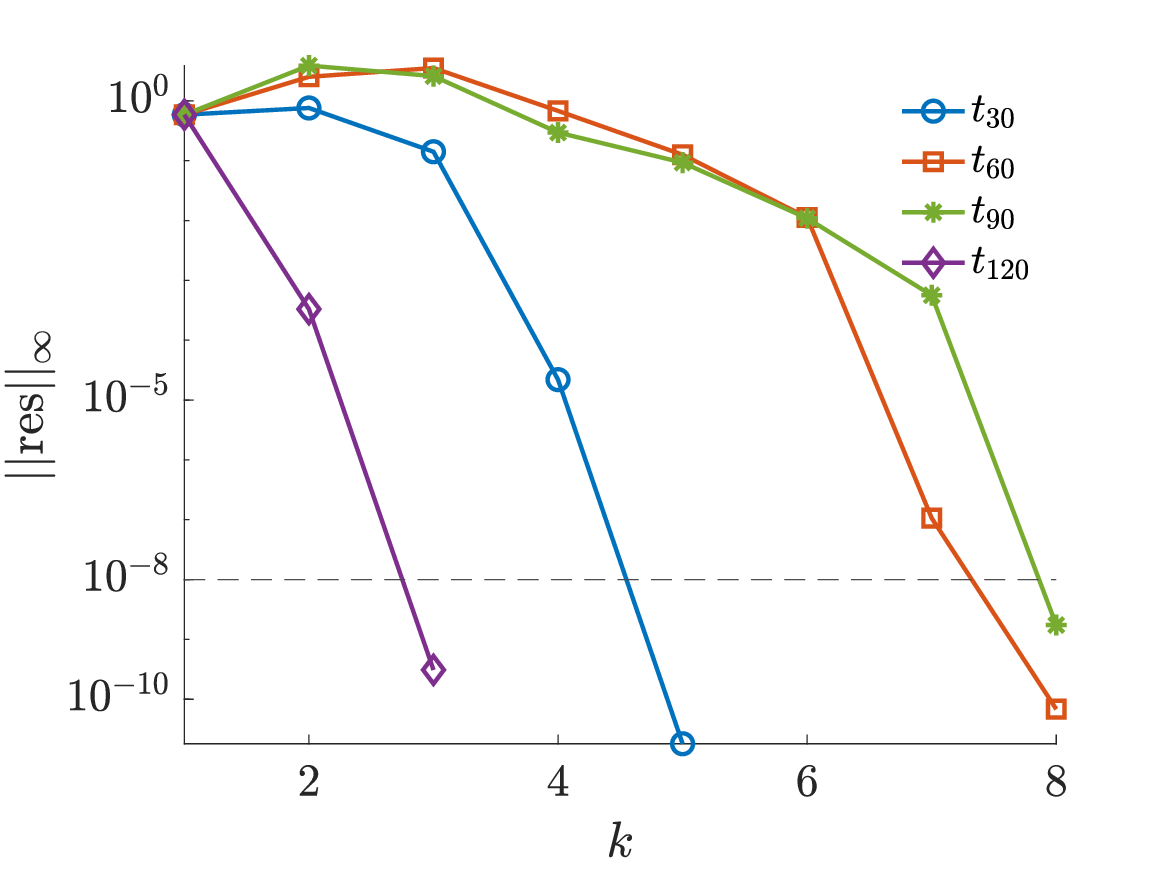}
      \caption{$\alpha=0.5$, $k_1=k_2=1500$}
  \end{subfigure}
  \begin{subfigure}{0.3 \textwidth}
      \centering
      \includegraphics[width=\textwidth]{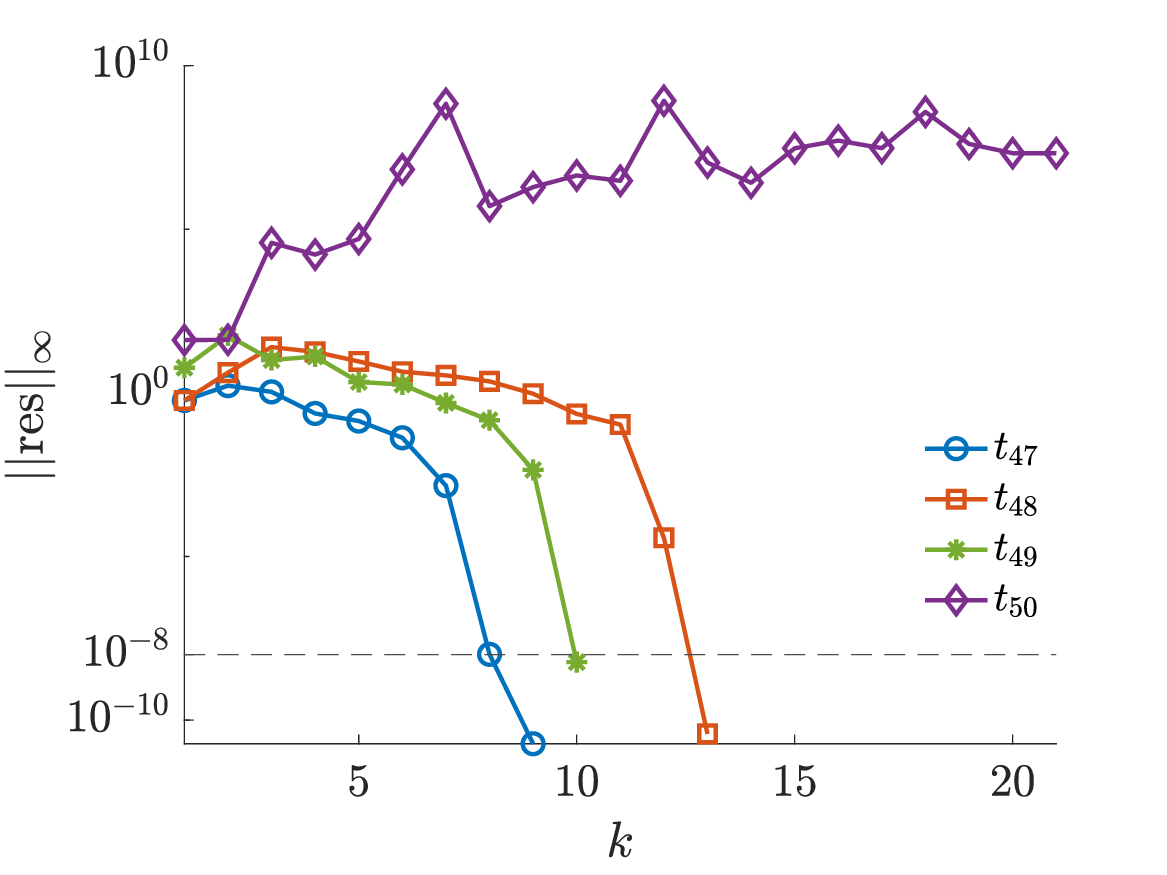}
      \caption{$\alpha=0.5$, $k_1=k_2=110$}
  \end{subfigure}
  \caption{Residual of Newton iterates for Newton step $k$. Convergence threshold was $10^{-8}$.}
  \label{fig:suplin1}
\end{figure}
Smaller hardening variables $k_1, k_2$ increase the Newton steps as seen in Figure \ref{fig:suplin1}(B) and finally result in convergence problems as seen in \ref{fig:suplin1}(C). This is partially due to substantially worse initial guesses if the plastic deformation is larger and a more and more ill-posed problem as mentioned in Theorem \ref{thm:solfrac}. Figure~\ref{fig:suplin} shows the number of necessary Newton steps at each time step $t_n$ for different values of $\alpha$ and varying interval matrix $\boldsymbol{\Delta}$. We see that the number of Newton steps is very robust with respect to both quantities.
\begin{figure}[htb]
  \centering
 \begin{subfigure}{0.45 \textwidth}
     \centering
     \includegraphics[width=\textwidth]{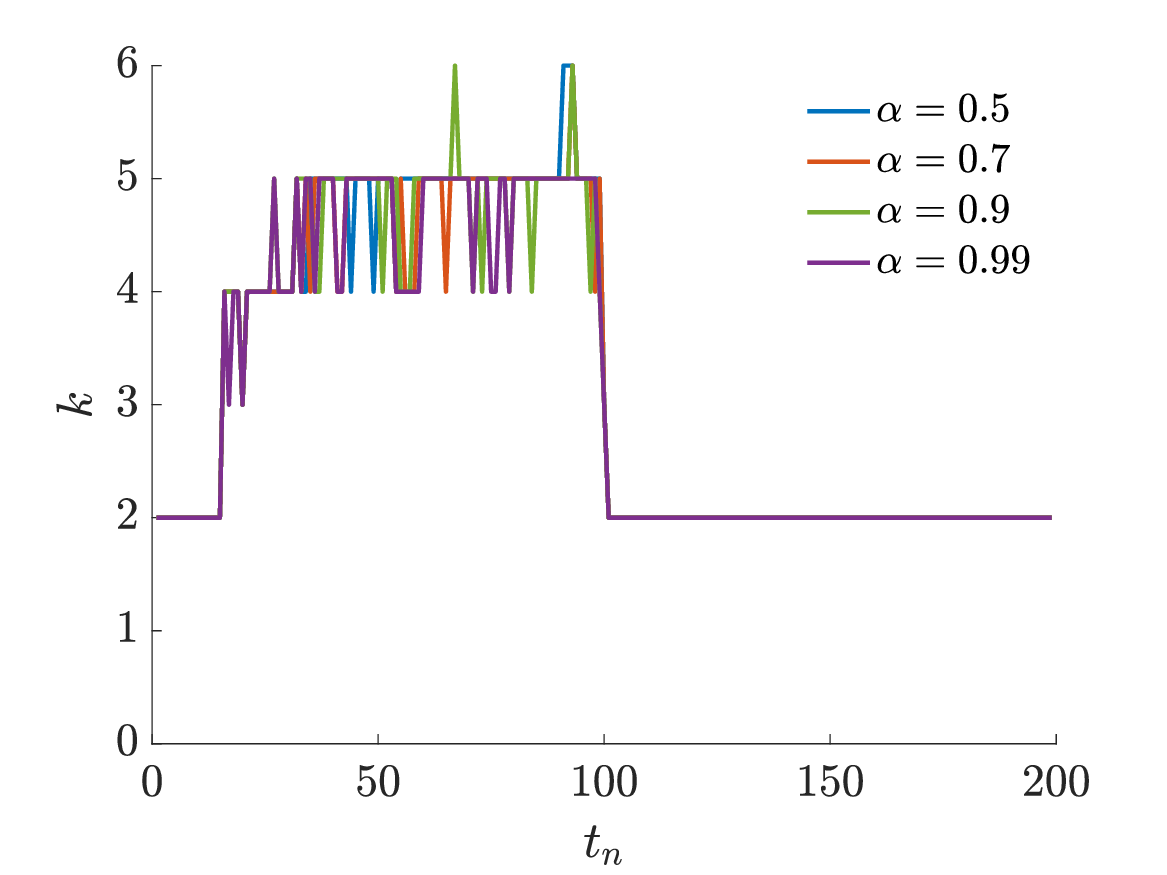}
     \caption{}
 \end{subfigure}
 \begin{subfigure}{0.45 \textwidth}
     \centering
     \includegraphics[width=\textwidth]{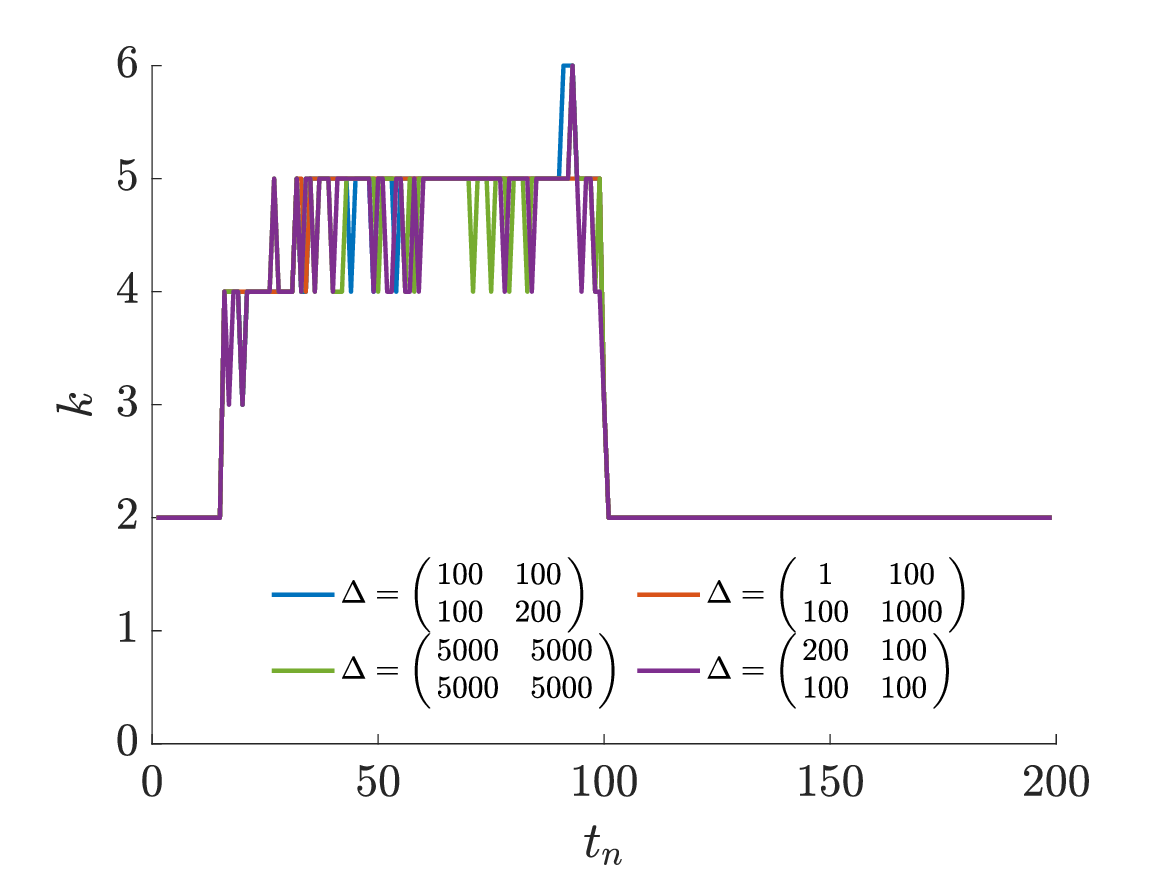}
     \caption{}
 \end{subfigure}
  \caption{Number of necessary Newton steps for different values of (A):$\alpha$ and (B): $\boldsymbol{\Delta}$. Convergence threshold is $10^{-8}$.}
  \label{fig:suplin}
\end{figure}

Next, we investigate how $\alpha$ and $\boldsymbol{\Delta}$ influence the overall plastic behavior. To that end, we take displacement measurements as seen in Figure~\ref{fig:domain}(C).
Evolution of the displacements values is presented during loading and unloading for different values of $\alpha$ and $\boldsymbol{\Delta}$.
\begin{figure}
  \centering
  \begin{subfigure}{0.45\textwidth}
      \centering
      \includegraphics[width=\textwidth]{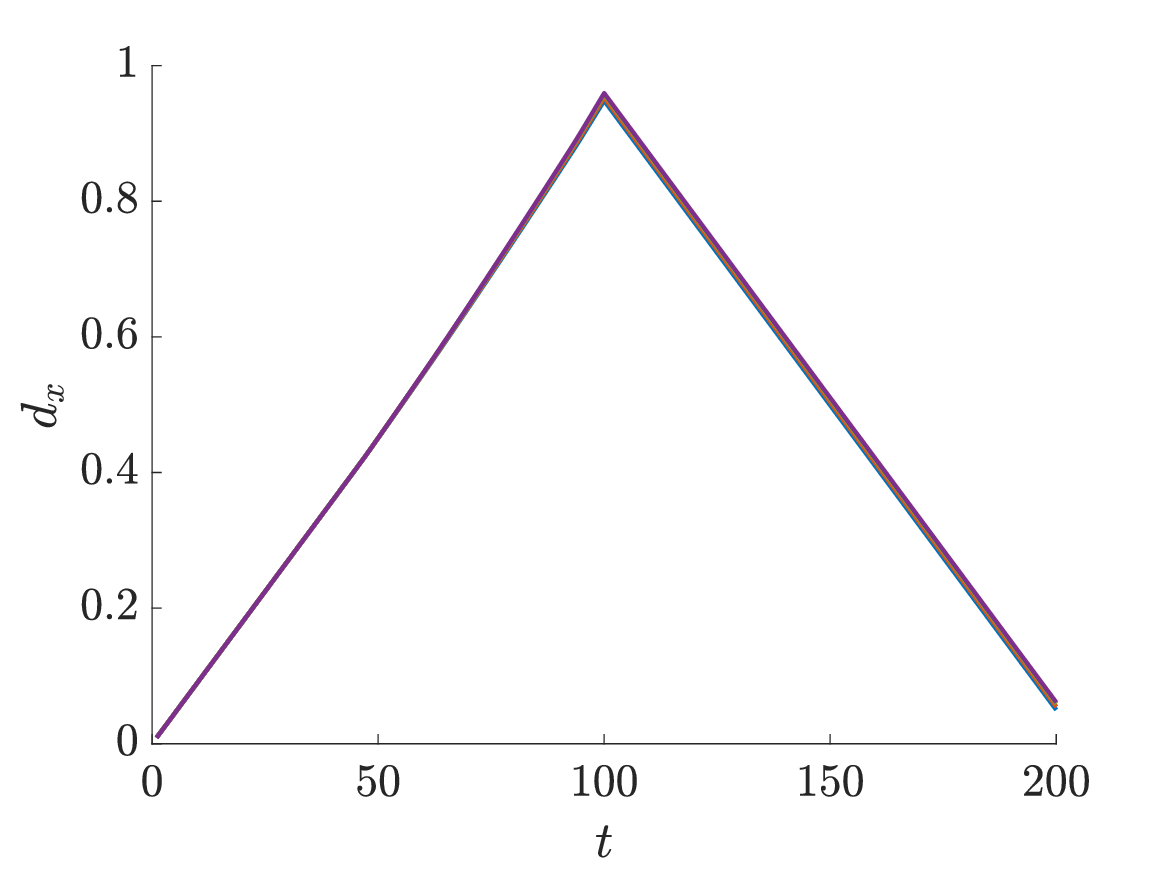}
      \caption{}
  \end{subfigure}
  \begin{subfigure}{0.45\textwidth}
      \centering
      \includegraphics[width=\textwidth]{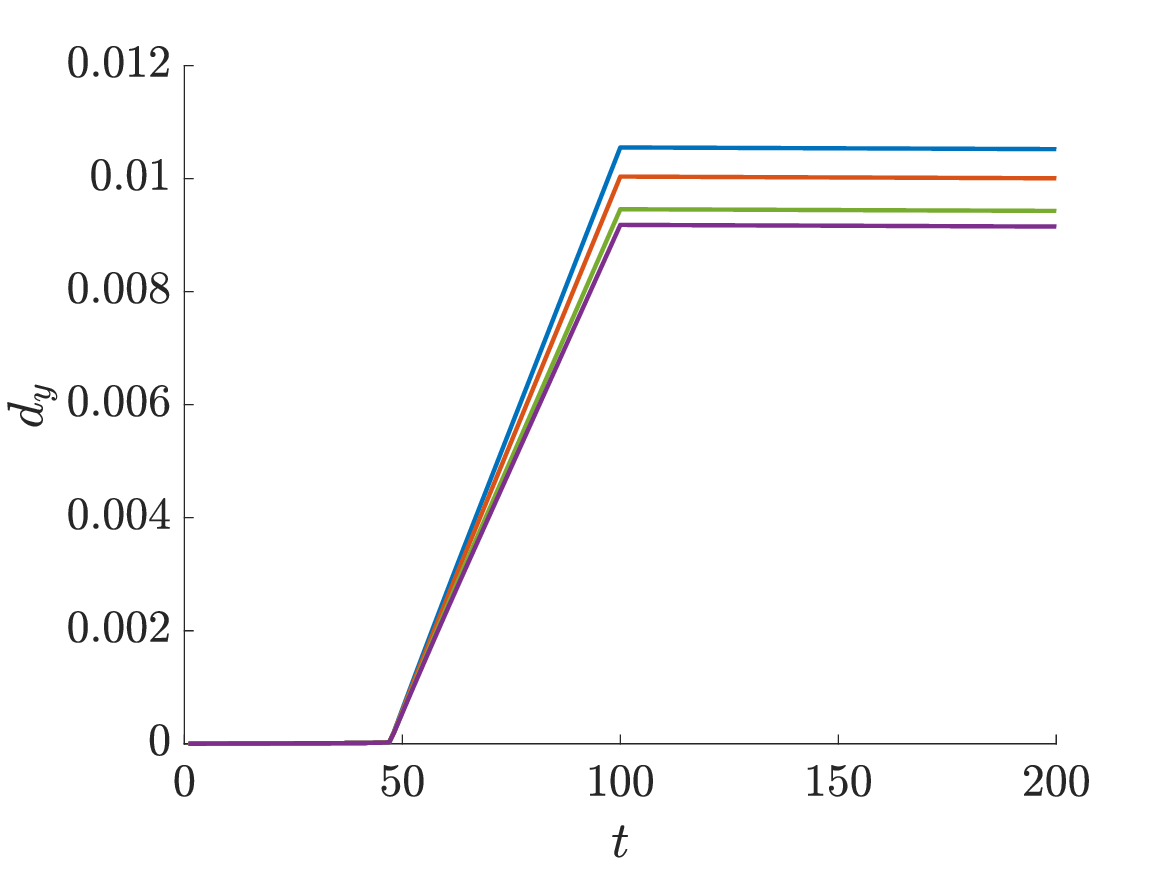}
      \caption{}
  \end{subfigure} \\
  \begin{subfigure}{0.45\textwidth}
      \centering
      \includegraphics[width=\textwidth]{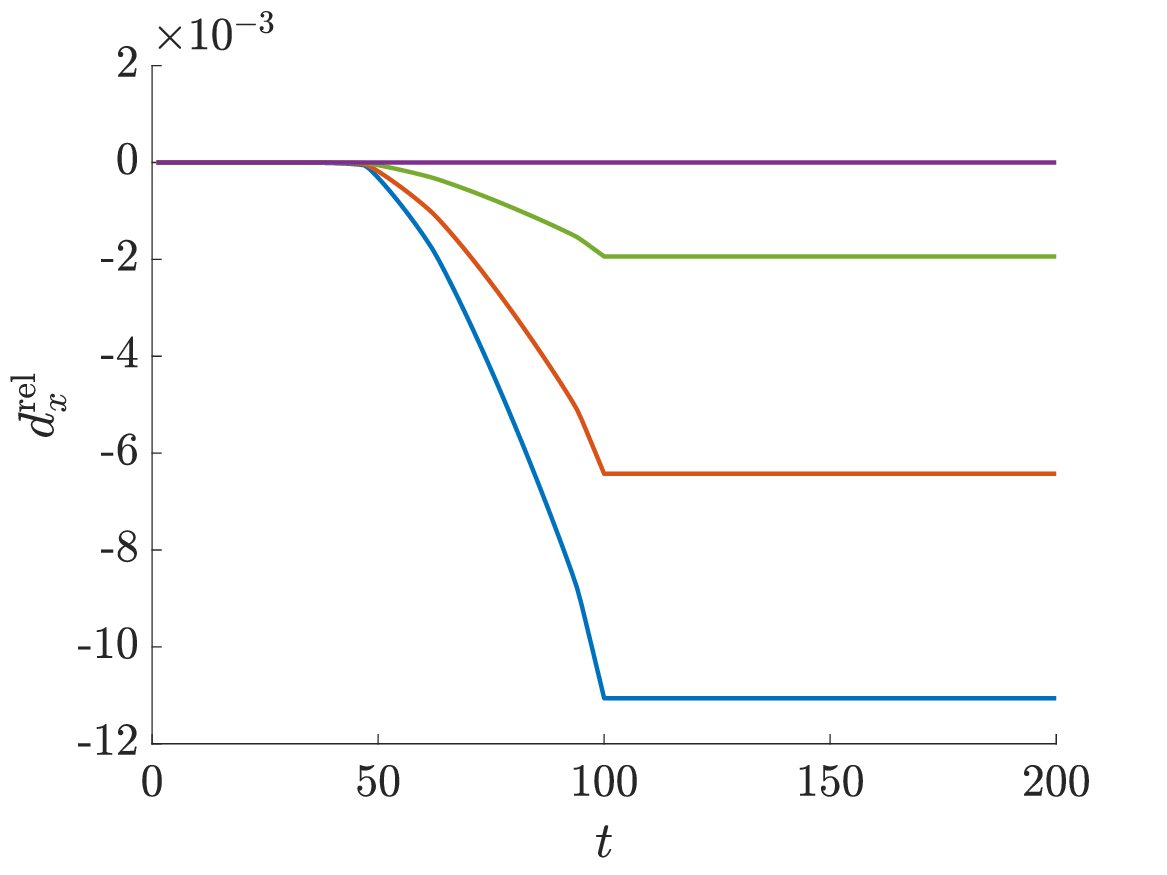}
      \caption{}
  \end{subfigure}
  \begin{subfigure}{0.45\textwidth}
      \centering
      \includegraphics[width=\textwidth]{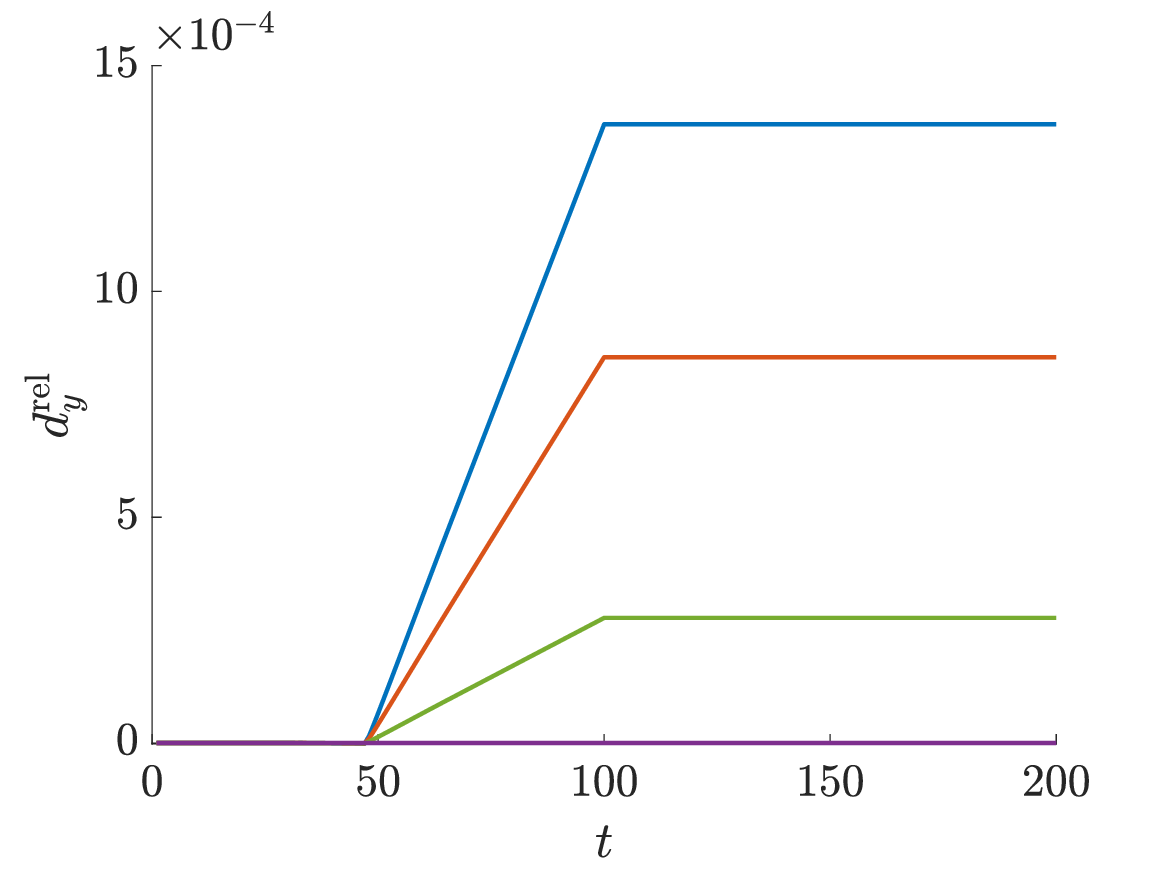}
      \caption{}
  \end{subfigure}
  \caption{Evolution of the displacement measurements $d_x$ (A) and $d_y$ (B) over time and for $\alpha \in \{\textcolor{color1}{0.5},\textcolor{color2}{0.7},\textcolor{color3}{0.9},\textcolor{color4}{0.99}\} $. (C) shows the displacement measurement $d_x$ relative to the measurement for $\alpha=0.99$. (D) shows the displacement measurement $d_y$ relative to the measurement for $\alpha=0.99$.}
  \label{fig:disps_alpha}
\end{figure}
\begin{figure}
  \centering
  \begin{subfigure}{0.45\textwidth}
      \centering
      \includegraphics[width=\textwidth]{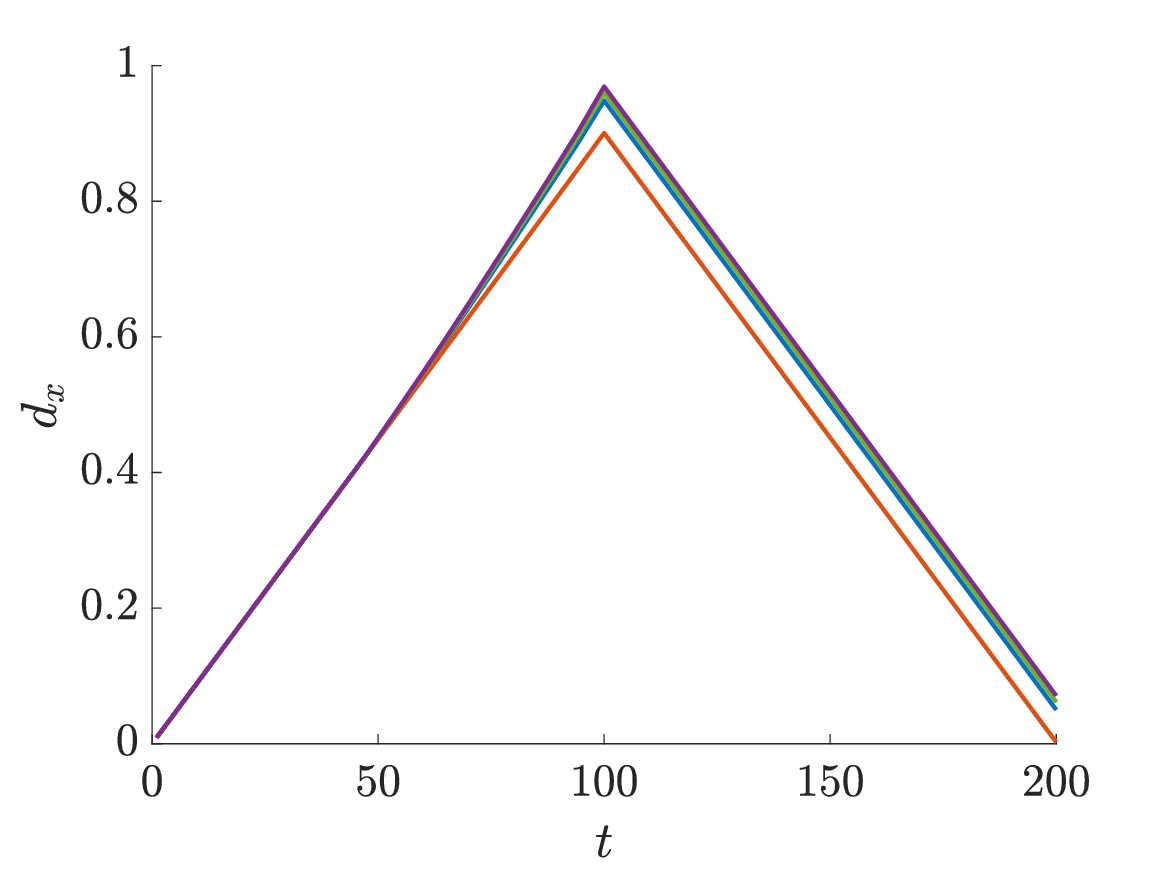}
      \caption{}
  \end{subfigure}
  \begin{subfigure}{0.45\textwidth}
      \centering
      \includegraphics[width=\textwidth]{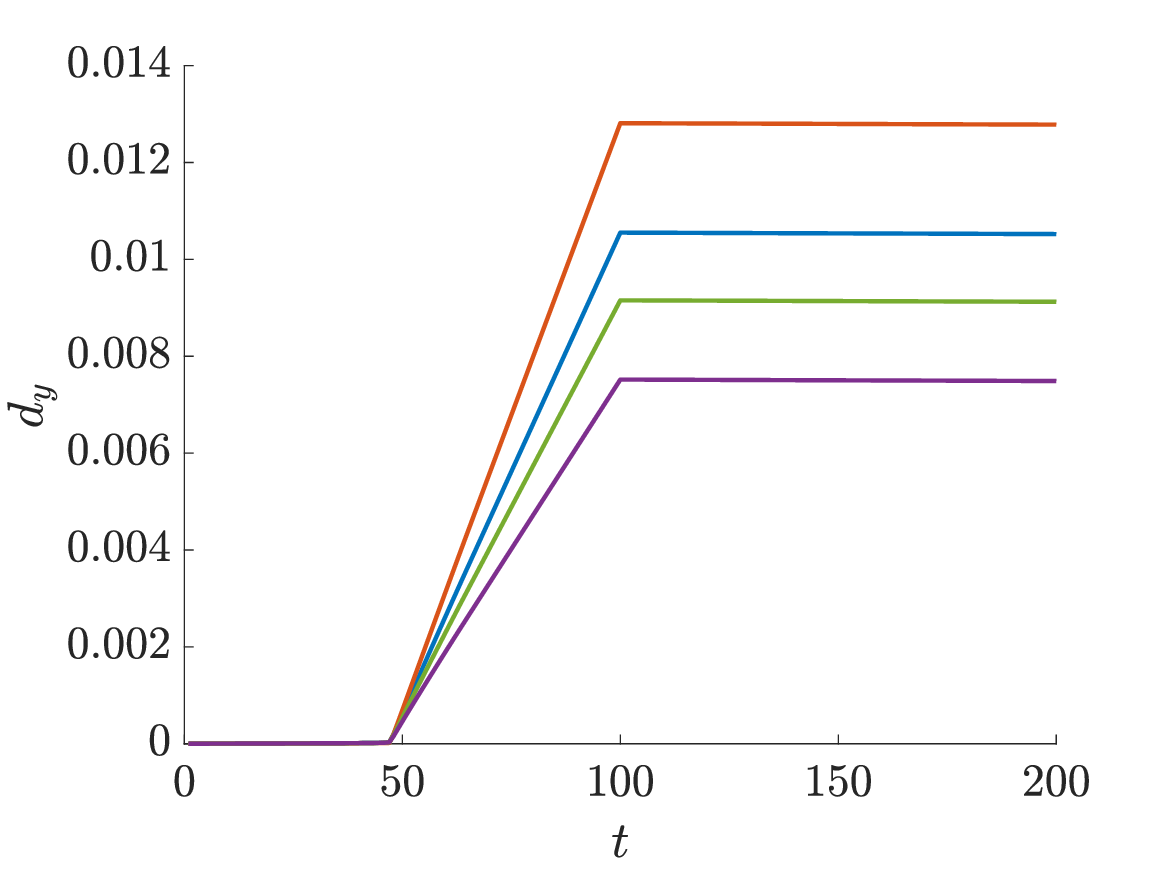}
      \caption{}
  \end{subfigure} \\
  \begin{subfigure}{0.45\textwidth}
      \centering
      \includegraphics[width=\textwidth]{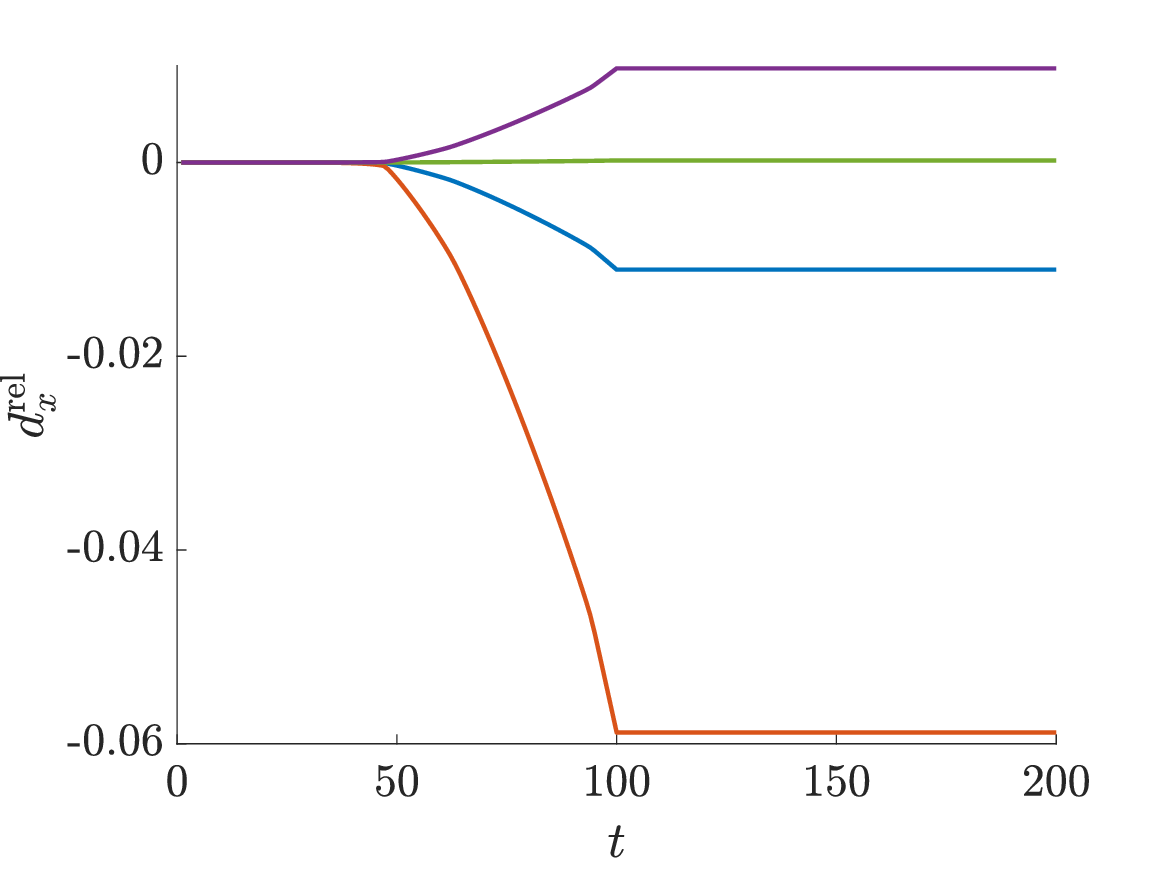}
      \caption{}
  \end{subfigure}
  \begin{subfigure}{0.45\textwidth}
      \centering
      \includegraphics[width=\textwidth]{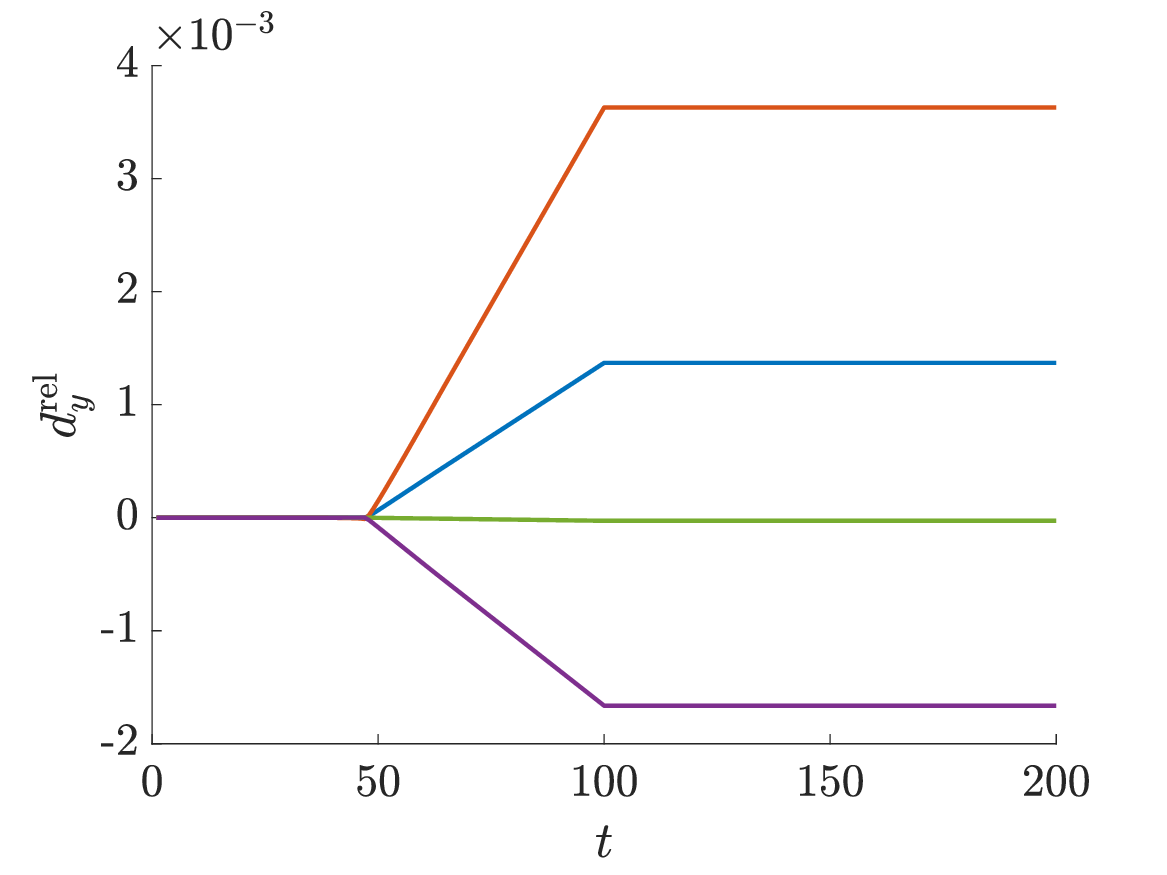}
      \caption{}
  \end{subfigure}
  \caption{Evolution of the displacement measurements $d_x$ (A) and $d_y$ (B) over time and for $\boldsymbol{\Delta} \in \Big\{ \textcolor{color1}{\begin{pmatrix}
      100 & 100 \\
      100 & 200
  \end{pmatrix}},\textcolor{color2}{\begin{pmatrix}
      1 & 100 \\
      100 & 1000
  \end{pmatrix}},\textcolor{color3}{\begin{pmatrix}
      5000 & 5000 \\
      5000 & 5000
  \end{pmatrix}},\textcolor{color4}{\begin{pmatrix}
      200 & 100 \\
      100 & 100
  \end{pmatrix}} \Big\}$. (C) shows the displacement measurement $d_x$ relative to the measurement for $\alpha=0.99$. (D) shows the displacement measurement $d_y$ relative to the measurement for $\alpha=0.99$.}
  \label{fig:disps_DELTA}
\end{figure}
Figure~\ref{fig:disps_alpha} shows that the plastic deformation is influenced by the value of $\alpha$. Around $t\approx 50$, the plasticity starts to dominate and plastic deformation $d_y$ is increasing for decreasing $\alpha$. Furthermore, we see in Figure~\ref{fig:disps_alpha}(A), that the displacement $d_x$ is mostly elastic. In the unloading phase (i.e. $t>100$), plastic deformation does not change, whereas elastic deformation vanishes as one would expect. Figure~\ref{fig:disps_DELTA} shows the influence of the interval matrix $\boldsymbol{\Delta}$ on the plastic behavior. Again the influence is different for $d_x$ and $d_y$. It is interesting, that for a interval matrix of the form $\boldsymbol{\Delta} =\lambda \begin{pmatrix}
  1 & 1 \\
  1 & 1
\end{pmatrix}, \lambda >0$, the plastic behavior is very close to the limit case $\alpha \nearrow 1$. Lastly we give qualitative pictures of the remaining deformation after unloading in Figure~\ref{fig:deformedgeometry}. The deformation is amplified by a factor of $20$. Clearly the previously discussed differences in plastic deformation are visible.
\begin{figure}
  \centering
  \includegraphics[trim= 0 400 0 400 , scale=0.2, clip]{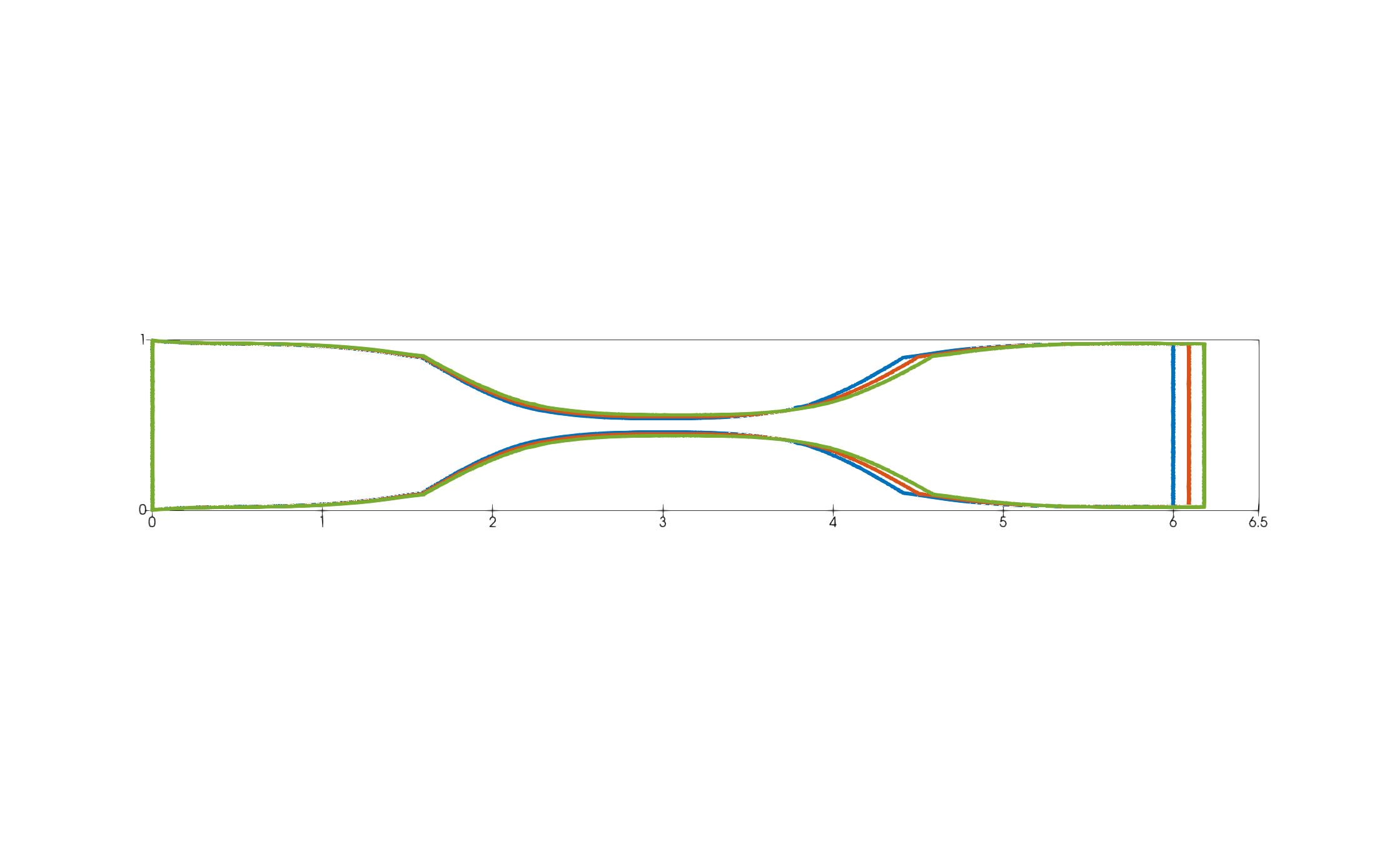}
  \caption{Qualitative pictures of the remaining deformation after unloading for different values of $\alpha$. Amplified by a factor of $20$. $\alpha \in \{\textcolor{color1}{0.5},\textcolor{color2}{0.7},\textcolor{color3}{0.9}\} $}
  \label{fig:deformedgeometry}
\end{figure}

\subsection{Three-dimensional domain}
The domain $\Omega \subset \mathbb{R}^3$ (see Figure~\ref{fig:3ddom}(A)) is a rectangular block with a vertical cylindrical hole in the center, which is anchored on the left and subjected to a pulling force upwards (parallel to the cylinder axis) on the right.
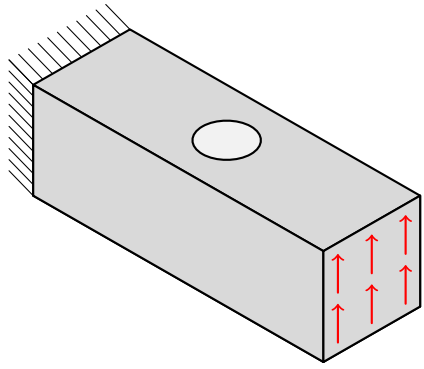
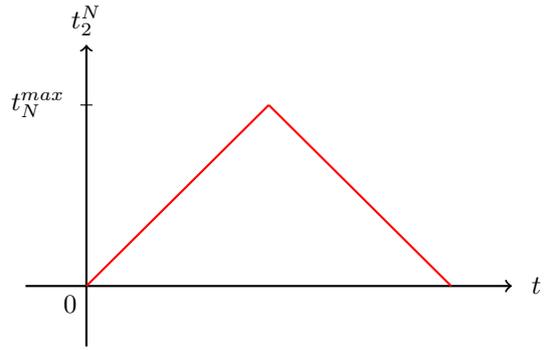
\begin{figure}
   \centering
   \begin{subfigure}{0.45 \textwidth}
       \centering
  \begin{tikzpicture}[scale=0.9, line cap=round,line join=round,%
   x={({-\xx cm,-\xy cm})},y={(\xx cm,-\xy cm)},z={(0 cm,\zy cm)}]
          % dimensions
          \def\a{2} % box side length
          \def\r{0.5} % cylinder radius
          \def\h{6} % box height
          % coordinates
          \coordinate (A) at (0,0,0);
          \coordinate (B) at (\a,0,0);
          \coordinate (C) at (\a,\h,0);
          \coordinate (D) at (0,\h,0);
          \coordinate (A') at (0,0,\a);
          \coordinate (B') at (\a,0,\a);
          \coordinate (C') at (\a,\h,\a);
          \coordinate (D') at (0,\h,\a);
          % cylinder
          
          %\draw[dashed] (1,3,0) circle (\r);
          \draw[thick] (1,3,\a) circle (\r);
          % box
          %\draw (A) -- (D) -- (D') -- (A') -- cycle;
          \draw[white,fill=gray!30] (A) -- (B) -- (B') -- (A') -- cycle;
          \draw[thick,fill=gray!30] (A') -- (B') -- (C') -- (D') -- cycle;
          %\draw[dashed] (A) -- (A');
          \draw[thick] (B) -- (B');
          \draw[thick] (C) -- (C');
          \draw[thick] (D) -- (D');
          %\draw[dashed] (A) -- (B);
          %\draw[dashed] (A) -- (D);
          \draw[thick]  (D) -- (C);
          \draw[thick]  (C) -- (B);
           \draw[thick,fill=gray!30] (B') -- (C') -- (C) -- (B) -- cycle;
          \draw[thick,fill=gray!10] (1,3,\a) circle (\r);
        \draw (\a,0,0)-- +(0,-0.5,0.2);
          \draw (\a,0,0.2)-- +(0,-0.5,0.2);
          \draw (\a,0,0.4)-- +(0,-0.5,0.2);
          \draw (\a,0,0.6)-- +(0,-0.5,0.2);
          \draw (\a,0,0.8)-- +(0,-0.5,0.2);
          \draw (\a,0,1)-- +(0,-0.5,0.2);
           \draw (\a,0,1.2)-- +(0,-0.5,0.2);
          \draw (\a,0,1.4)-- +(0,-0.5,0.2);
          \draw (\a,0,1.6)-- +(0,-0.5,0.2);
          \draw (\a,0,1.8)-- +(0,-0.5,0.2);
          \draw (\a,0,2)-- +(0,-0.5,0.2);
           \draw (\a,0,2)-- +(0,-0.5,0.2);
          \draw (\a-0.2,0,2)-- +(0,-0.5,0.2);
          \draw (\a-0.4,0,2)-- +(0,-0.5,0.2);
          \draw (\a-0.6,0,2)-- +(0,-0.5,0.2);
          \draw (\a-0.8,0,2)-- +(0,-0.5,0.2);
          \draw (\a-1,0,2)-- +(0,-0.5,0.2);
           \draw (\a-1.2,0,2)-- +(0,-0.5,0.2);
          \draw (\a-1.4,0,2)-- +(0,-0.5,0.2);
          \draw (\a-1.6,0,2)-- +(0,-0.5,0.2);
          \draw (\a-1.8,0,2)-- +(0,-0.5,0.2);
          \draw (\a-2,0,2)-- +(0,-0.5,0.2);
        \draw[thick,fill=gray!30] (C) -- (C') -- (D') -- (D) -- cycle;
          % hidden lines
          %\draw[dashed] (A') -- (C');
          %\draw[dashed] (A) -- (C);
          \draw[->, red, thick] (0.3,\h,0.2) -- +(0,0,0.7);
          \draw[->, red, thick] (1,\h,0.2) -- +(0,0,0.7);
          \draw[->, red, thick] (1.7,\h,0.2) -- +(0,0,0.7);
          \draw[->, red, thick] (0.3,\h,1.1) -- +(0,0,0.7);
          \draw[->, red, thick] (1,\h,1.1) -- +(0,0,0.7);
          \draw[->, red, thick] (1.7,\h,1.1) -- +(0,0,0.7);
          
      \end{tikzpicture}
  \caption{$\Omega \subset \mathbb{R}^3$.}
  \label{fig:3ddom_side}
   \end{subfigure}
   \begin{subfigure}{0.45 \textwidth}
       \centering
           \begin{tikzpicture}[scale=0.8]
  \draw[->,thick] (-1,0) --(7,0);
  \draw[->,thick] (0,-1) -- (0,4);
  \draw[red,thick] (0,0) -- (3,3);
  \draw[red,thick] (3,3)--(6,0);
  \node at(7.4,0) {$t$};
  \node at (0,4.4) {$t^N_2$};
  \draw (-0.1,3)--(0.1,3);
  \node[left] at (-0.2,3) {$t^{max}_N$};
  \node[anchor=north east] at (0,0) {$0$};
\end{tikzpicture}
  \caption{$\mathbf{t}_N =(0,t_2^N,0)$ over time.}
   \end{subfigure}
   \caption{The domain $\Omega \subset \mathbb{R}^3$ and the loading regime over time.}
   \label{fig:3ddom}
\end{figure}
We have $\mathbf{b}=0$ and $\mathbf{t}_N$ vanishing everywhere but on the right face, as seen in Figure~\ref{fig:3ddom}, where it is constant in space, but varying in time as shown depicted in Figure~\ref{fig:3ddom}(B).
 \begin{figure}
    \centering
    \begin{subfigure}{0.45 \textwidth}
        \centering
        \includegraphics[width=\textwidth]{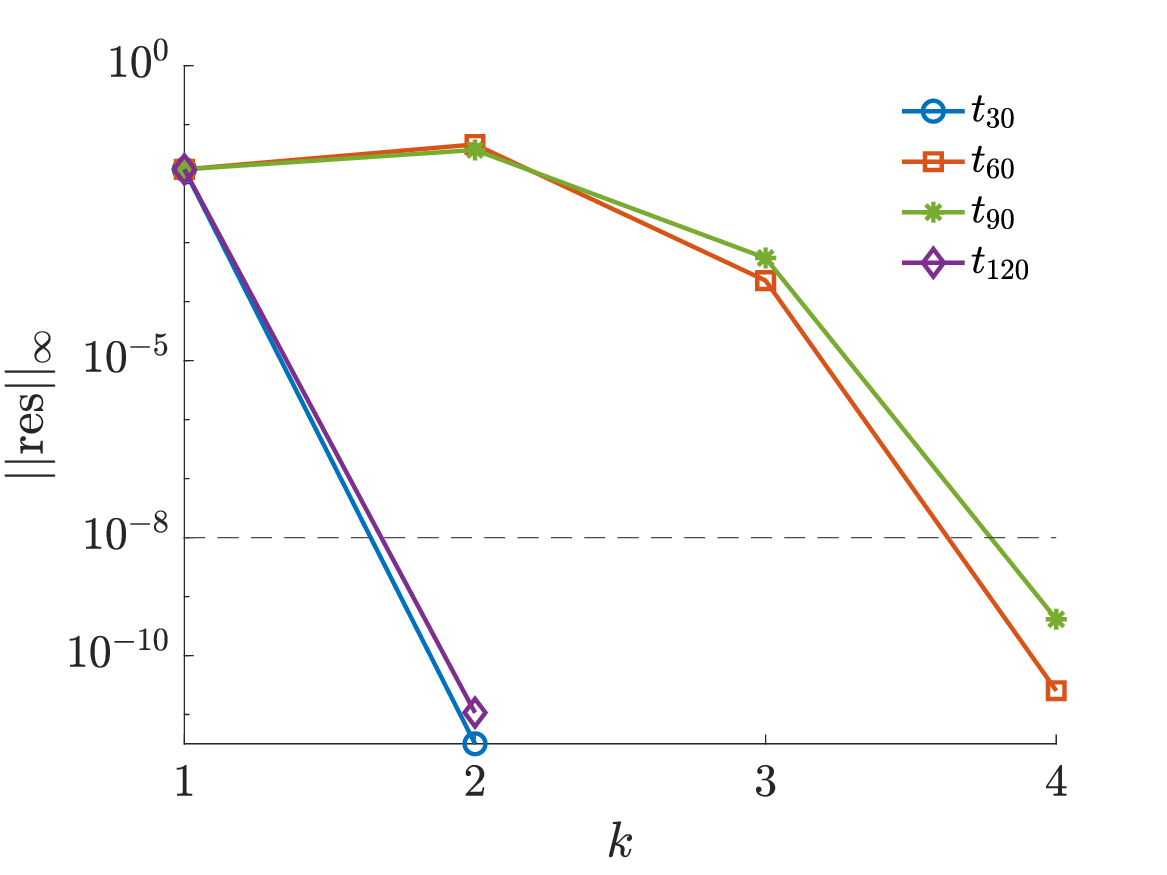}
        \caption{$\alpha=0.5$, Convergence threshold = $10^{-8}$}
    \end{subfigure}
    \begin{subfigure}{0.45 \textwidth}
       \centering
       \includegraphics[width=\textwidth]{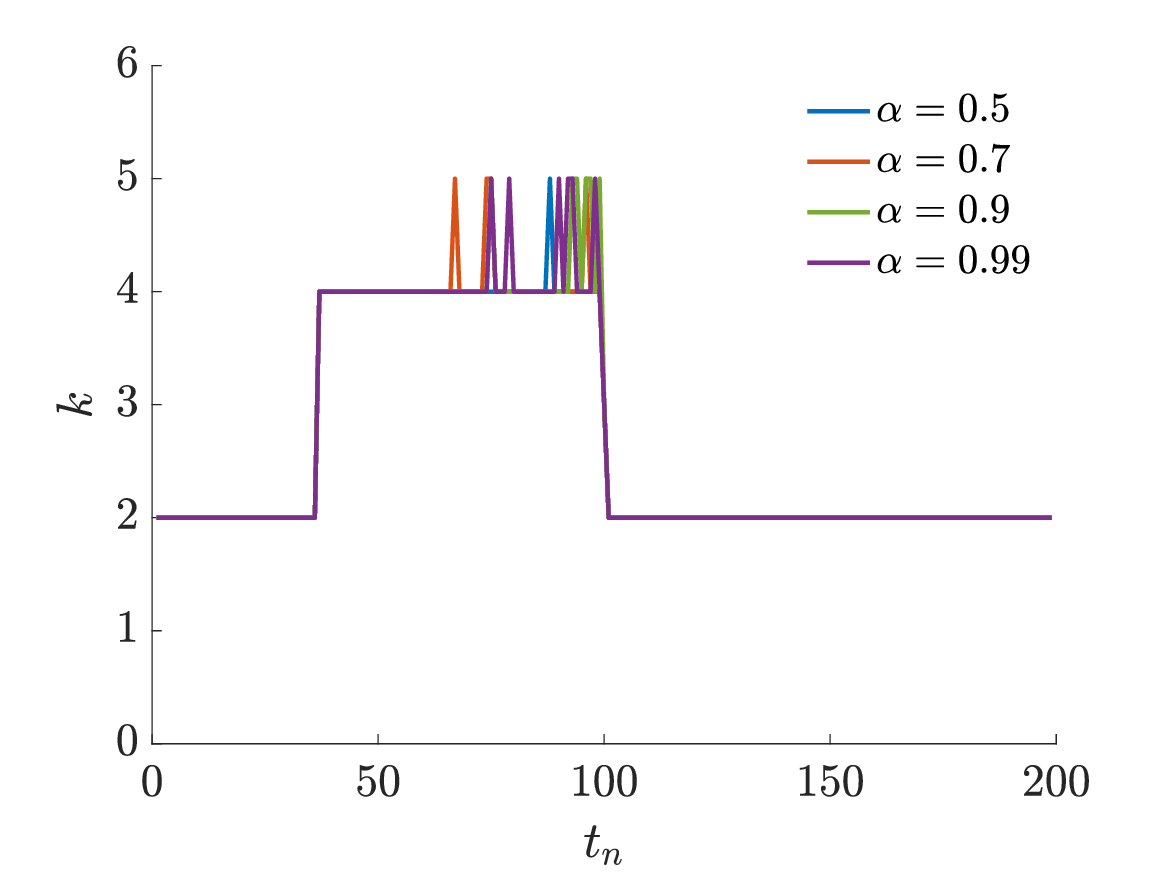}
       \caption{Convergence threshold = $10^{-8}$}
   \end{subfigure}
    \caption{(A): Residuals of Newton iterates for Newton step $k$. (B): Number of necessary Newton steps for different values of $\alpha$.}
    \label{fig:suplin3d}
\end{figure}
To solve~\eqref{eq:modelproblem}, we use the same piecewise affine ansatz space $V_h$ on a mesh $\mathcal{T}$ with $\approx 1.1\cdot 10^6$ degrees of freedom. If not mentioned otherwise, the material parameters from Table~\ref{tab:3dparas} are used, chosen similar to the ones from the 2d-experiment but adapted to fit the given constraints.
\begin{table}
	\centering
	\begin{tabular}{c |c}
		
		Material parameters & Value \\
		\hline 
		$\mu$ & $120000$\\
		$\kappa$ & $80000$ \\
		$Y_0$ & $50000$ \\
		$k_1$ & $200000$ \\
		$k_2$ & $200000$\\
		$t_N^{max}$ & $5000$\\
            $\boldsymbol{\Delta}$&$\begin{pmatrix}
                100 & 100 & 100\\
                100 & 500 & 100\\
                100 & 100 & 900
            \end{pmatrix}$
	\end{tabular}
	\caption{Default material parameters.}
	\label{tab:3dparas}
\end{table}

To judge the performance of the semismooth Newton method, we plot the norms of the residuals in each necessary Newton step for $\alpha=0.5$ in Figure~\ref{fig:suplin3d}(A) and clearly observe superlinear performance.
In Figure \ref{fig:suplin3d}(B), we show the number of necessary Newton steps over time for different values of $\alpha$. As already seen on our two-dimensional domain, $\alpha$ does not influence the convergence behavior substantially. To investigate the plastic behavior we measure the vertical displacement of the midpoint of the right-face of $\Omega$ in Figure~\ref{fig:dy3d} as well as the maximum equivalent von-Mises stress $||\sigma_{\rm eq}||_\infty=||\text{dev}(\sigma)||_\infty$ in Figure~\ref{fig:sigeq}. The displacement behavior is again dominated by the elasticity and increases with $\alpha$. For the equivalent stress, we clearly see the onset of plasticity at $t \approx 40$. Here smaller values of $\alpha$ give rise to larger stresses.
\begin{figure}
    \centering
    \begin{subfigure}{0.45\textwidth}
        \centering
        \includegraphics[width=\textwidth]{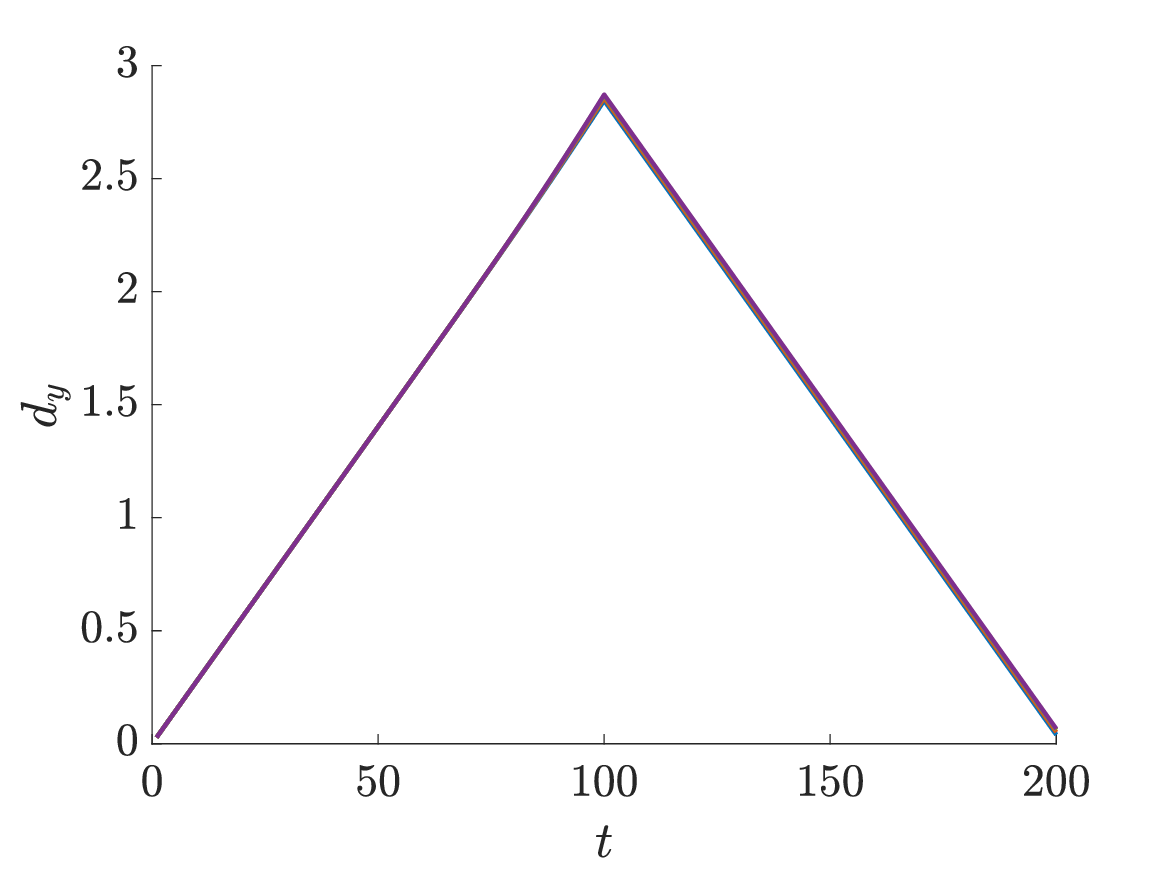}
        \caption{}
    \end{subfigure}
    \begin{subfigure}{0.45 \textwidth}
        \centering
        \includegraphics[width=\textwidth]{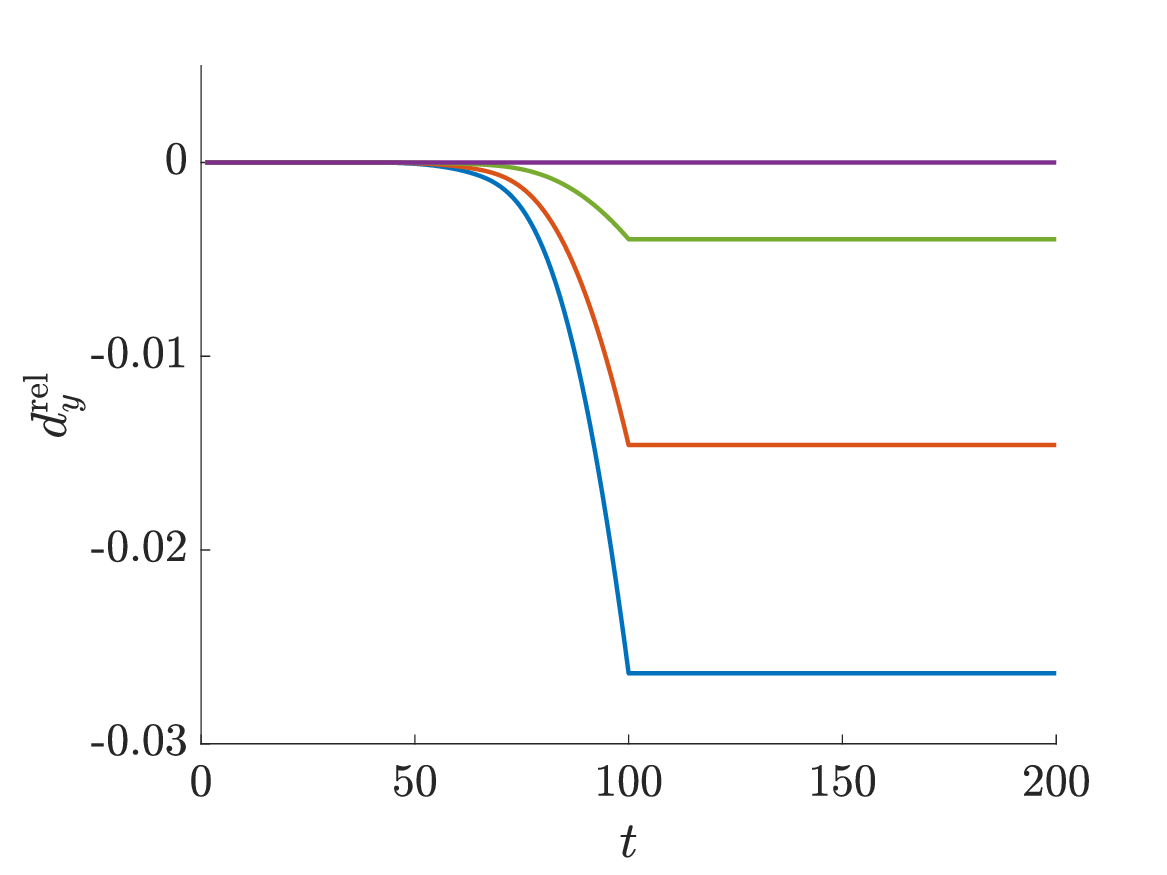}
        \caption{}
    \end{subfigure}
    \caption{(A): Vertical displacement of the midpoint of the right-face of $\Omega$ for $\alpha \in \{\textcolor{color1}{0.5},\textcolor{color2}{0.7},\textcolor{color3}{0.9},\textcolor{color4}{0.99}\}$. (B): Vertical displacement of the midpoint of the right-face relative to the displacement for $\alpha=0.99$. }
    \label{fig:dy3d}
\end{figure}
\begin{figure}
    \centering
    \begin{subfigure}{0.45\textwidth}
        \centering
        \includegraphics[width=\textwidth]{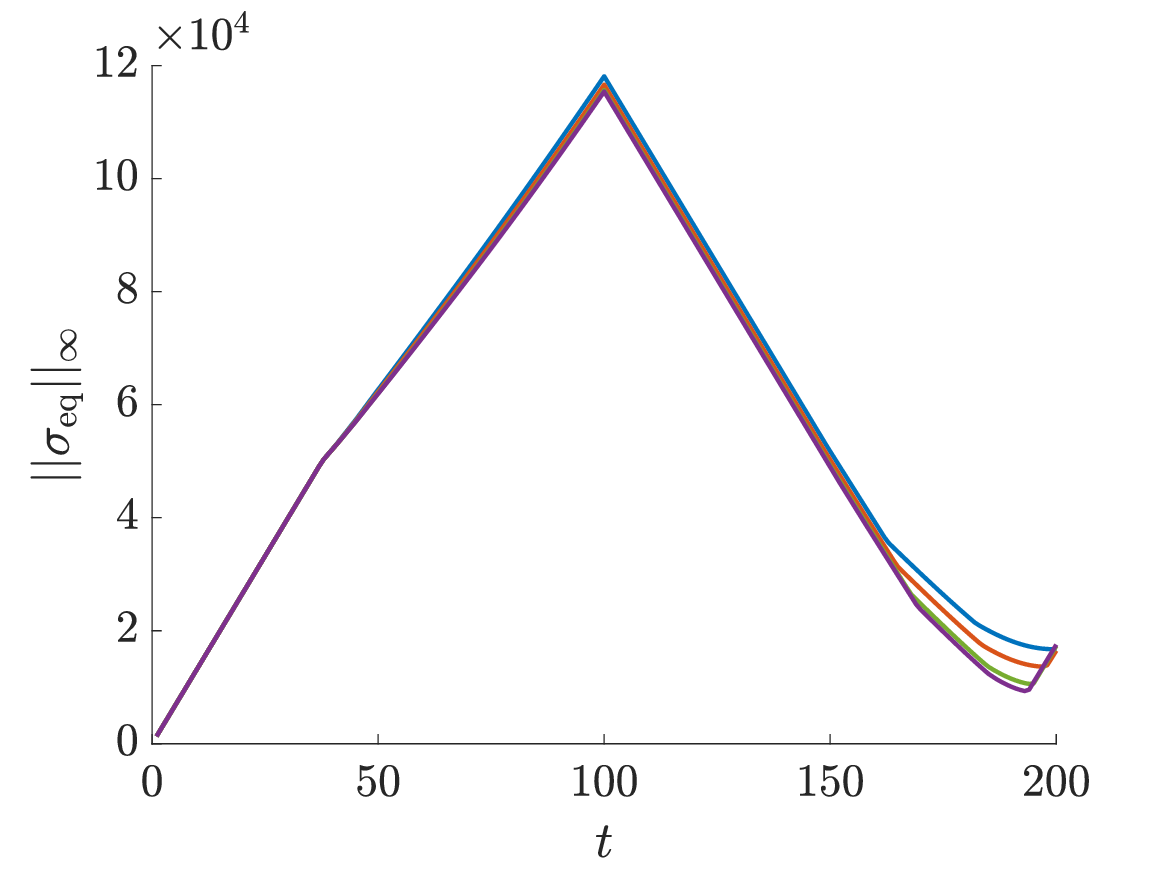}
        \caption{}
    \end{subfigure}
    \begin{subfigure}{0.45 \textwidth}
        \centering
        \includegraphics[width=\textwidth]{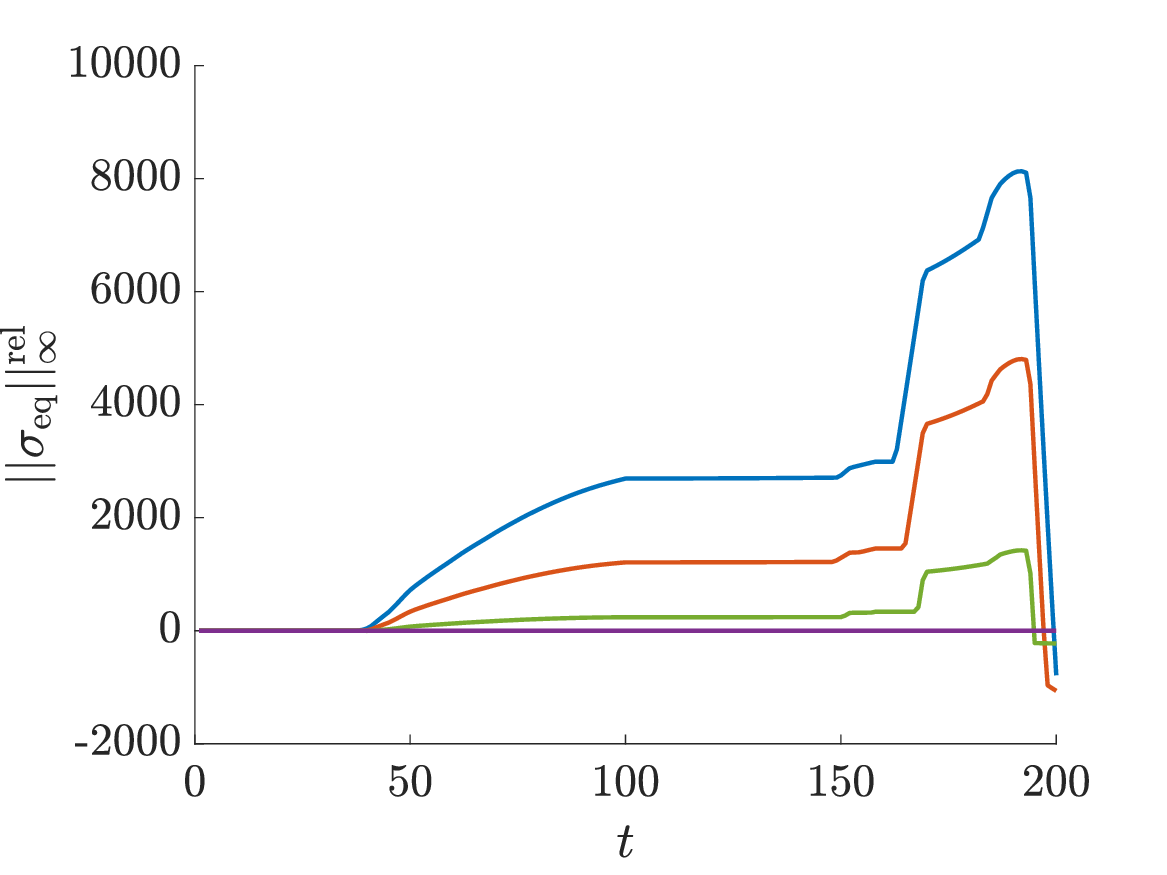}
        \caption{}
    \end{subfigure}
    \caption{(A): Maximum equivalent von-Mises stress $||\text{dev}(\sigma)||_\infty$ for $\alpha \in \{\textcolor{color1}{0.5},\textcolor{color2}{0.7},\textcolor{color3}{0.9},\textcolor{color4}{0.99}\}$. (B): Maximum equivalent von-Mises stress relative to the one for $\alpha=0.99$.}
    \label{fig:sigeq}
\end{figure}
\begin{figure}
\centering
 \begin{subfigure}{0.45 \textwidth}
     \centering
    \includegraphics[width=\textwidth]{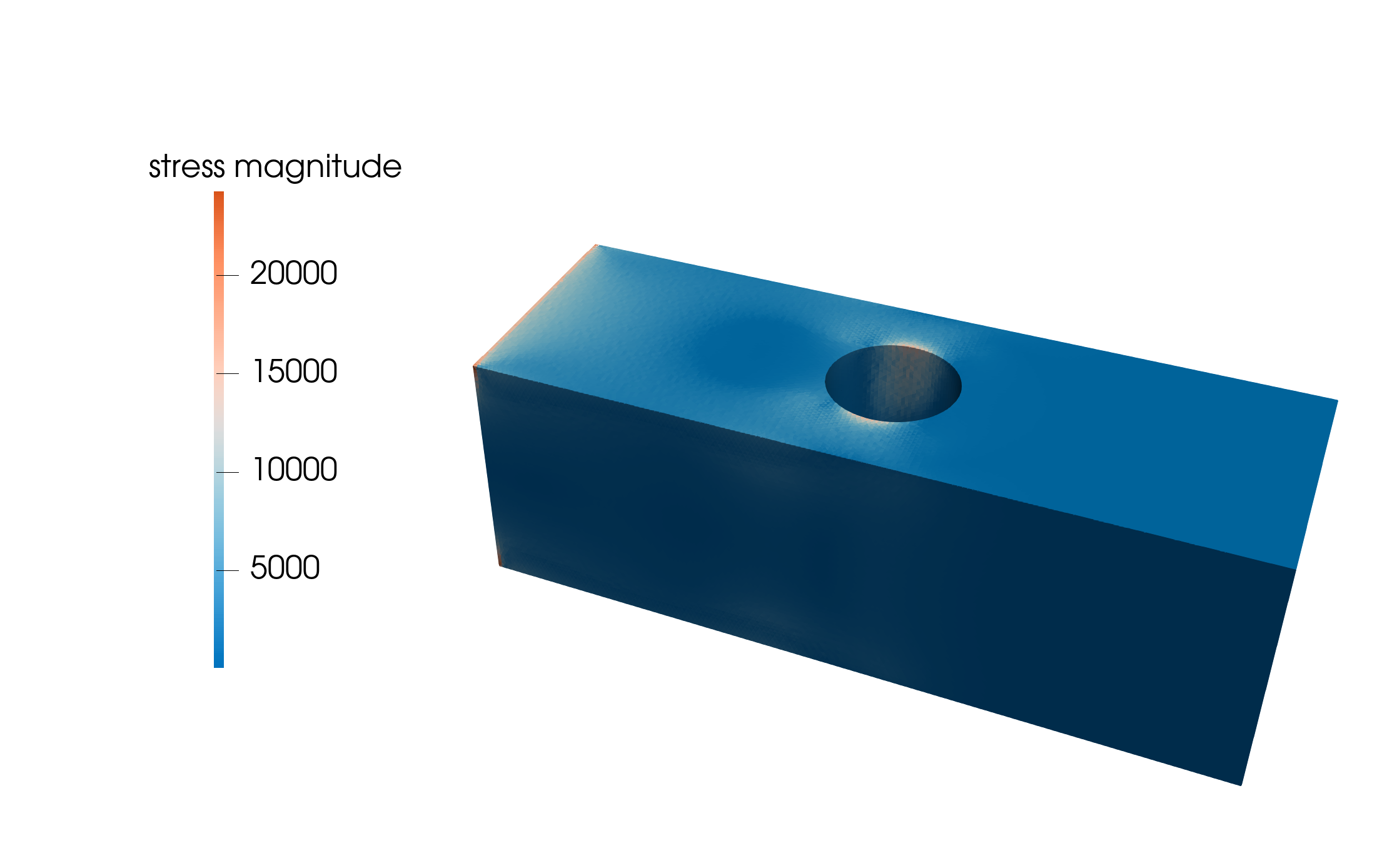}
    \caption{Stress magnitude after unloading for $\alpha=0.5$. }
    \label{fig:3dparaview}
 \end{subfigure}
    \begin{subfigure}{0.45 \textwidth}
        \centering
        \includegraphics[width=\textwidth]{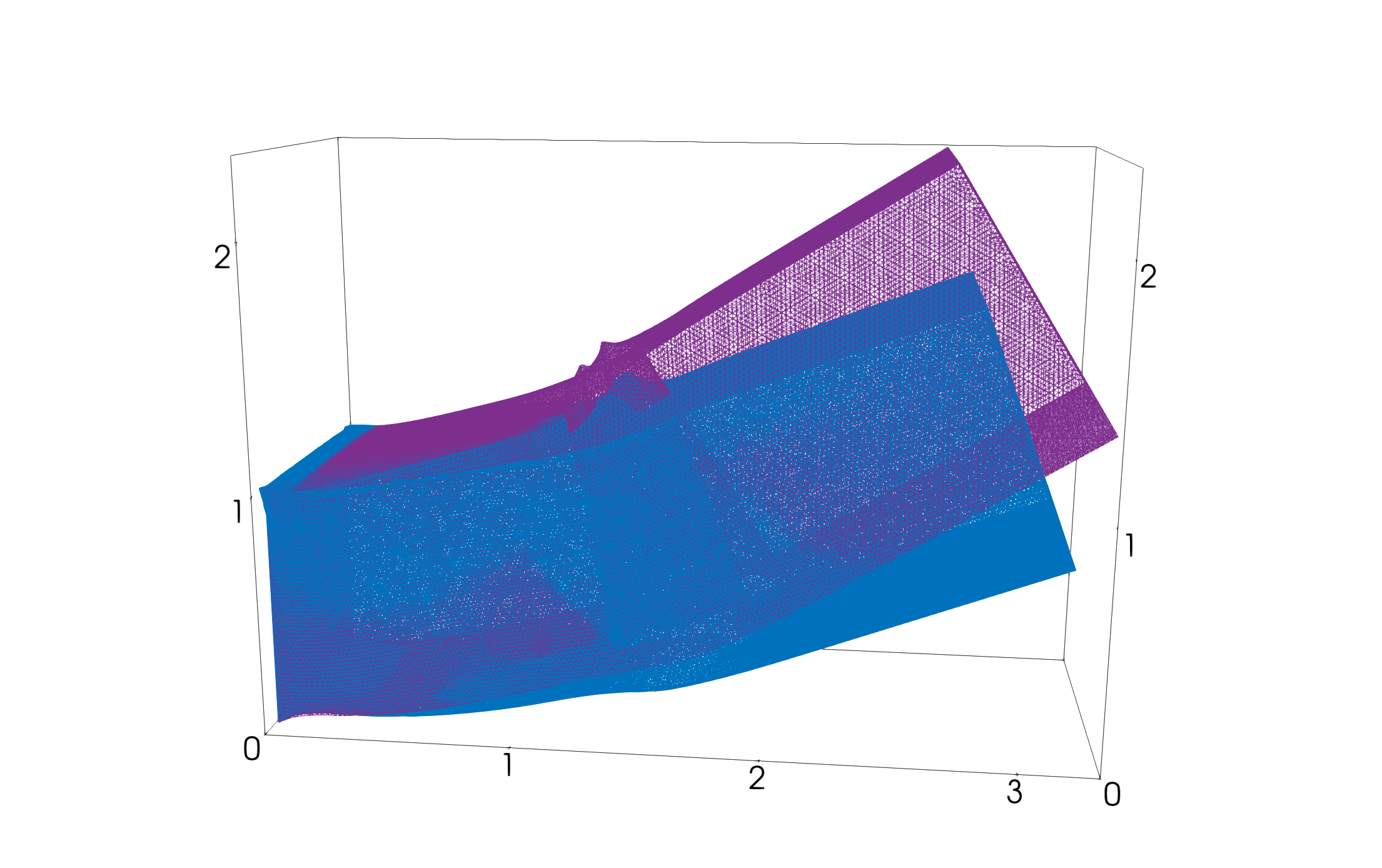}
        \caption{Comparison of deformation after unloading for $\alpha \in \{\textcolor{color1}{0.5},\textcolor{color4}{0.99}\}$. Amplified by a factor of $20$. }
        \label{fig:3dparaview_def}
    \end{subfigure}
    \caption{Qualitative pictures of stress and deformation for $\Omega \in \mathbb{R}^3$.}
    \label{fig:defo3d}
\end{figure}
Finally, we give qualitative pictures of the remaining stress and deformation after unloading in Figure~\ref{fig:defo3d}. In the deformation plot, where deformation is amplified by a factor of $20$, we clearly see the aforementioned results. The stress peaks arise at the boundaries of the cylindrical hole and at the spatially fixed face.
\FloatBarrier
\bibliographystyle{siamplain}
\bibliography{literature}
\end{document}